\documentclass[11pt,reqno]{amsart}
\usepackage{amsfonts}
\usepackage{pdfsync}
\usepackage{amsmath}
\usepackage{amssymb}
\usepackage{amsthm}
\usepackage{scalerel}
\usepackage{amsbsy}
\usepackage{enumerate}
\usepackage[T1]{fontenc}
\usepackage[utf8]{inputenc}
\usepackage{lmodern}
\usepackage{comment}
\usepackage{bigints}
\usepackage{wrapfig}
\usepackage{scalerel}

\newtheorem{theorem}{Theorem}[section]
\newtheorem{lemma}[theorem]{Lemma}
\newtheorem{corollary}[theorem]{Corollary}
\newtheorem{proposition}[theorem]{Proposition}
\newtheorem{claim}[theorem]{Claim}
\newtheorem{remark}[theorem]{Remark}

\newtheorem{definition}[theorem]{Definition}
\newtheorem{example}[theorem]{Example}
\newtheorem{proposition:mi bruna}[theorem]{Proposition}
\newtheorem*{mthm}{Main Theorem}
\newtheorem*{lemaBNW}{Lemma 2.2 (Bruna - Nagel - Wainger)}

\newcommand{\re}{{\mathop{\rm Re }}}
\newcommand{\im}{{\mathop{\rm Im }}}
\newcommand{\sgn}{{\mathop{\rm sgn}}}

\def\scaleint#1{\vcenter{\hbox{\scaleto[3ex]{\displaystyle\int}{#1}}}}

\def\bs{\!\!\!\!\!}

\def\d{\partial}

\def\vec{\boldsymbol}
\def\cj{\overline}

\begin{document}

\title{An Application of John Ellipsoids to the Szeg\H{o} kernel on unbounded convex domains}

\author[Benguria]{Soledad Benguria}

\address{ Mathematics Department, University of Wisconsin - Madison, 480
Lincoln Dr, Madison, WI, USA}
\email{ benguria@math.wisc.edu}


\begin{abstract}
We use convex geometry tools, in particular John ellipsoids, to obtain a size estimate for the Szeg\H{o} kernel on the boundary of a class of unbounded convex domains in $\mathbb{C}^n.$ Given a polynomial $b:\mathbb{R}^n \rightarrow \mathbb{R}$ satisfying a certain growth condition, we consider domains of the type $\Omega_b = \{ z\in\mathbb{C}^{n+1}\,:\, \im[z_{n+1}] > b(\re[z_1],\ldots,\re[z_n]) \}.$

\end{abstract}

\maketitle

\section{Introduction}

The study of the behavior of Szeg\H{o} kernels near the boundary of domains has been of great interest in the field of several complex variables during the past few decades. In this work we obtain a size estimate for the Szeg\H{o} kernel on the boundary of a class of unbounded domains in $\mathbb{C}^n.$ Our approach to this problem is the use of classical convex analysis techniques, and in particular an application of John ellipsoids.

\subsection{Background}

Given a strictly convex polynomial $b:\mathbb{R}^n\rightarrow\mathbb{R},$ consider the domain

$$\Omega_b = \{ (z_1,\ldots,z_{n+1})\in\mathbb{C}^{n+1}\,:\, \im[z_{n+1}]>b(\re[z_1],\ldots, \re[z_n])\}.$$

\noindent For unbounded domains of this type, it is convenient to define the Szeg\H{o} projection as in \cite{HaNaWa10}. We can identify the boundary $\d \Omega_b$ with $\mathbb{C}^{n}\times\mathbb{R}$ so that a point $ (\vec{z},t) \in  \mathbb{C}^{n}\times\mathbb{R}$ corresponds to $(\vec{z},t+ib(\re[z_1],\ldots, \re[z_n])) \in \d\Omega_b.$

\bigskip

Let $\mathcal{O}(\Omega_b)$ be the set of holomorphic functions in $\Omega_b.$ Given $F\in \mathcal{O}(\Omega_b)$ and $\epsilon>0,$ set 

$$ F_{\epsilon}(\vec{z},t) = F(\vec{z}, t + ib(\re[z_1],\ldots, \re[z_n])+i\epsilon).$$

\noindent The Hardy space $\mathcal{H}^2(\Omega_b)$ is defined as

\begin{equation}\label{eq:defh2}
 \mathcal{H}^2(\Omega_b) = \left\{ F \in\mathcal{O}(\Omega_b)\,:\, \sup_{\epsilon>0}\int_{\mathbb{C}^n\times\mathbb{R}}|F_{\epsilon}(\vec{z},t)|^2\,d\vec{z}\,dt  \equiv ||F||_{\mathcal{H}^2}^2 <\infty\right\}.
 \end{equation}

\bigskip

Let $\rho(z_1,\ldots,z_{n+1})=b(\re[z_1],\ldots, \re[z_n])- \im[z_{n+1}]$ be a defining function for the domain, i.e.,  $\Omega_b = \{ \vec{z} \in \mathbb{C}^{n+1} \,:\, \rho(\vec{z})<0\}$ where $\rho\in C^\infty(\mathbb{C}^{n+1})$ is such that $\nabla \rho \not= 0$ when $\rho=0.$ A Cauchy-Riemann operator is an operator of the form $ L = \sum_{j=1}^{n+1} a_j \frac{\d}{\d \overline{z_j}}.$ We say that $L$ is tangential if in addition $ L(\rho) = 0.$

\bigskip

 For a class of convex polynomials $b:\mathbb{R}^n\rightarrow \mathbb{R}$ satisfying a certain growth condition, we will define the Szeg\H{o} projection $\Pi:L^2(\d\Omega_b)\rightarrow \mathcal{H}^2(\Omega_b)$ to be the orthogonal projection from $L^2(\d\Omega_b)$ to the closed subspace of functions $f\in L^2(\d\Omega_b)$ that are annihilated in the sense of distributions by all tangential Cauchy-Riemann operators on $\d\Omega_b.$ For a thorough discussion on why such a map is well-defined, refer to Appendix B on  \cite{Pe15}. It can be shown that the Szeg\H{o} projection is given by integration against a kernel. That is, 

$$ \Pi [f](z) = \int_{\partial\Omega_b} S(z,w) f(w) d\sigma(w),$$ 

\noindent where $d\sigma$ is an appropriate measure on $\d\Omega_b$  (defined below, just before the statement of the Main Theorem). Here, $S(z,w),$ is called the Szeg\H{o} kernel and is the object we study.

\bigskip

The growth condition that we will impose on the polynomials $b$ on this paper is the following.

\begin{definition}\label{def:``combined degree''}
Let $m_1,\ldots,m_n$ be positive integers. We will say that a polynomial $p:\mathbb{R}^n \rightarrow \mathbb{R}$ is of {\bf combined degree} $(m_1,\ldots,m_n)$ if it is of the form 

$$p(\vec{x}) = \sum_{\alpha} c_{\alpha}\vec{x}^{\alpha},$$

\noindent where the exponents of its pure terms of highest order are $2m_1,\ldots,2m_n$ respectively and each index $\alpha = (\alpha_1,\ldots,\alpha_n)$ satisfies 

\begin{enumerate}
\medskip

\item {$ \dfrac{\alpha_1}{2m_1} + \ldots +\dfrac{\alpha_n}{2m_n} \le 1;$}
\medskip

\item {$ \dfrac{\alpha_1}{2m_1} + \ldots +\dfrac{\alpha_n}{2m_n} = 1$ if and only if there exists some $j$ such that $\alpha_j = 2m_j.$}

\end{enumerate}

\end{definition}

\bigskip

\begin{example}
\medskip
\noindent The polynomial $p(x_1,x_2) = x_1^2 +x_1x_2+ x_1^2x_2^2 +x_1^4 + x_2^6$ is of combined degree $(2,3).$ However, the polynomial  $\tilde{p}(x_1,x_2) = x_1^2 +x_1x_2+ x_1^2x_2^3 +x_1^4 + x_2^6$ is not a polynomial of combined degree. 
\end{example}

\bigskip

Throughout the rest of this work we will assume that 

\begin{equation*}\label{eq:pr}
\Omega_b = \{ (z_1,\ldots,z_{n+1})\in\mathbb{C}^{n+1}\,:\, \im[z_{n+1}]>b(\re[z_1],\ldots, \re[z_n])\},
\end{equation*}

\noindent where $b:\mathbb{R}^n\rightarrow\mathbb{R}$ is a strictly convex polynomial of combined degree $(m_1,\ldots,m_n).$ We will also identify $\partial \Omega_b$ with $\mathbb{R}^n\times \mathbb{R}^n\times \mathbb{R}.$ That is, given $(\vec{z},z_{n+1})\in \mathbb{C}^{n+1},$  we write $\vec{z} = \vec{x}+i\vec{y},$ and denote a point on $\partial \Omega_b$ by using the notation $(\vec{x},\vec{y},t),$ where $t =  \re[z_{n+1}].$ As is \cite{HaNaWa10}, and to avoid degeneracy issues due to the unboundedness of the domain, we take Lebesgue measure $d\sigma= d\vec{x} d\vec{y} dt$ as the measure on the boundary. 

\bigskip

We obtain the following size estimate for the Szeg\H{o} kernel on the boundary of $\Omega_b:$ 

\bigskip

\begin{mthm}

Let $(\vec{x},\vec{y},t)$ and $(\vec{x'},\vec{y'},t')$ be any two points in $\partial\Omega_b.$ Define $\tilde{b}(\vec{v}) = b\left( \vec{v} +\frac{\vec{x}+\vec{x'}}{2}\right) - \nabla b\left(\frac{\vec{x}+\vec{x'}}{2}\right)\cdot\vec{v} - b\left(\frac{\vec{x}+\vec{x'}}{2}\right);$ $\delta(\vec{x},\vec{x'}) = b(\vec{x}) +b(\vec{x'}) -2b\left(\frac{\vec{x}+\vec{x'}}{2} \right);$ and $w = (t'-t)  + \nabla b \left(  \frac{\vec{x}+\vec{x'}}{2}\right) \cdot (\vec{y'}-\vec{y}).$ Then

\begin{equation}\label{eq:bound principal}
\left\vert S\left((\vec{x},\vec{y},t);(\vec{x'},\vec{y'},t')\right)\right\vert \le \frac{C}{\sqrt{\delta^2+ \widetilde{b}(\vec{y}-\vec{y'})^2+ w^2 }\left| \left\{ \vec{v} \,:\, \tilde{b}(\vec{v}) < \sqrt{\delta^2+ \widetilde{b}(\vec{y}-\vec{y'})^2+ w^2 }\right\} \right|^2}.
\end{equation}

\noindent Here the constant $C$ depends on the exponents $\{m_1,\ldots,m_n\}$ and the dimension of the space, but is otherwise independent of the choice of $b$ and of the two given points.

\end{mthm}

\begin{remark}
The condition that $b$ is of combined degree is sufficient to ensure the existence of estimates as above. However, similar results could probably be obtained under weaker assumptions (e.g., a finite-type assumption).  The methods we use in this work, though, rely heavily on the combined degree hypothesis. New methods would have to be developed to generalize this result. In particular, the proof of Claims \ref{cl:bound g sin coefs en 2 dim} and \ref{cl:forma cuadratica} would need to be adapted. The result of these claims (especially of the first one) are used repeatedly throughout the paper. The fact that $b$ is a polynomial function also plays an important role throughout this work (see, e.g., Claims \ref{cl:porte de v0 gen},
 \ref{cl:bound de las derivadas}, and \ref{cl:simetria}). 

\end{remark}

\bigskip

For {\it bounded} domains $\Omega$ the Szeg\H{o} kernel has been extensively studied. Among others, the works by Gindikin \cite{Gi64}, Fefferman \cite{Fe74}, Boutet de Monvel and Sj{\"o}strand \cite{BoSj75}, Phong and Stein \cite{PhSt77}, Boas \cite{Bo85, Bo87}, Christ \cite{Ch88}, Fefferman, Kohn, Machedon  \cite{FeKoMa90}, Machedon \cite{Ma88}, McNeal and Stein \cite{McSt97}, Lanzani and Stein \cite{LaSt04} \cite{LaSt13}, and Krantz \cite{Kr80} \cite{Kr14}, have given great insight on the behavior of the Szeg\H{o} kernel and the Szeg\H{o} projection. Much less is known about the Szeg\H{o} kernel on {\it unbounded} domains $\Omega.$ Rather than broad general results, only particular domains have been studied in the latter case. 
\bigskip

When studying the Szeg\H{o} kernels for unbounded domains, domains of the type $\Omega = \{ (\vec{z},z_{n+1})\in\mathbb{C}^{n}\times\mathbb{C}\,:\, \im[z_{n+1}]>\phi(\vec{z})\}$ for different choices of functions $\phi$ have been of particular interest. In fact, when $n=1$ the Szeg\H{o} projection and kernel on such domains have been extensively studied. For example, Greiner and Stein \cite{GrSt78} obtain a closed formula for the Szeg\H{o} kernel for $\phi(z) = |z|^{2k}, k\in\mathbb{N}.$  The singularities of the Szeg\H{o} kernel have been studied by Haslinger \cite{Ha95} when $\phi(z)= |\re[z]|^{\alpha},$ $\alpha>\frac{4}{3};$  by Carracino \cite{Ca07} for a particular choice of a non-convex $\phi$; by Gilliam and Halfpap \cite{GiHa11}, \cite{GiHa14} when $\phi$ is a non-convex even degree polynomial with positive leading coefficient; and by Halfpap, Nagel and Wainger \cite{HaNaWa10}, who consider, among others, functions such that $\phi(z) = \exp(-|z|^{-a}),$ $a>0,$ for $|z|$ small and $\phi(z) = z^{2m}$ for $|z|$ large. Nagel \cite{Na86} studies the Szeg\H{o} kernel on the boundary of domains of the kind $\Omega = \{ \vec{z} \in \mathbb{C}^{2} : {\rm Im}[z_{2}] > \phi({\rm Re}[z_1])\},$ where $\phi$ is a subharmonic, non-harmonic polynomial with the property that $\Delta \phi(z) = \Delta\phi(x+iy) $ is independent of $y.$ He shows that in this case the Szeg\H{o} kernel is bounded by $|B|^{-1},$ where $|B|$ is the volume of a certain non-isotropic ball. Nagel's result was later generalized by Nagel, Rosay, Stein and Wainger \cite{NaRoStWa88} \cite{NaRoStWa89} to domains where $\phi$ is a subharmonic, non-harmonic polynomial in $\mathbb{C}.$ Furthermore, they obtain similar estimates for the derivatives of the Szeg\H{o} kernel, and use these to obtain $L^p$ bounds for the Szeg\H{o} projection. See also \cite{Ma88} for a related problem. More recently, Peterson \cite{Pe15} considered domains where $\phi$ is a smooth, subharmonic, nonharmonic function and $\d\Omega$ satisfies a uniform finite-type hypothesis of order $m.$

\bigskip

Less progress has been made in the case $n>1.$ See, however, \cite{FrHa95}, and the references therein.  In \cite{FrHa95}, Franciscs and Hanges take $\phi(\vec{z},\vec{\xi}) = ||\vec{z}||^2 +||\xi||^{2p}$ for $\vec{z}\in\mathbb{C}^n,$ $\xi\in\mathbb{C}^m,$ $n\ge0,$ $m\ge 1$ and $p$ a positive integer, and obtain a closed formula for the Szeg\H{o} kernel. 

\bigskip 
It was the goal of obtaining similar results to those of Nagel \cite{Na86}, but for $n>1,$ that led me to work on the problem at hand. One of the main difficulties of the problem in several dimensions stems from the fact that the polynomial $b$ we consider can exhibit different growth rates along different directions. This is where the John ellipsoids come into play, allowing one to introduce a rescaling that takes care of this issue. 

\bigskip 

In \cite{RaTi15}, Raich and Tinker study a similar problem. They consider domains $\Omega = \{ (z,\vec{w})\in\mathbb{C}\times\mathbb{C}^n\,:\, \im[\vec{w}]=P(\re[z])\},$ where  $P=(a_1p,\ldots,a_np)$ with $p :\mathbb{R} \rightarrow \mathbb{R}$ a convex polynomial, $a_n=1,$  and $\vec{a} = (a_1,\ldots,a_n) \in \mathbb{R}^n$. They obtain a bound for the Szeg\H{o} kernel and its derivatives in terms of the volume of a ball defined by a certain pseudometric, as well as an explicit formula for the Szeg\H{o} kernel when $p(x)=x^2.$   

\bigskip

\subsection{Methods}

Except for the setup of the problem, which is outlined in Section \ref{prelim}, all the methods we use are classical convex analysis techniques. In fact, an application of John ellipsoids is the key ingredient in the proof of the Main Theorem. 

\bigskip

In Section \ref{prelim} we derive an integral formula for the Szeg\H{o} kernel. Our estimates all follow from a study of this integral expression, given by

\begin{equation}\label{eq:integral 1}
S((\vec{x},\vec{y},t);(\vec{x'},\vec{y'},t')) = 
{\scaleint{8ex}}_{\bs 0}^{\infty} e^{-2\pi\tau [b(\vec{x'}) + b(\vec{x}) + i(t'-t)]} 
\left({\scaleint{8ex}}_{\bs \mathbb R^n}  \frac{e^{2\pi \vec{\eta} \cdot [\vec{x}+\vec{x'} - i(\vec{y'}-\vec{y})] }   }{\displaystyle\int \limits_{\mathbb R^n} e^{4\pi [\vec\eta \cdot \vec{v}  - b(\vec{v})\tau]} \, d\vec{v}} \, d\vec{\eta} \,\right) d\tau.
\end{equation}

In Sections \ref{coefs}  and \ref{denominador} we build the tools that we use in the proof of the Main Theorem, which is presented in Section \ref{main theorem}. We devote Section \ref{coefs} to a study of the coefficients of convex polynomials in several variables. In the one-variable case it was shown in \cite{BrNaWa88} that the absolute value of the coefficients of a convex polynomial with no constant or linear terms can be bounded, up to a constant that depends only on the degree of the polynomial, by the value of the polynomial at $1.$ It is not possible to obtain such a bound in more variables, since the polynomial might be growing in some directions but not along others. However, we show that the absolute value of the coefficients can be bounded by the average of the polynomial over a circle of arbitrary positive radius, up to a constant that only depends on the degree of the polynomial and the chosen radius. In Section \ref{denominador} we use this result to prove a technical lemma that will be key in dealing with the denominator integral of equation (\ref{eq:integral 1}). 

\bigskip

In Section \ref{main theorem} we present the proof of the Main Theorem. The proof is, at its core, an application of John ellipsoids. We introduce a change of variables in the integral expression for the Szeg\H{o} kernel comprised of factors defined by the length of the axes of the unique maximal inscribed ellipsoid associated to a symmetrization of the convex body $R=\left\{ \vec{v} \,:\, \tau\tilde{b}(\vec{v}) \le 1 \right\}.$  The construction of these factors is presented in Section \ref{mus}, and it is from the product of these factors (which appear in the denominator as the Jacobian of the change of variables), that the volume expression in the estimate given in equation (\ref{eq:bound principal}) is obtained. 

\bigskip 

So as to make the computations simpler, we have split the proof of the Main Theorem into the proof of three separate bounds. The first bound, in terms of $\delta,$ is given in Section \ref{delta}; the second bound, in terms of $\widetilde{b}(\vec{y}-\vec{y'}),$ is presented in Section \ref{segundo}; and the bound in terms of $w$ is given in Section \ref{tercer}. These bounds are then combined to yield the estimate of the Main Theorem in Section \ref{conclusion}.

\bigskip 

Geometric tools, such as the ones we employ in this paper, have often been used in the study of the Szeg\H{o} kernel. For example, similar geometric ideas are used by McNeal and Stein \cite{McSt97} to obtain a bound for the Szeg\H{o} kernel $S(z,w)$ for smoothly bounded convex domains $\Omega$ of finite type in $\mathbb{C}^n$ in terms of the smallest {\it tent} in $\partial\Omega$ containing $z$ and $w$ (see also \cite{Mc94}). That is, they show that for smoothly bounded convex domains of finite type in $\mathbb{C}^n,$ there exists a constant $C$ so that for all $z,w\in\overline{\Omega}\times\overline{\Omega} \setminus \{\text{diagonal in } \d\Omega\},$

$$ |S(z,w)| \le \frac{C}{|T(z,\gamma)|}.$$

\noindent Here $T(z,\gamma) = P_{\gamma}(\pi(z))\cap\overline{\Omega};$ the projection $\pi:U\rightarrow\d\Omega$ is a smooth map such that if $b\in\d\Omega,$ $\pi(b)=b$ and $\pi^{-1}(b)$ is a smooth curve, transversally intersecting $\d\Omega$  at $b;$ and $\gamma =|r(z)|+|r(w)| + \inf\{\epsilon>0\,:\, w\in T(z,\epsilon)\}.$ The geometric constructions used in \cite{Mc94} and \cite{McSt97} are based on the length of the sides of a certain polydisc, as opposed to the lengths of the axes of an ellipsoid. The use of John ellipsoids, however, seems more natural for the domains we consider. In fact, one of the key components of our proof is the use of universal bounds for the coefficient of convex polynomials in terms of the average of the polynomials over circles of arbitrary positive radius (as described above, and in more detail in Section \ref{coefs}). Thus, it makes sense to consider ellipsoids (which can be rescaled into spheres), rather than polydiscs. The particular tools we employ (i.e, the approximation by John ellipsoids) have not been used before in this context and provide a new approach to the problem.

\bigskip

\section{Preliminaries}\label{prelim}

 In this section we derive an integral formula for the Szeg\H{o} kernel for the domains under consideration. We follow the analogous derivation for the one-dimensional case found in \cite{Na86}.
 
 \bigskip
\begin{proposition}\label{prop:integral szego1}
The Szeg\H{o} kernel on the boundary of domains of the type $\Omega_b = \{ \vec{z} \in \mathbb{C}^{n+1} : {\rm Im}[z_{n+1}] >  b({\rm Re}[z_1],\ldots,{\rm Re}[z_n])\}$ where $b:\mathbb{R}^n\rightarrow\mathbb{R}$ is a convex polynomial of combined degree is given by

\begin{equation}\label{eq:integral original2}
S((\vec{x},\vec{y},t);(\vec{x'},\vec{y'},t')) = 
{\scaleint{8ex}}_{\bs 0}^{\infty} e^{-2\pi\tau [b(\vec{x'}) + b(\vec{x}) + i(t'-t)]} 
\left({\scaleint{8ex}}_{\bs \mathbb R^n}  \frac{e^{2\pi \vec{\eta} \cdot [\vec{x}+\vec{x'} - i(\vec{y'}-\vec{y})] }   }{\displaystyle\int \limits_{\mathbb R^n} e^{4\pi [\vec\eta \cdot \vec{v}  - b(\vec{v})\tau]} \, d\vec{v}} \, d\vec{\eta} \,\right) d\tau, 
\end{equation}

\noindent where $(\vec{x},\vec{y},t)$ and $(\vec{x'},\vec{y'},t')$ are any two points on $\d\Omega_b.$

\end{proposition}

\begin{proof}

Let 

\begin{equation}\label{eq:def function en m}\begin{split}
 \rho (z_1,\ldots ,z_{n+1}) = & \,b({\rm Re}[z_1],\ldots, {\rm Re}[z_n]) - {\rm Im}[z_{n+1}]\\
 = &\,b\left(\frac{z_1 + \cj{z}_1}{2}, \ldots, \frac{z_n + \cj{z}_n}{2}\right)\,  - \, \frac{z_{n+1} - \cj{z}_{n+1}}{2i}\\
\end{split}\end{equation}

\bigskip

\noindent be a defining function for our domain. The Szeg\H{o} projection is the orthogonal projection $\Pi : L^2(\d \Omega_b) \rightarrow H^2(\d \Omega_b).$ It can be shown (see, e.g., \cite{HaNaWa10}, \cite{Be14}), that $H^2(\d\Omega_b)$ as defined in equation (\ref{eq:defh2}) is equivalent to the space 

$$\{ f\in L^2(\d\Omega_b) : \overline{Z}(f) = 0 \,\,\text {as a distribution, for all }\overline{Z} \in T^{0,1}(\partial \Omega_b)\}. $$

 \noindent We begin by finding a base for the tangential Cauchy-Riemann operators. We can let 

$$ \cj{Z}_j = 2 \left(\frac{\d}{\d \cj{z}_j} + A_j(z_1,\ldots,z_n)\frac{\d}{\d \cj{z}_{n+1}} \right)\qquad j=1, \ldots, n.$$ 

\noindent For these operators to be tangential they must satisfy $\cj{Z}_j(\rho) = 0.$ Thus, 

$$ \cj{Z}_j = 2\left(\frac{\d}{\d \cj{z}_j} - i \frac{\d b}{\d x_j}(\vec{x})\frac{\d}{\d \cj{z}_{n+1}} \right) \qquad j=1,\ldots,n $$  are a basis for the space of tangential Cauchy-Riemann operators for our domain in $\mathbb{C}^{n+1}$.  We can identify $\d \Omega_b$ with $\mathbb{C}^n \times \mathbb{R}$ via the diffeomorphism

$$ (z_1,\ldots,z_n,t) \in \mathbb{C}^n \times \mathbb{R} \leftrightarrow (z_1,\ldots,z_n,t+ib({\rm Re}[z_1],\ldots, {\rm Re}[z_n])) \in \d \Omega_b .$$

\bigskip

\noindent Our operators $\cj{Z}_j$ are operators in $\mathbb{C}^{n+1} $. The pushforward of these operators to $\mathbb{C}^n \times \mathbb{R}$ is

\begin{equation}\label{eq:CR1 en m}
\cj{Z}_j = \frac{\d}{\d x_j} + i \left( \frac{\d}{\d y_j} - \frac{\d b}{\d x_j}(\vec{x}) \frac{\d}{\d t}\right).
\end{equation}

\bigskip

\begin{lemma} Given $\vec{x}\in \mathbb{R}^n,$ $\vec{\eta}\in \mathbb{R}^n,$ and $\tau \in \mathbb{R},$ let 

$$ \mathcal{M}[g](\vec{x},\vec{\eta},\tau) = e^{-2\pi[\vec{\eta}\cdot \vec{x} - b(\vec{x}) \tau]}g(\vec{x},\vec{\eta},\tau),$$
 and define the partial Fourier transform 
 $$\mathcal{F}[f](\vec{x},\vec{y},\tau) = \hat{f}(\vec{x},\vec{\eta},\tau) = \int\limits_{\mathbb{R}^{n+1}}  e^{-2\pi i (\vec{y}\cdot \vec{\eta} +t\tau)}f(\vec{x},\vec{y},t)\,d\vec{y} \, dt.$$
 
\noindent Then $$ \mathcal{M} :L^2(\mathbb{R}^{2n+1}, d\vec{x}d\vec{\eta}\,d\tau\,) \rightarrow L^2(\mathbb{R}^{2n+1}, e^{4\pi[\vec{\eta}\cdot\vec{x}  -b(\vec{x})\tau]} d\vec{x}\,d\vec{\eta}\,d\tau\,   ) $$ is an isometry for $f\in L^2(\mathbb{R}^{2n+1}, d\vec{x}d\vec{\eta}\,d\tau\,),$ and
 
 $$ \cj{Z}_j[f] = \mathcal{F}^{-1} \mathcal{M}^{-1} \frac{\partial}{\partial x_j} \mathcal{M} \mathcal{F}[f] \qquad j=1,\ldots n.$$
 
\end{lemma}

\begin{proof}

It is easy to check that $\mathcal M$ is an isometry in this weighted $L^2$ space. Also,

\begin{equation*}
\begin{split} \cj{Z}_j[f] = &\, \cj{Z}_j \mathcal{F}^{-1}(\hat{f}) \\ = &\, \displaystyle\int\limits_{\mathbb{R}^{n+1}} e^{2\pi i (\vec{y}\cdot\vec{\eta} + t\tau)}\left(
\frac{\d \hat{f}(\vec{x},\vec{\eta},\tau)}{\d x_j} - 2\pi \eta_j \hat{f}(\vec{x},\vec{\eta},\tau)   +\frac{\d b}{\d x_j} (\vec{x})2\pi\tau \hat{f}(\vec{x},\vec{\eta},\tau) \right) \, d\vec{\eta} \, d\tau  \\  = & \, \displaystyle\int\limits_{\mathbb{R}^{n+1}} e^{2\pi i (\vec{y}\cdot\vec{\eta} +t\tau)}e^{2\pi [\vec{\eta}\vec{x}  -b(\vec{x})\tau]}\frac{\partial}{\partial x_j}\left(e^{-2\pi [\vec{\eta}\cdot \vec{x}  -b(\vec{x})\tau]} \hat{f}(\vec{x},\vec{\eta},\tau)\right) \,d\vec{\eta}\, d\tau \\ = &\, \mathcal{F}^{-1} \mathcal{M}^{-1} \frac{\partial}{\partial x_j} \mathcal{M} \mathcal{F}[f]. \\
\end{split}\end{equation*}

\end{proof}

 Since $\mathcal{F}$ and $\mathcal{M} $ are isometries, instead of projecting onto the null space of the tangential Cauchy-Riemann operators we can project onto the closed subspace of functions in $ L^2(\mathbb{R}^{2n+1}, e^{4\pi[\vec{\eta}\cdot\vec{x} -b(\vec{x})\tau]} \allowbreak d\vec{x}\,d\vec{\eta}\,d\tau\,   )$ which are a.e. constant in $x.$

 \bigskip
 
 More precisely, as in \cite{Na86}, let 
 
 $$\Sigma = \left\{ (\vec{\eta},\tau)\in \mathbb{R}^{n+1}\, \left| \,\,\,\displaystyle\int\limits_{\mathbb{R}^n} e^{4\pi[\vec{\eta}\cdot \vec{x}  - b(\vec{x})\tau]}\, d\vec{x} <\infty \right. \right\}.$$
 
 \noindent Then, because of the growth hypothesis on $b,$ 
 
 $$\Sigma = \{  (\vec{\eta},\tau)\in \mathbb{R}^{n+1} \,| \,\tau > 0\}.$$ 
 
 \noindent This follows from the fact that $b$ is positive and grows at least quadratically in all directions. 
 
 \bigskip 
 
 Let $\widehat\Pi_{\vec{\eta},\tau} $ be the projection of $L^2(\mathbb{R}^n,e^{4\pi[\vec{\eta}\cdot\vec{x} -b(\vec{x})\tau]}\, d\vec{x})$ onto the constants if $ (\vec{\eta},\tau) \in \Sigma,$ and let $\widehat\Pi_{\vec{\eta},\tau} =0$ otherwise. Then, if $(\vec{\eta},\tau)\in \Sigma,$ 

\begin{equation*}\begin{split}  \widehat\Pi_{\vec{\eta},\tau} g = \,&\, \frac{<g,1>1}{<1,1>}   =  \scaleint{7ex}_{\bs \mathbb{R}^n} g(\vec{x'})\left(\,\,\,\frac{ e^{4\pi [\vec{\eta}\cdot\vec{x'} - b(\vec{x'})\tau]} \,}{\displaystyle\int_{\mathbb{R}^n} e^{4\pi[\vec{\eta}\cdot \vec{v} - b(\vec{v})\tau]}\, d\vec{v}} \right) \, d\vec{x'}.\\ 
\end{split}\end{equation*}

 We define the projection 
 
 $$\widehat{\Pi}:  L^2(\mathbb{R}^{2n+1}, e^{4\pi[\vec{\eta}\cdot\vec{x} -b(\vec{x})\tau]} d\vec{x}\,d\vec{\eta}\,d\tau\,   ) \rightarrow  L^2(\mathbb{R}^{2n+1}, e^{4\pi[\vec{\eta}\cdot\vec{x} -b(\vec{x})\tau]} d\vec{x}\,d\vec{\eta}\,d\tau\,   )$$
 
 \noindent by

  $$\widehat{\Pi}\,g(\vec{x},\vec{\eta},\tau) = \widehat\Pi_{\vec{\eta},\tau}( g_{\vec{\eta},\tau})(\vec{x}),$$  
 
 \noindent where $g_{\vec{\eta},\tau}(\vec{x}) = g(\vec{x},\vec{\eta},\tau).$ Then $\Pi [f] = \mathcal{F}^{-1} \mathcal{M}^{-1} \widehat{\Pi} \mathcal{M} \mathcal{F}[f]$ is the projection from $L^2(\mathbb{R}^{2n+1}, d\vec{x}\,d\vec{y}\,dt\,)$ onto the null space of the tangential Cauchy-Riemann operators. Thus, if $f\in L^2(\mathbb{R}^{2n+1}, d\vec{x}\,d\vec{y}\,dt\,),$ the Szeg\H{o} projection is given by 

\allowdisplaybreaks\begin{align*}
 & \Pi [f](\vec{x},\vec{y},t) =  \mathcal{F}^{-1} \mathcal{M}^{-1} \widehat{\Pi} \mathcal{M} \mathcal{F}[f](\vec{x},\vec{y},t)  = \scaleint{8ex}_{\bs \mathbb{R}^{2n+1}} f(\vec{x'},\vec{y'},t')  S(   (\vec{x},\vec{y},t);  (\vec{x'},\vec{y'},t'))    d\vec{x'}\,d\vec{y'} \, dt',
 \\
\end{align*}

\noindent where the Szeg\H{o} kernel is given by

\begin{equation*}
S(   (\vec{x},\vec{y},t);  (\vec{x'},\vec{y'},t')) =   {\scaleint{7ex}_{\bs 0}^{\infty}} e^{-2\pi\tau [b(\vec{x'}) + b(\vec{x}) + i(t'-t)]} \left( {\scaleint{7ex}_{\bs \mathbb R^n}} \frac{e^{2\pi\vec{\eta}\cdot [\vec{x}+\vec{x'}- i(y'-y)] }   }{\displaystyle\int_{\mathbb R^n} e^{4\pi [\vec{\eta}\cdot \vec{v}  - b(\vec{v})\tau]} \, d\vec{v}}\, d\vec{\eta} \right) \, d\tau.
\end{equation*}

\noindent This finishes the proof of Proposition \ref{prop:integral szego1}.

\end{proof}

\bigskip
\bigskip

\section{Coefficients of convex polynomials of several variables}\label{coefs}

In this section we obtain bounds for the absolute value of the coefficients of convex polynomials of several variables with no constant or linear terms. 
\bigskip

\begin{theorem}\label{prop:mi bruna}
Let $\Gamma(M) = \{ (\alpha_1,\ldots,\alpha_n) \in \mathbb{N}^n\,:\, 2\le |\alpha|\le M\}.$ Let $S(M)$ be the set of convex polynomials of the form $g(\vec{v})= \sum_{\alpha\in\Gamma(M)}c_{\alpha}\vec{v}^{\alpha}.$ Then for any fixed $a>0,$ there exists a positive constant $C(M,a)$ that depends only on $M$ and the constant $a$ such that if $g\in S(M),$

\begin{equation}\label{eq:coeficientes acotados}
\sum_{\alpha\in\Gamma(M)}|c_{\alpha}| \le C(M,a) \int_{|\sigma|=a} g(\sigma) \,d\sigma.
\end{equation}

\noindent 
\end{theorem}

\begin{remark}
This is a generalization of the result in one variable by Bruna, Nagel and Wainger in \cite{BrNaWa88} (Lemma 2.1). 
\end{remark}

\begin{proof}
Let $a>0$ be a fixed positive constant and let $|\Gamma(M)|$ denote the cardinality of the set of indices $\Gamma(M).$ We identify the space $S(M)$ with a cone in $\mathbb{R}^{|\Gamma(M)|}$ via the identification

$$ g(\vec{v}) = \sum_{\alpha\in\Gamma(M)}c_{\alpha}\vec{v}^{\alpha}\in S(M) \leftrightarrow (c_1,\ldots,c_{|\Gamma(M)|})\in\mathbb{R}^{|\Gamma(M)|},$$

\noindent where $c_j$ corresponds to the coefficient $c_{\alpha}$ for the $j^{th}$ element $\alpha$ in $\Gamma(M),$ under some fixed but arbitrary ordering of $\Gamma(M).$

\bigskip

 Let 

\begin{equation}\label{eq:definicion de sigma}
\Sigma_M = \{ g(\vec{v})=\sum_{\alpha\in\Gamma(M)}c_{\alpha}\vec{v}^{\alpha}\in S(M) \,:\, \sum_{\alpha\in\Gamma(M)}|c_{\alpha}| = 1 \}.
\end{equation}

\bigskip

\noindent We claim that $\Sigma_M$ is a compact subset of the cone $S(M).$ In fact, let $\{\vec{c_n}\}_{n\in\mathbb{N}}$ in $\mathbb{R}^{|\Gamma(M)|}$ be a sequence of tuples associated to a sequence of polynomials $\{q_n\}_{n\in\mathbb{N}}$ in $\Sigma_M.$ Since $\{\vec{c_n}\}_{n\in\mathbb{N}}$ is a sequence contained in the compact set $B_M =  \{  (c_1,\ldots,c_{|\Gamma(M)|})\in\mathbb{R}^{|\Gamma(m)|}\,:\,   \sum_{1\le j\le |\Gamma(M)|}|c_j| = 1  \},$ it has a convergent subsequence $\{\vec{c_{n_i}}\}_{n_i\in\mathbb{N}}$ . Let $\vec{c}$ be the limit of this subsequence, and let $q$ be the polynomial associated to this tuple. We claim that $q$ is an element of $\Sigma_M.$ In fact, the identification preserves the degree of the polynomial and the fact that there are no constant or linear terms. Also, since $\vec{c}$ is an element of $B_M,$ it satisfies that $\sum_{1\le j\le |\Gamma(M)|}|c_j| = 1.$ Thus, is suffices to show that $q$ is convex. This follows easily, since given any polynomial $q_{n_i}$ associated to an element of the convergent subsequence $\{\vec{c_{n_i}}\}_{n_i\in\mathbb{N}},$ we have that $ q_{n_i}(\alpha\vec{x}+(1-\alpha)\vec{y})\le \alpha q_{n_i}(\vec{x}) + (1-\alpha)q_{n_i}(\vec{y})$ for all $0\le\alpha\le 1$ and for all points $\vec{x},\vec{y}$ in $\mathbb{R}^n.$ Thus, and since $q_{n_i}(\alpha\vec{x}+(1-\alpha)\vec{y})\rightarrow q(\alpha\vec{x}+(1-\alpha)\vec{y});$ $\alpha q_{n_1}(\vec{x}) \rightarrow \alpha q(\vec{x});$ and $(1-\alpha)q_{n_i}(\vec{y})\rightarrow (1-\alpha)q(\vec{y}),$ the convexity of $q$ follows immediately.

\bigskip 

 Let 
 $$\Phi_I(g) = \frac{1}{\omega_n(a)}\int_{|\sigma|=a} g(\sigma) \,d\sigma,$$
\noindent where $\omega_n(a)$ is the surface area of the sphere of radius $a$ in $\mathbb{R}^n$ and 
$$\Phi_{II}(g) = \sum_{\alpha\in\Gamma(M)}|c_{\alpha}|.$$
\noindent Notice that these functions are continuous on $S(M),$ and that $\Phi_{II}(g) = 1$ on $\Sigma_M.$ 

\bigskip

 We claim that $\Phi_I(g)$ is strictly positive on $\Sigma_M.$ In fact, since $g$ is convex, $g(\vec{0}) = 0$ and $\nabla g(\vec{0}) = \vec{0}$ it follows that $g$ is nonnegative. Moreover, on $\Sigma_M$ at least one of the coefficients of $g$ must be different from zero, so $g$ can not be the zero polynomial. Thus $g$ must be positive almost everywhere. In particular, the average over the circle of radius $a$ must be strictly positive. 

\bigskip

 Therefore, and since $\Phi_I(g)$ is continuous as a function of $g,$ it attains a minimum in $\Sigma_M,$ and this minimum is strictly positive. Thus, and since $\Phi_{II}(g) = 1$ on $\Sigma_M,$ there exists a constant $C>0$ such that for any $g\in\Sigma_M,$

$$ \Phi_I(g) \ge C = C\Phi_{II}(g).$$ 

\noindent That is, 
$$ \frac{1}{\omega_n(a)}\int_{|\sigma|=a} g(\sigma) \,d\sigma \ge C \Phi_{II}(g) = C \sum_{\alpha\in\Gamma(M)}|c_{\alpha}|,$$

\noindent as desired. 

\bigskip

 Consider now a polynomial $g(\vec{v})=\sum_{\alpha\in\Gamma(M)}c_{\alpha}\vec{v}^{\alpha}\in S(M),$ but which is not necessarily in $\Sigma_M.$ Then let $ h(\vec{v}) = \sum_{\alpha\in\Gamma(M)} b_{\alpha}\vec{v}^{\alpha}$ where 

$$b_{\alpha} = \frac{c_{\alpha}}{\sum_{\beta\in\Gamma(M)}|c_{\beta}|}$$ 

\noindent so that $\sum_{\alpha\in\Gamma(M)} |b_{\alpha}| = 1$ and $h(\vec{v})\in\Sigma_M.$ It follows from the previous case that 

$$ \frac{1}{\omega_n(a)}\int_{|\sigma|=a} h(\sigma) \,d\sigma \ge C.$$ 

\noindent That is, 

$$ \frac{1}{\omega_n(a)}\int_{|\sigma|=a} \frac{g(\sigma)}{\sum_{\beta\in\Gamma}|c_{\beta}|} \,d\sigma \ge C.$$ 

\noindent This gives the desired inequality.

\end{proof}

\begin{corollary}\label{cor:coeficientes acotados}

Let $g(\vec{v})= \sum_{\alpha}c_{\alpha}\vec{v}^{\alpha}$ be a convex polynomial such that $g(\vec{0}) = 0$ and $\nabla g(\vec{0})=\vec{0}.$ Suppose there exist two positive constants $A$ and $B$ such that  

\begin{equation}\label{eq:metido en john}
\{\vec{v}\,:\, |\vec{v}| \le A  \}  \subseteq \{ \vec{v}\,:\, g(\vec{v}) \le 1 \} \subseteq \{\vec{v}\,:\, |\vec{v}| \le B \}.
\end{equation}

\noindent Then there exists a constant $C$ that depends only on $A$ and the degree of $g$ such that

 \begin{equation}\label{eq:coeficientes acotados 2}
\sum_{\alpha}|c_{\alpha}| \le C.
\end{equation}

\noindent Moreover, for any point $\vec{x} =(x_1,\ldots,x_n)$ on the sphere of radius $A,$ there exist constants $C_1>0,$ $C_2>0$ that depend only on $A,$ $B$ and the degree of $g$ such that

\begin{equation}\label{eq:cota de g por debajo}
g(\vec{x}) \ge C_1 \ge C_2 \sum_{\alpha}|c_{\alpha}|.
\end{equation}

\begin{remark}
The bound given by equation (\ref{eq:coeficientes acotados 2}) can be obtained using just the left containment, i.e, the existence of a constant $A>0$ such that $\{\vec{v}\,:\, |\vec{v}| \le A  \}  \subseteq \{ \vec{v}\,:\, g(\vec{v}) \le 1 \}.$ The second bound, however, requires the existence of both an inner and an outer ball. 

\end{remark}

\end{corollary}

\begin{proof}
The first result follows immediately from the previous claim. In fact, we showed that 

$$ \sum_{\alpha} |c_{\alpha}| \le C \int_{|\sigma|=A} g(\sigma)\,d\sigma.$$ 

\noindent But by (\ref{eq:metido en john}) we have that $g(\sigma) \le 1$ for all $\sigma$ such that $|\sigma|=A.$ The result follows. 

\bigskip

 Observe that the bound $g(\vec{x}) \geq  C_2 \sum_{\alpha}|c_{\alpha}|$ will be an immediate consequence of the above bound on the coefficients once we show that $g(\vec{x}) \geq C_1$.
\bigskip

 The proof of equation (\ref{eq:cota de g por debajo}) requires the use of Lemma 2.1 of \cite{BrNaWa88}. The lemma states that given a convex polynomial of one variable of degree $M$ of the form

$$ p(t) = \sum_{j=2}^M a_jt^j$$

\noindent there exists a constant $C_M>0$ that depends only on $M$ such that

\begin{equation}\label{eq:Bruna Nagel Wainger}
C_M\sum_{j=2}^M |a_j| t^j \le p(t) \le \sum_{j=2}^M |a_j|t^j \qquad \forall t\ge0.
\end{equation}

\noindent In particular, this result implies that for any $\lambda>1$ and $t\ge0,$ 

\begin{equation}\label{eq:p en lambda}
 p(\lambda t) \le \sum_{j=2}^M |a_j| \lambda^j t^j \le \lambda^M  \sum_{j=2}^M |a_j| t^j \le \frac{\lambda^M}{C_M} p(t).
\end{equation}

\bigskip

Given a point $\vec{x}$ on the sphere of radius $A$ centered at the origin, we will let

$$ p(t) = g(t\vec{x}).$$

\noindent Notice that this defines a convex polynomial of one variable for which the bounds in equation (\ref{eq:Bruna Nagel Wainger}) apply.  Taking $t=1$ and $\lambda = \frac{B}{A}$ (where $A$ and $B$ are the radius of the inner and outer ball respectively) in equation $(\ref{eq:p en lambda}),$ we have that 

$$ p(1) \ge \frac{C_M}{\lambda^M}\, p\left(\frac{B}{A}\right).$$

\noindent That is, 

\begin{equation}\label{eq:penultima cota}
g(\vec{x}) \ge C_M\frac{ A^M}{B^M} \,g\left(\frac{B\vec{x}}{A} \right).
\end{equation}

\noindent Since $|\vec{x}|=A,$ then $ \left\vert\frac{B}{A}\vec{x}\right\vert=B.$ Because $g\left(\frac{B\vec{x}}{A} \right) \ge 1 $ by assumption, it follows that $g(\vec{x}) \ge C_M\frac{ A^M}{B^M}.$ This completes the proof of Corollary \ref{cor:coeficientes acotados}.

\end{proof}

\begin{remark} Notice that the convexity of $g$ implies that 

\begin{equation*}
g(\vec{x}) \ge C_1 \ge C_2 \sum_{\alpha}|c_{\alpha}|.
\end{equation*}

\noindent for any point $\vec{x}$ such that $|\vec{x}| \ge A.$

\end{remark}

\section{Decay of $\theta(\vec{\eta})$}\label{denominador}

In this section we study the decay of a function $\theta(\vec{\eta}),$ which we will presently define. The decay properties obtained for $\theta(\vec{\eta}),$ stated in Lemma \ref{lem:Schwartz2general m}, will be used in the next section to study the decay of the denominator integral in the integral formula for the Szeg\H{o} kernel derived in Proposition \ref{prop:integral szego1}.

\begin{lemma}\label{lem:Schwartz2general m} Let

\begin{equation}\label{eq:g}
g(\vec{v}) = \sum\limits_{\alpha \in \Gamma} c_{\alpha}\vec{v}^\alpha
\end{equation}

\noindent be a strictly convex polynomial in $\mathbb{R}^n$ such that

\begin{enumerate}[i)]
\item\label{itm:gorigen} $g(\vec{0})=0;$ 
\item\label{itm:ggrad} $\nabla g(\vec{0})=0;$ 
\item\label{itm:John} there exists a constant $0<A<1$ such that $\{\vec{v}\,:\, |\vec{v}|\le A\} \subseteq \{\vec{v}\,:\,g(\vec{v})\le 1\} \subseteq \{\vec{v}\,:\,|\vec{v}|\le 1\};$ and 
\item\label{itm:grado} there exist positive integers $m_1,\ldots,m_n$ such that the combined degree  of $g$ is \\$(m_1,\ldots,m_n)$  (refer to definition on page \pageref{def:``combined degree''}).
\end{enumerate}

\noindent Then
$$\theta(\vec{\eta}) = \left[ \,\, \displaystyle\int \limits_{\mathbb R^n} e^{\vec{\eta}\cdot \vec{v} - g(\vec{v}) }\, d\vec{v} \right]^{-1}$$

\noindent is a Schwartz function. Moreover, its decay depends only on the constant $A$ and the exponents $\{m_1,\ldots,m_n\}.$
\end{lemma}

\bigskip

\begin{remark}\label{itm:gsuma}

Notice that under these assumptions, the coefficients of the polynomial $g(\vec{v}) = \sum_{\alpha\in\Gamma} c_{\alpha} \vec{v}^{\alpha}$ satisfy  $ \sum_{\alpha\in\Gamma}|c_{\alpha}|\le C,$ where $C$ depends only on the constant $A,$ on the degree of the polynomial and  on the dimension of the space. This was shown in Corollary \ref{cor:coeficientes acotados} on page \pageref{cor:coeficientes acotados}.

\end{remark}

\bigskip

\bigskip

Let $I = \displaystyle\int_{\mathbb{R}^n} e^{\vec{\eta}\cdot\vec{v} - g(\vec{v}) }\, d\vec{v}.$ We will show that $I$ grows at an exponential rate. We can write 

\begin{equation}\label{eq:def de I}
 I = e^{h(\vec{v_0})} \int_{\mathbb{R}^n} e^{h(\vec{v})-h(\vec{v_0})} \, d\vec{v},
\end{equation}
\noindent where $h(\vec{v}) = \vec{\eta}\cdot \vec{v} - g(\vec{v})$ and $\vec{v_0}$ is the point where $h(\vec{v})$ attains its maximum (notice that $\vec{\eta} = \nabla g(\vec{v_0})).$ Notice that from the growth condition of $g,$ and from the fact that $g$ is positive, it follows that $h$ is bounded above, and therefore attains a maximum. The convexity hypothesis further ensures that $h$ is strictly concave, and therefore that the maximum, $\vec{v_0},$ is unique. 

\bigskip

 Notice that $h(\vec{v_0}) = L(\vec{\eta}) = \sup_{\vec{v}} \left\{ \vec\eta\cdot\vec{v} - g(\vec{v})\right\}$ is the Legendre Transform of $g.$ We will show that the dominant term, $e^{h(\vec{v_0})}$, grows at an exponential rate in $\vec{\eta}.$ This term will provide the desired decay for $I^{-1}. $ We will then show that $ \int e^{h(\vec{v})-h(\vec{v_0})} \, d\vec{v}$ does not decrease too fast, that is, that it does not annul the growth of the dominant term. 
 
 \bigskip

\subsection{The dominant term}

We begin by studying the growth of the term $e^{h(\vec{v_0})}=e^{L(\vec{\eta})}.$ We show that $e^{L(\vec{\eta})}$ grows exponentially as a function of $\vec{\eta}.$ Moreover, we claim that the growth is independent of the choice of $g,$ but rather depends only on the constant $A,$ on the combined degree of $g$ and on the dimension of the space. More precisely, we show that there exist positive constants $C, \widetilde{C}$ which depend only on the combined degree of $g,$ the dimension of the space and the constant $A,$ such that

\begin{equation}\label{eq:estimate de dom term}
e^{L(\vec{\eta})} \ge \exp\left[{\widetilde{C} \left(|\eta_1|^\frac{2m_1}{2m_1-1}  +\ldots+|\eta_n|^\frac{2m_n}{2m_n-1}\right)}-C\right].
\end{equation}

\bigskip

We begin by showing that the polynomial $g$ is dominated, independently of its coefficients, by its pure terms of highest order. 

\bigskip

\begin{claim}\label{cl:bound g sin coefs en 2 dim} 

If $g(\vec{v})$ is as in the statement of Lemma \ref{lem:Schwartz2general m}, then there exists a constant $C>0$ that depends only on the constant $A$, on the combined degree of the polynomial and on the dimension of the space such that

$$g(\vec{v})  \le C(1 + v_1^{2m_1}+\ldots +v_n^{2m_n}).$$

\end{claim}
\bigskip

\begin{proof} 
Let 
$$r(\vec{v}) = v_1^{2m_1} +\ldots+ v_n^{2m_n}.$$

\noindent Notice that for any $\vec{v}\in\mathbb{R}^n,$

\begin{equation}\label{eq:primer bound de r 1}
v_1^{\alpha_1}\cdots v_n^{\alpha_n} \le r(\vec{v})^{\frac{\alpha_1}{2m_1}+\cdots+\frac{\alpha_n}{2m_n}}.
\end{equation}

\bigskip

\noindent Also, since $g \ge 0,$ 

\begin{equation}\label{eq:un bound de g}\begin{split}
 g(\vec{v})  = \left\vert g(\vec{v})\right\vert  \le \sum\limits_{\alpha\in\Gamma} \left\vert c_{\alpha}\right\vert \left\vert v_1^{\alpha_1}\cdots v_n^{\alpha_n} \right\vert 
 \le\sum\limits_{\alpha\in\Gamma} \left\vert c_{\alpha}\right\vert \,\left\vert\, r(\vec{v})^{\frac{\alpha_1}{2m_1}+\cdots+\frac{\alpha_n}{2m_n}}\right\vert. \\
 \end{split}\end{equation}

\bigskip

\noindent Moreover, recall that since $g$ is of combined degree $(m_1,\ldots,m_n),$ any index $\alpha\in\Gamma$ satisfies that 

$$ \frac{\alpha_1}{2m_1} + \ldots +\frac{\alpha_n}{2m_n} \le 1.$$ 

\noindent Hence, 

\begin{equation}\label{eq:otro bound de g}
  r(\vec{v})^{\frac{\alpha_1}{2m_1}+\cdots+\frac{\alpha_n}{2m_n}} \le 1 +r(\vec{v}).
 \end{equation}

\noindent Thus, and since $\sum\limits_{\alpha\in\Gamma} |c_{\alpha}| \le C,$ it follows from equations (\ref{eq:un bound de g}) and (\ref{eq:otro bound de g}) that

$$ g(\vec{v}) \le \sum\limits_{\alpha\in\Gamma} |c_{\alpha}| (1+r(\vec{v})) \le C(1+r(\vec{v})).$$

\noindent This finishes the proof of Claim \ref{cl:bound g sin coefs en 2 dim}

\end{proof}

\bigskip

Since this estimate does not depend on the coefficients of $g,$ it is now easy to obtain a lower bound for $h(\vec{v_0})$ in terms of $\vec{\eta}$ which does not depend on the choice of $g.$    

\bigskip

\begin{claim}\label{cl:leg en gen en 2 dim} The Legendre Transform of $g(\vec{v})$ where $\vec{v} \in \mathbb{R}^n$ is large for large values of $|\vec{\eta}|$. More precisely, 

$$L(\vec{\eta}) \ge \widetilde{C} \left( \,|\eta_1|^\frac{2m_1}{2m_1-1}  +\ldots+ |\eta_n|^\frac{2m_n}{2m_n-1}\,\right)  -C,$$ 
\noindent where $C $ and $\widetilde{C}$ are positive constants that depend only on the constant $A,$ on the combined degree of $g$ and on the dimension of the space. 
\end{claim}

\bigskip

\begin{proof}

It follows from the previous claim that 

\begin{equation*}\begin{split}
L(\vec{\eta}) = & \sup_{\vec{v}} \text \{\vec{\eta}\cdot\vec{v} - g(\vec{v})\} \\ \ge  &\sup_{\vec{v}} \text \{\vec{\eta}\cdot\vec{v} - C - C|v_1|^{2m_1} - \ldots  - C|v_n|^{2m_n}\} \\ 
= & -C+  \sup_{v_1} \text \{\eta_1v_1 -  C|v_1|^{2m_1}\}  +\ldots + \sup_{v_n} \text \{\eta_nv_n -  C|v_n|^{2m_n}\} \\
\end{split}\end{equation*}

\noindent But given $w\in\mathbb{R},$ the Legendre Transform of $\frac{B}{2k} |w|^{2k}$ is $ \widetilde{B} |\eta|^\frac{2k}{2k-1},$ where $$\widetilde{B} = B^\frac{-1}{2k-1} \left( \frac{2k-1}{2k} \right).$$

\noindent Thus,

$$ L(\vec{\eta}) \ge \widetilde{C}\left(|\eta_1|^\frac{2m_1}{2m_1-1} + \ldots + |\eta_n|^\frac{2m_n}{2m_n-1}\right) - C, $$

\noindent where $\widetilde{C} = \min \left\{ B_1^\frac{-1}{2m_1-1} \left( \frac{2m_1-1}{2m_1}\right),\ldots, B_n^\frac{-1}{2m_n-1} \left( \frac{2m_n-1}{2m_n} \right) \right\},$ and $B_j = C2m_j.$

\end{proof}

\bigskip

This finishes the proof that the dominant term, $e^{L(\vec{\eta})},$ grows at an exponential rate in $\vec{\eta},$ independently of the coefficients of $g.$ More precisely,  we have shown that 

\begin{equation}\label{eq:termino dominante}
e^{L(\vec{\eta})} \ge \exp\left[{\widetilde{C} \left(|\eta_1|^\frac{2m_1}{2m_1-1}  +\ldots+|\eta_n|^\frac{2m_n}{2m_n-1}\right)}-C\right].
\end{equation}

\subsection{A polynomial bound for the remaining terms}

\noindent It suffices now to show that 
$$ J = \int_{\mathbb{R}^n} e^{h(\vec{v}) - h(\vec{v_0})} \, d\vec{v}$$

\noindent  is not too small to obtain the desired decay for $I^{-1}.$ Recall that

$$J = \int_{\mathbb{R}^n} e^{\vec{\eta}\cdot (\vec{v} - \vec{v_0}) +g(\vec{v_0}) - g(\vec{v}) } \, d\vec{v}.$$

\noindent In order to estimate this integral, we will use the fact that if $f:\mathbb{R}^n \rightarrow \mathbb{R}$ is a convex function such that $f(\vec{0})=0$ and $\nabla f(\vec{0}) = 0$ then (see Appendix)

$$ \int_{\mathbb{R}^n}e^{-f(\vec{w})}\,d\vec{w} \approx |\{ \vec{w} \, : \, f(\vec{w}) \le 1 \}|.$$

\bigskip

Since $\vec{\eta} = \nabla g(\vec{v_0}),$ and making the change of variables $\vec{w} = \vec{v} - \vec{v_0},$ we can write 

$$J = \int_{\mathbb{R}^n} e^{-f(\vec{w})} \, d\vec{w},$$ 

\noindent where

\begin{equation}\label{eq:f}
f(\vec{w}) = - \nabla g(\vec{v_0})\cdot\vec{w} - g(\vec{v_0}) + g(\vec{v_0}+\vec{w}).
\end{equation}

\noindent Clearly $f(\vec{0}) = 0$ and $\nabla f(\vec{0}) = 0.$ Also, since $g$ is convex, so is $f.$ Thus,

\begin{equation}\label{eq:jota en 2}
J \approx |\{\vec{w} : f(\vec{w}) \le 1\} |.
\end{equation}

\noindent That is, 

\begin{equation}\label{eq:numerador der}
I = \int_{\mathbb{R}^n} e^{\vec{\eta}\cdot\vec{v} - g(\vec{v}) }\, d\vec{v} \approx e^{h(\vec{v_0})} |\{\vec{w} : f(\vec{w}) \le 1\} |.
\end{equation}

\bigskip

Our goal is to show that as $|\vec{\eta}|$ grows, the volume given in equation (\ref{eq:jota en 2}) decreases slower than the rate of growth we obtained for $e^{h(\vec{v_0})}.$ We begin by obtaining an upper bound for $f$ that is independent of the choice of $g,$ but rather depends only on its combined degree, on the dimension of the space and on the constant $A$ (where the constant $A$ is from assumption {\it (iii)} of Lemma \ref{lem:Schwartz2general m}). To do so we will write $f$ as an integral in terms of the quadratic form associated to the Hessian of $g.$ In Claim \ref{cl:forma cuadratica} we obtain an upper bound for this quadratic form in terms of a polynomial that is independent of the coefficients of $g.$ In Claim \ref{cl:bound de f en gen 2} we use this estimate to obtain the desired bound for $f.$ 

\bigskip

\begin{claim}\label{cl:forma cuadratica}
There is a constant $C$ depending only on the combined degree of $g$ so that 

$$\sum_{i,j=1}^n g_{ij}(\vec{v}) w_iw_j \le C \left(1+r(\vec{v})\right)|\vec{w}|^2,$$

\noindent where $ r(\vec{v}) = v_1^{2m_1} +\ldots +v_n^{2m_n}$ and $g_{ij}(\vec{v}) =\dfrac{\d^2g}{\d v_i \d v_j}(\vec{v}).$

\end{claim}

\bigskip

\begin{proof}

\noindent Let $L$ be the Hessian matrix of $g$ so that $ \vec{w}^{\textbf{T}} \, L \, \vec{w} = \sum_{i,j=1}^n g_{ij} w_iw_j.$ Since $L$ is symmetric, it has $n$ linearly independent eigenvectors. Let $\vec{u_i},$ $i=1,\ldots,n$ be the eigenvectors of L, and $\lambda_i,$ $i=1,\ldots,n$ be the corresponding eigenvalues. Since $g$ is convex, the matrix $L$ is positive semi-definite, so its eigenvalues are non-negative. 

\bigskip

 Let $P = Tr(L)I,$ where $Tr(L) = \lambda_1 +\ldots +\lambda_n$ is the trace of the matrix $L$ and $I$ is the identity matrix. Let $Q = P-L.$ We claim that $Q$ is positive semi-definite, and hence that $L \le P$ as quadratic forms. In fact, notice that for $i=1,\ldots,n$ 
$$Q\vec{u_i} =  P\vec{u_i} - L\vec{u_i} =  Tr(L)\vec{u_i} - \lambda_i \vec{u_i} = \left(Tr(L)-\lambda_i\right)\vec{u_i} = \sum_{\substack{1\le j \le n\\  j \ne i}}\lambda_j \vec{u_i}.$$

\bigskip

\noindent Thus, for $i=1,\ldots,n,$ $\vec{u_i}$ is an eigenvector of $Q,$ with eigenvalue 
$$\mu_i = \sum_{\substack{1\le j \le n\\  j \ne i}}\lambda_j>0.$$

\bigskip

\noindent Thus, since $Q$ is a symmetric matrix whose eigenvalues are non-negative, $Q$ is positive semi-definite. Hence, since $ \vec{w}^{\textbf{T}} \, L \, \vec{w}\, \le \, \vec{w}^{\textbf{T}} \, P\,\vec{w},$ it follows that   

\begin{equation*} \begin{split}
 0\le &\,\sum_{i,j=1}^n g_{ij}(\vec{v}) w_iw_j 
 \le   \left(|g_{11}(\vec{v})|+\ldots+ |g_{nn}(\vec{v})|\right)|w|^2.\\
 \end{split}\end{equation*}
 
\bigskip

\noindent Notice that each $g_{jj}(\vec{v})$ is a polynomial of the form $\sum\limits_{\beta} \widetilde{c_{\beta}}\vec{v}^{\beta}$ where the indexes $\beta$ satisfy 

$$ \frac{\beta_1}{2m_1} +\cdots+ \frac{\beta_n}{2m_n} \le 1 - \frac{1}{m_j} < 1.$$

\noindent In particular, this implies that 

\begin{equation*}\label{eq:r alto bajo}
r(\vec{v})^{\frac{\beta_1}{2m_1}+\cdots+\frac{\beta_n}{2m_n}} \le 1 +r(\vec{v}).
\end{equation*}

\noindent Moreover, since for any $\vec{v}\in\mathbb{R}^n$ we have that $v_1^{\beta_1}\cdots v_n^{\beta_n} \le r(\vec{v})^{\frac{\beta_1}{2m_1}+\cdots+\frac{\beta_n}{2m_n}}$ it follows that

\begin{equation*}\begin{split}
 |g_{jj}(\vec{v})| \le &\, \sum\limits_{\beta} \left\vert\widetilde{c_{\beta}}\right\vert \left\vert v_1^{\beta_1}\cdots v_n^{\beta_n}\right\vert \le   \sum\limits_{\beta} \left\vert\widetilde{c_{\beta}}\right\vert \left(1+r(\vec{v})\right).\\
\end{split}\end{equation*}

\bigskip

Since $\sum_{\alpha\in\Gamma}|c_\alpha| \le C,$ the sum $\sum_{\beta\in\Gamma}\widetilde{c_{\beta}}$ is also bounded by a constant that does not depend on the choice of $g,$ but rather on the combined degree of $g.$ Then, and by the previous inequality, we have that

$$|g_{11}(\vec{v})| + \ldots + |g_{nn}(\vec{v})|  \le C( 1 + r(\vec{v})),$$

\noindent where $C$ depends only on the combined degree of $g.$ That is, 

$$ \sum_{i,j=1}^ng_{ij}(\vec{v})w_iw_j \le C (1+r(\vec{v}))|\vec{w}|^2.$$ 

\end{proof}

\bigskip

Using this result it is now possible to obtain an upper bound for $f$ which is independent of the choice of $g.$ We do so in the following claim.  

\bigskip

\begin{claim}\label{cl:bound de f en gen 2} If $$r(\vec{v}) = v_1^{2m_1} +\ldots+v_n^{2m_n}$$

\noindent and 

$$f(\vec{w}) = - \nabla g(\vec{v_0})\cdot\vec{w} - g(\vec{v_0}) + g(\vec{v_0}+\vec{w})$$

\noindent then, 

$$ f(\vec{w}) \lesssim |\vec{w}|^2 ( 1+ r(\vec{v_{0}})  + r( \vec{w})),$$

\noindent where the constant depends only on the constant $A,$ the combined degree of $g$ and the dimension of the space. 

\end{claim}
\bigskip

\begin{proof} We begin by rewriting $f$ as an integral in terms of the quadratic form associated to the Hessian, so that we can apply our previous estimate. Integrating by parts, we can write

\begin{equation*}
 g(\vec{v_0}+\vec{w}) - g(\vec{v_0}) = \nabla g(\vec{v_0})\cdot \vec{w} + \int_0^1 \sum_{i,j=1}^n g_{ij}(\vec{v_0} + t\vec{w}) w_iw_j (1-t) \, dt.
 \end{equation*}

\noindent It follows that

\begin{equation}\label{eq:f en int en 2}
f(\vec{w}) =   \int_0^1 \sum_{i,j=1}^n g_{ij}(\vec{v_0} + t\vec{w}) w_iw_j(1-t) dt. 
\end{equation}

\noindent In particular, by convexity of $g$ we have that $f \ge 0.$

\bigskip

We can now use the bound for the Hessian obtained in Claim \ref{cl:forma cuadratica}. It follows that 

\begin{equation}\label{eq:f f}
f(\vec{w}) \lesssim \int_0^1 (1+r(\vec{v_{0}} + t \vec{w})) |\vec{w}|^2 (1-t) \, dt.
\end{equation}

\bigskip

\noindent Using convexity, it is easy to show that $r(\vec{u}+\vec{v}) \le  \max\{2^{2m_1-1},\ldots,2^{2m_n-1} \}[ r(\vec{u}) + r(\vec{v})].$ Applying this inequality to $r(\vec{v_{0}} + t \vec{w}), $ and since  $0\le t\le 1,$ we have that 

$$ r(\vec{v_{0}} + t \vec{w}) \lesssim r(\vec{v_{0}})  + r( t\vec{w}) \le r(\vec{v_{0}})  + r( \vec{w}).$$

\noindent Hence, it follows from equation (\ref{eq:f f}) that

$$ f(\vec{w}) \lesssim |\vec{w}|^2 ( 1+ r(\vec{v_{0}})  + r( \vec{w})).$$

\end{proof}

\bigskip
 In the next three claims we show that there is a polynomial $P$ and a constant $C$ depending only on the degrees $\{m_1,\ldots,m_n\}$  so that
$$ |\{ \vec{w}\,:\, f(\vec{w}) \le 1\}|^{-1} \le C(1+P(|\vec{\eta}|))^\frac{n}{2}.$$ 
\noindent In Claim \ref{cl:cota de f} we show that $ |\{ \vec{w}\,:\, f(\vec{w}) \le 1\}|^{-1}$ is bounded by a polynomial in terms of $r(\vec{v_0}),$ and in Claim \ref{cl:porte de v0 gen} we compare the sizes of $|\vec{v_0}|$ and  $|\vec{\eta}|.$ In Claim \ref{cl: r final} we conclude that $r(\vec{v_0})$ grows at most at a polynomial rate in $|\vec{\eta}|.$ 

\bigskip

\begin{claim}\label{cl:cota de f} If $ f(\vec{w}) \lesssim |\vec{w}|^2 ( 1+ r(\vec{v_0}) +r(\vec{w})),$ then

$$|\{\vec{w}:f(\vec{w}) \le 1 \}| \gtrsim (1+r(\vec{v_0}))^{-\frac{n}{2}}.$$ 

\end{claim}

\bigskip

\begin{proof} 

Let $C$ be such that $ f(\vec{w}) \le C |\vec{w}|^2 ( 1+ r(\vec{v_0}) +r(\vec{w}))$  and let

  $$T(\vec{w}) = C |\vec{w}|^2 ( 1+ r(\vec{v_0}) +r(\vec{w})).$$

\noindent  Then,

 $$|\{\vec{w}\,:\,f(\vec{w}) \le 1 \}| \ge |\{\vec{w}\,:\,T(\vec{w}) \le 1 \}|.$$

\bigskip

Let $ \Sigma = \{ \vec{u}\,:\, T(\vec{u})=1\} $  and let $m = \min \{ \,|\vec{u}|\,:\,\vec{u}\in\Sigma \}.$ Choose $\vec{w_T}$ such that $T(\vec{w_T}) = 1 $ and $|\vec{w_T}| = m.$ Then the set $|\{\vec{w}\,:\, T(\vec{w}) \le1 \}|$ is bounded from below by the volume of the ball of radius $|\vec{w_T}|.$ That is, 

$$|\{\vec{w}\,:\, T(\vec{w}) \le1 \}| \ge C_{n}|\vec{w_T}|^{n},$$

\noindent where $C_{n} = \frac{\pi^\frac{n}{2}}{\Gamma\left( \frac{n}{2}+1\right)}.$

\bigskip

 Let $a=  1+ r(\vec{v_0}).$ Our goal is to show that $|\vec{w_T}|^n \gtrsim a^{-\frac{n}{2}}. $   If  $|\vec{w_{T}}|^2  \geq \frac{1}{2Ca}, $ then $|\vec{w_{T}}|^n  \gtrsim a^{-\frac{n}{2}}, $ as desired. Otherwise, we have that $|\vec{w_{T}}|^2  < \frac{1}{2Ca}.$ Since $1 = T(\vec{w_T}) = C |\vec{w_T}|^2 ( 1+ r(\vec{v_0}) +r(\vec{w_T}))$ it follows that

$$ a|\vec{w_{T}}|^2 + |\vec{w_{T}}|^2  r(\vec{w_{T}}) = \frac{1}{C}.$$

\noindent But since $|\vec{w_{T}}|^2  < \frac{1}{2Ca},$ it follows that 

\begin{equation}\label{eq:cota con c}
 \frac{1}{2C} <  |\vec{w_{T}}|^2 r( \vec{w_{T}}).
 \end{equation}

\bigskip

 Also, since $a \ge 1,$ we have that $|\vec{w_{T}}|^2  < \frac{1}{2C}.$ Thus, $w_{Tj}^2 < \frac{1}{2C}$ for every $1\le j\le n.$  Hence, 

$$ r( \vec{w_{T}})  <  \left(\frac{1}{2C}\right)^{m_1} +\cdots + \left(\frac{1}{2C}\right)^{m_n}.$$ 

\noindent Using this in equation (\ref{eq:cota con c}) we have that 

$$\frac{1}{2C} <  |\vec{w_{T}}|^2  \left( \left(\frac{1}{2C}\right)^{m_1} +\cdots + \left(\frac{1}{2C}\right)^{m_n}\right).$$

\noindent Thus, and since $a \ge 1,$ it follows that 

$$ |\vec{w_{T}}|^2 > \Lambda \ge \frac{\Lambda}{a},$$

\noindent where $\Lambda  = \frac{1}{2C } \left( \left(\frac{1}{2C}\right)^{m_1} +\cdots + \left(\frac{1}{2C}\right)^{m_n}\right)^{-1}$ is a strictly positive constant. 

\bigskip

 Therefore, 

$$ |\{\vec{w}:f(\vec{w}) \le 1 \}| \gtrsim |\vec{w_T}|^n \gtrsim a^{-\frac{n}{2}}. $$ 

\noindent This finishes the proof of Claim \ref{cl:cota de f}. 

\end{proof}

\bigskip

\begin{claim}\label{cl:porte de v0 gen}
\noindent There exist positive constants $\beta_1$ and $\beta_2$ such that 

$$ |\vec{v_0}| \le \beta_1|\vec{\eta}| +\beta_2.$$ 

\noindent The constants depend only on $m_1,\ldots,m_n$ and the dimension of the space.

\end{claim}

In the proof of Claim \ref{cl:porte de v0 gen}, we use Lemma 2.2 of \cite{BrNaWa88}. For the reader's convenience, we state the Lemma below.

\begin{lemaBNW}
Let $C(m,T)$ denote  the space of polynomials 

$$P(t) = \sum_{j=0}^m a_jt^j$$

\noindent which satisfy:

\begin{enumerate}[(a)]
\item{The degree of $P$ is no bigger than $m;$}
\item{ $P(0) = a_0 =0;$ $P'(0) = a_1 =0;$}
\item{$P$ is convex for $0\le t\le T.$}
\end{enumerate}

\noindent Then there is a constant $C_m,$ independent of $T,$ so that if $P\in C(m,T),$ $P(t) = \sum_{j=2}^m a_j t^j,$ then 

$$P'(t) \ge C_m \sum_{j=2}^m |a_j| t^{j-1}$$

\noindent for $0\le t\le T.$ In particular, 

\begin{equation}\label{eq:Bruna derivadas}\begin{split}
& P'(t) \ge C_m t^{m-1} \sum_{j=2}^m |a_j|  \qquad \, {\text if} \,\,\, 0\le t \le 1, \qquad {\text and}\\
& P'(t) \ge C_m t \sum_{j=2}^m |a_j|  \qquad \qquad  {\text if} \,\,\, 1\le t \le T.\\
\end{split}\end{equation}

\end{lemaBNW} 

\bigskip

\begin{proof}[Proof of Claim \ref{cl:porte de v0 gen}] The statement is trivial if $|\vec{v_0}|\le 1,$ so we will assume that $|\vec{v_0}| > 1.$ Let $G(t) = g\left(\frac{t\vec{v_0}}{|\vec{v_0}|}\right).$  Then

$$G'(t) = \nabla g\left(\frac{t\vec{v_0}}{|\vec{v_0}|}\right)\cdot \left(\frac{\vec{v_0}}{|\vec{v_0}|}\right).$$ 

\noindent Thus, since $\nabla g(\vec{0})= \vec{0}$ by hypothesis,  $G'(0) = \nabla g(\vec{0}) \cdot \left(\frac{\vec{v_0}}{|\vec{v_0}|}\right)= 0.$ Also, notice that since $g$ is convex, so is $G.$ Hence, and since $G$ is a polynomial, $G'(t) >0$ if $t>0.$  By Cauchy-Schwarz, 

$$ |G'(t)| \le \left|\nabla g\left(\frac{t\vec{v_0}}{|\vec{v_0}|}\right)\right|.$$

\noindent Evaluating at $t= |\vec{v_0}|$ we have that $ |G'(|\vec{v_0}|)| \le |\nabla g(\vec{v_0})| = |\vec{\eta}|.$ But since $|\vec{v_0}|>0, $ it follows that 

\begin{equation}\label{eq:bound de G}
 G'(|\vec{v_0}|)= |G'(|\vec{v_0}|)| \le |\vec{\eta}|.
 \end{equation}
 
\noindent It suffices now to obtain a polynomial lower bound for $G'(|\vec{v_0}|)$ in terms of $|\vec{v_0}|.$ To do so, we use Lemma 2.2 of \cite{BrNaWa88}. 

\bigskip 

Notice that $G(t)$ is a convex polynomial of one variable such that $G(0) = G'(0)=0,$ so we can use the aforementioned result. Write 

$$ G(t)  = \sum_{j=2}^m a_jt^j.$$

\noindent Since we are considering $|\vec{v_0}|> 1,$ it follows from equations (\ref{eq:bound de G}) and (\ref{eq:Bruna derivadas}) that 

\begin{equation}\label{eq:ya con bruna}
|\vec{v_0}| \sum_{j=2}^m |a_j| \lesssim G'(|\vec{v_0}|) \le |\vec{\eta}|.
\end{equation}

\noindent It suffices now to obtain a lower bound for $\sum_{j=2}^m |a_j|,$ which must be independent of the choice of $g.$ To do so, we use assumption (iii) of Lemma \ref{lem:Schwartz2general m}, namely, the fact that

$$ \{ \vec{v} \,:\,g(\vec{v}) \le 1\} \subseteq \{ \vec{v}\,:\,|\vec{v}| \le 1 \}.$$ 

\noindent In particular, if $|\vec{v}| = 1,$ it must follow that $g(\vec{v}) \ge 1.$ Thus, evaluating at $t=1,$ it follows that

$$ G(1) = g\left( \frac{\vec{v_0}}{|\vec{v_0}|}\right) \ge 1.$$ 

\noindent But $G(1) = \sum_{j=2}^m a_j \le \sum_{j=2}^m |a_j|.$ Using this bound on equation (\ref{eq:ya con bruna}) yields $ |\vec{v_0}|  \lesssim |\vec{\eta}|.$ This finishes the proof of Claim \ref{cl:porte de v0 gen}.

\end{proof}

\bigskip

\begin{claim}\label{cl: r final}
 For $p\in \mathbb{R},$ $(1+r(\vec{v_0})^p)^{\frac{n}{2}}$ is at most of polynomial growth in $|\vec{\eta}|.$ 
\end{claim}

\begin{proof} The proof is trivial. In fact, since $|\vec{v_0}| \le \beta_1|\vec{\eta}| +\beta_2,$ it is clear that $r(\vec{v_0}) = v_{01}^{2m_1}+\ldots+ v_{0n}^{2m_n} \le |\vec{v_0}|^{2m_1} + \ldots +|\vec{v_0}|^{2m_n}$ is bounded from above by a polynomial in $|\vec{\eta}|.$ 

\end{proof}

\noindent It follows from these claims that there exists a polynomial $P(|\vec{\eta}|),$ which does not depend on the choice of $g,$ such that 

$$ \theta(\vec{\eta}) = I^{-1} \lesssim \text{exp}\left[-C\left(|\eta_1|^\frac{2m_1}{2m_1-1}+ \ldots +|\eta_n|^\frac{2m_n}{2m_n-1}\right)\right] (1+P(|\vec{\eta}|))^{\frac{n}{2}}.$$

\bigskip

\noindent  This finishes the proof that $\theta(\vec{\eta})$ decays at an exponential rate. Moreover, this decay is independent of the coefficients of the polynomial $g$ that defines it. We must now show that the same is true of all the derivatives of $\theta(\vec{\eta}).$

\bigskip

\subsection{Decay of the derivatives}

 The derivatives of $\theta(\vec{\eta})$ consist of sums of terms of the form 

\begin{equation}\label{eq:derivadas en gen}
\frac{C\left[\int_{\mathbb{R}^n} e^{\vec{\eta}\vec{v} - g(\vec{v})}v_1^{i_{1,1}}\cdots v_n^{i_{n,1}}\,d\vec{v}\right]^{a_1}\cdots \left[\int_{\mathbb{R}^n} e^{\vec{\eta}\vec{v} - g(\vec{v})}v_1^{i_{1,r}}\cdots v_n^{i_{n,r}}\,d\vec{v}\right]^{a_r}}{\left[\int_{\mathbb{R}^n} e^{\vec{\eta} \vec{v} - g(\vec{v})}\,d\vec{v}\right]^d},
\end{equation}

\noindent where $i_{1,1},\ldots, i_{n,r}, a_1, \ldots, a_r,d \in \mathbb{N} $ and $a_1+ \ldots + a_r + 1 = d.$ We show that each of these terms decays rapidly, and that the decay depends only on the coefficients of $g.$

\begin{remark}
The fact that $a_1+ \ldots + a_r -d <0$ is crucial. As before (equation (\ref{eq:def de I})), we can factor out a term $e^{h(\vec{v_0})}$ for each of these integrals. That is, we will factor out $ \left(e^{h(\vec{v_0})}\right)^{a_1+\ldots+a_r - d } = e^{-h(\vec{v_0})}.$ This term will provide the desired decay. 
\end{remark}

\bigskip

In order to understand the decay of the derivatives of $\theta(\vec{\eta})$ we need to study integrals of the form $\displaystyle\int_{\mathbb{R}^n} e^{\vec{\eta}\vec{v} - g(\vec{v})} v_1^{i_1} \cdots v_n^{i_n}  \, d\vec{v}.$

\bigskip

\begin{claim}\label{cl:bound de las derivadas}

$$\left\vert \int_{\mathbb{R}^n} e^{\vec{\eta}\vec{v} - g(\vec{v})} v_1^{i_1} \cdots v_n^{i_n}  \, d\vec{v}\right\vert \lesssim e^{h(\vec{v_0})} H_f[|\vec{v_0}|],$$

\noindent where

\begin{equation*}\begin{split}
 H_f\left[|\vec{v_0}|\right] &\, = \sum_{s_1=0}^{i_1} \cdots \sum_{s_n=0}^{i_n} \binom{i_1}{s_1}\cdots  \binom{i_n}{s_n} |\vec{v_0}|^{i_1+\ldots+i_n-s} \\
&\, \times \left( \, |\{\vec{w}\,:\, f(\vec{w})\le 1 \}|  \,  (1 + |\vec{v_0}|^{sB}) + \Theta \, \right);\\
\end{split}\end{equation*}

\noindent $\Theta$ is a constant that depends only on the combined degree of $g$ and the dimension of the space; $s= {s_1}+\cdots+{s_n};$ and $B =4 \max \{m_1,\ldots,m_n\}.$

\end{claim}

\begin{proof} As before (equation (\ref{eq:def de I})), we can write 

\begin{equation}\label{eq:I tilde en m}
 \widetilde{I} = \int_{\mathbb{R}^n} e^{\vec{\eta}\vec{v} - g(\vec{v})} v_1^{i_1} \cdots v_n^{i_n}  \, d\vec{v} = e^{h(\vec{v_0})} \int_{\mathbb{R}^n} e^{h(\vec{v})-h(\vec{v_0})} v_1^{i_1} \cdots v_n^{i_n}\, d\vec{v},
\end{equation}

\noindent where $h(\vec{v}) = \vec{\eta}\cdot \vec{v} - g(\vec{v})$ and $\vec{v_0}$ is the point where $h(\vec{v})$ attains its maximum;  and 

$$ \widetilde{J} = \int_{\mathbb{R}^n} e^{h(\vec{v})-h(\vec{v_0})} v_1^{i_1} \cdots v_n^{i_n}\, d\vec{v} = \int_{\mathbb{R}^n}(w_1 + v_{01})^{i_1} \cdots (w_n + v_{0n})^{i_n} e^{-f(\vec{w})} \, d\vec{w},$$

\noindent where $f(\vec{w})=g(\vec{v_0}+\vec{w}) -g(\vec{v_0}) - \nabla g(\vec{v_0})\cdot \vec{w}$ as in equation (\ref{eq:f}). Writing

$$ \widetilde{J} = \sum_{s_1=0}^{i_1} \cdots \sum_{s_n=0}^{i_n} \binom{i_1}{s_1}\cdots  \binom{i_n}{s_n}v_{01}^{i_1-s_1}\cdots  v_{0n}^{i_n-s_n} \int_{\mathbb{R}^n}w_1^{s_1} \cdots w_n^{s_n} e^{-f(\vec{w})}\, d\vec{w},$$

\noindent it follows that

$$|\widetilde{J}| \le \sum_{s_1=0}^{i_1} \cdots \sum_{s_n=0}^{i_n} \binom{i_1}{s_1}\cdots  \binom{i_n}{s_n}|\vec{v_0}|^{i_1 +\ldots + i_n -(s_1 + \ldots + s_n)} \int_{\mathbb{R}^n}|\vec{w}|^{{s_1}+\ldots+{s_n}} e^{-f(\vec{w})}\, d\vec{w}.$$

 Let $s= {s_1}+\cdots+{s_n}$ and ${J_s} =\displaystyle \int_{\mathbb{R}^n}|\vec{w}|^s  e^{-f(\vec{w})}\, d\vec{w}.$ Write

\begin{equation}\label{eq:jota s}
J_s = \int_{\{\vec{w}\in \mathbb{R}^n\,:\, |\vec{w}|\le 1 \}} |\vec{w}|^s e^{-f(\vec{w})}\, d\vec{w} +\int_{\{\vec{w}\in \mathbb{R}^n\,:\, |\vec{w}| > 1 \}} |\vec{w}|^s e^{-f(\vec{w})}\, d\vec{w} = J_{s_1} + J_{s_2}.
\end{equation}

\noindent  Then

\begin{equation}\label{eq:parte chica}
 J_{s_1} \le  \int_{\mathbb{R}^n} e^{-f(\vec{w})}\, d\vec{w} \approx |\{ \vec{w} \,:\, f(\vec{w})\le 1\}|.
\end{equation}

\noindent Given $\vec{v_0},$ we can estimate the size of $J_{s_2}$ by splitting the integral into the two following regions:

\begin{equation}\label{eq:subdivision}
 J_{s_2} = \int\limits_{\substack{ |\vec{w}|>1 \\ |\vec{w}|\le\lambda |\vec{v_0}|^B }} |\vec{w}|^s e^{-f(\vec{w})}\, d\vec{w} + \int\limits_{\substack{ |\vec{w}|>1 \\ |\vec{w}|>\lambda |\vec{v_0}|^B }} |\vec{w}|^s e^{-f(\vec{w})}\, d\vec{w},
\end{equation}

\bigskip

\noindent for some large constant $\lambda$ yet to be determined and $B =4 \max \{m_1,\ldots,m_n\}.$   Then

\begin{equation}\label{eq:jota s 1}
 \int\limits_{\substack{ |\vec{w}|>1 \\|\vec{w}|\le\lambda |\vec{v_0}|^B }} |\vec{w}|^s e^{-f(\vec{w})}\, d\vec{w}  \le \, \lambda^s |\vec{v_0}|^{sB} \int_{\mathbb{R}^n} e^{-f(\vec{w})}\, d\vec{w} \approx \lambda^s |\vec{v_0}|^{sB} | \{\vec{w}\,:\, f(\vec{w}) \le 1 \}|.
\end{equation}

\bigskip

 In order to estimate $\displaystyle\int_{\{\vec{w}\,: \,|\vec{w}|>1, |\vec{w}|>\lambda |\vec{v_0}|^B \}} |\vec{w}|^s e^{-f(\vec{w})}\, d\vec{w},$ we will find a lower bound in this region for $f(\vec{w})$ in terms of $|\vec{w}|^2$ and we will then bound the integral by a constant. Since $f \ge 0,$

\begin{equation}\label{eq:division de f}\begin{split}
f(\vec{w}) \ge&\,  |g(\vec{v_0}+\vec{w})| - |g(\vec{v_0})| - |\nabla g(\vec{v_0})\cdot\vec{w} |.  
\end{split}\end{equation}

\noindent We will show that $g(\vec{v_0}+\vec{w})$ is bounded from below by a constant multiple of $|\vec{w}|^2.$ It will then suffice to show that the remaining terms in the above expression can be dominated by this bound. 

\bigskip

Let 

$$ F(t)  = g\left(\frac{t(\vec{v_0}+\vec{w})}{|\vec{v_0}+\vec{w}|} \right),$$ 

\noindent where $t\in\mathbb{R}.$ Then $F(t)$ is a convex polynomial in one variable, such that $F(0) = F'(0) = 0.$ We will write 

$$ F(t) = \sum_{j=2}^M a_j t^j.$$

\noindent Notice that in the region we are considering, and since $B>1,$ we have that

\begin{equation}\label{eq:desig tr en v0}
 |\vec{v_0}+\vec{w}| \geq |\vec{w}| - |\vec{v_0}| > |\vec{w}|  - \frac{|\vec{w}|^{\frac{1}{B}}}{\lambda^{\frac{1}{B}}}  \geq  |\vec{w}|\left( 1-\frac{1}{\lambda^\frac{1}{B}}\right)>1-\frac{1}{\lambda^\frac{1}{B}}. 
\end{equation} 

\noindent In particular, if $\lambda > 2^B,$ then $ |\vec{v_0}+\vec{w}| > 1/2.$ But it follows from Lemma 2.1 of \cite{BrNaWa88} (refer to equation (\ref{eq:Bruna Nagel Wainger}) on page \pageref{eq:Bruna Nagel Wainger}), that there exists a constant $C_M>0$ that depends only on the degree of $F$ such that

$$  g(\vec{v_0}+\vec{w}) = F(|\vec{v_0}+\vec{w}|)   \geq C_M |\vec{v_0}+\vec{w}|^2\, \sum_{j=2}^M |a_j|.$$

\noindent Furthermore, we claim that  $\sum_{j=2}^M |a_j| \ge 1$ so that $  g(\vec{v_0}+\vec{w})  \geq C_M |\vec{v_0}+\vec{w}|^2.$ In fact, since by hypothesis $ \{ \vec{v} \,:\,g(\vec{v}) \le 1 \} \subseteq \{ \vec{v} \,:\,|\vec{v}| \le 1 \},$ it follows that

$$ F(1) = g\left( \frac{\vec{v_0}+\vec{w}}{|\vec{v_0}+\vec{w}|}\right) \ge 1. $$

\noindent Therefore, $ \sum_{j=2}^M |a_j| \ge \sum_{j=2}^M a_j = F(1) \ge 1.$ Also, by equation (\ref{eq:desig tr en v0}), it follows that for $\lambda > \left( \frac{\sqrt{2}}{\sqrt{2}-1}\right)^B,$

\begin{equation}\label{eq:bound de g v0 w}
 g(\vec{v_0}+\vec{w}) \ge C_M |\vec{v_0}+\vec{w}|^2  \ge C_M |\vec{w}|^2\left( 1-\frac{1}{\lambda^\frac{1}{B}}\right)^2  \ge \frac{C_M|\vec{w}|^2}{2}.
\end{equation}

\noindent We would now like to obtain an upper bound for $|g(\vec{v_0})|.$  Recall that by Claim \ref{cl:bound g sin coefs en 2 dim}, for any $\vec{v} \in \mathbb{R}^n$ we have that $g(\vec{v}) \le C (1+r(\vec{v})).$ Thus, and since $ \max \{2m_1,\ldots,2m_n\}  = B/2 < B,$
 
 \begin{equation*}\begin{split}
  |g(\vec{v_0})| & \,  \le   C (1+|\vec{v_0}|^{2m_1} + \ldots + |\vec{v_0}|^{2m_n}) \le C(1+n + n|\vec{v_0}|^B) < C\left(1+n+\frac{n|\vec{w}|}{\lambda}\right).\\
  \end{split}\end{equation*}

\noindent Thus, for $\lambda > \frac{8Cn}{C_M},$ it follows that

\begin{equation}\label{eq:bound gv0}
  |g(\vec{v_0})| \le C\left(1+n+ \frac{n}{\lambda} + \frac{n|\vec{w}|^2}{\lambda} \right)  \le  C\left(1+n+ \frac{n}{\lambda} \right) + \frac{C_M|\vec{w}|^2}{8}.
\end{equation}

\noindent It now suffices to obtain an upper bound for $|\nabla g(\vec{v_0})\cdot\vec{w}|.$  Notice that for each $1\le j\le n,$ the $j^{th}$ entry of $\nabla g$ is a polynomial whose exponents satisfy 

$$  \frac{2\alpha_1}{B} +\cdots +\frac{2\alpha_n}{B} \le  \frac{\alpha_1}{2m_1} +\cdots +\frac{\alpha_n}{2m_n} \le 1 - \frac{1}{2m_j} < 1.$$

\noindent That is, $ \alpha_1 + \ldots + \alpha_n < \frac{B}{2}.$ Thus, we can bound each entry of $|\nabla g(\vec{v_0})|$ by a constant multiple of $1+|\vec{v_0}|^\frac{B}{2}.$ The coefficients of each of these entries are multiples of the coefficients of $g,$ where the factors depend only on the degree of $g.$ Thus, since $\sum_{\alpha\in\Gamma} |c_{\alpha}| \le C,$ there exists a constant $C_1$ that depends only on $C$ and the degree of $g$ such that

$$ |\nabla g(\vec{v_0})| \le C_1 (1+|\vec{v_0}|^\frac{B}{2}).$$

\noindent Hence, in the region under consideration we have 

$$ |\nabla g(\vec{v_0})| \le C_1 \left(1+\frac{|\vec{w}|^\frac{1}{2}}{\lambda^\frac{1}{2}}\right).$$

\noindent It follows that 

\begin{equation*}
|\nabla g(\vec{v_0})\cdot\vec{w}| \le |\nabla g(\vec{v_0})| |\vec{w}| \le   C_1 \left(1+\frac{|\vec{w}|^\frac{1}{2}}{\lambda^\frac{1}{2}}\right) |\vec{w}| =  C_1 \left(|\vec{w}|+\frac{|\vec{w}|^\frac{3}{2}}{\lambda^\frac{1}{2}}\right).
\end{equation*}

\noindent Since $|\vec{w}| \le  \frac{1}{A}|\vec{w}|^2 +A^2$ for any constant $A>0,$ and since $|\vec{w}|^\frac{3}{2}\le |\vec{w}|^2 +1,$ we have that

\begin{equation*}
|\nabla g(\vec{v_0})\cdot\vec{w}| \le  C_1 \left(  \frac{1}{A}|\vec{w}|^2 +A^2+\frac{|\vec{w}|^2}{\lambda^\frac{1}{2}}+\frac{1}{\lambda^\frac{1}{2}} \right).
\end{equation*}

\noindent Then for $A> \frac{16C_1}{C_M}$ and $\lambda> \left( \frac{16C_1}{C_M}\right)^2$ it follows that 

\begin{equation}\label{eq:bound gradiente}
|\nabla g(\vec{v_0})\cdot\vec{w}| \le  C_1 \left( A^2+\frac{1}{\lambda^\frac{1}{2}} \right) + \frac{C_M|\vec{w}|^2}{8}.
\end{equation}

\noindent Therefore, by equations (\ref{eq:bound de g v0 w}),(\ref{eq:bound gv0}) and (\ref{eq:bound gradiente}), and taking 

$$\lambda> \max\left\{ \left( \frac{\sqrt{2}}{\sqrt{2}-1}\right)^B, \frac{8Cn}{C_M},\left( \frac{16C_1}{C_M}\right)^2\right\}$$

\noindent it follows that 

\begin{equation}\label{eq:bound de f schwartz}
 f(\vec{w}) \ge |\vec{w}|^2 \left( \frac{C_M}{2} -\frac{C_M}{8} - \frac{C_M}{8} \right) - E = \frac{C_M|\vec{w}|^2}{4} -E,
 \end{equation}
 
 \noindent where $E$ is a constant that depends on the combined degree of $g$ and the dimension of the space, but is otherwise independent. 
 
 \bigskip 
 
 Recall that our goal is to obtain an upper bound for 

$$ I_s \,\,\,= \int\limits_{\substack{ |\vec{w}|>1\\ |\vec{w}|>\lambda |\vec{v_0}|^B }} |\vec{w}|^s e^{-f(\vec{w})}\, d\vec{w}.$$

\noindent Using the lower bound for $f(\vec{w})$ obtained in equation(\ref{eq:bound de f schwartz}) we have that

$$ I_s \,\,\, \lesssim \int\limits_{\substack{ |\vec{w}|>1 \\ |\vec{w}|>\lambda |\vec{v_0}|^B }} |\vec{w}|^s e^{\frac{-C_M|\vec{w}|^2}{4} }\, d\vec{w} \lesssim \int_0^{\infty} r^s e^{\frac{-C_Mr^2}{4}}\,dr.$$

\noindent Since $C_M$ is a strictly positive constant, the above integral converges. This finishes the proof of Claim \ref{cl:bound de las derivadas}.

\end{proof}

It follows from equation (\ref{eq:numerador der}) and Claim \ref{cl:bound de las derivadas} that the derivatives of $\theta({\vec{\eta}})$ are bounded from above by a sum of terms of the form 

$$ \frac{H_f[|\vec{v_0}|]^{d-1}}{e^{h(\vec{v_0})}|\{\vec{w}\,:\, f(\vec{w})\le 1 \}|^d}.$$

\noindent Moreover, by Claim \ref{cl:bound de las derivadas} these terms can be bounded by terms of the form

$$ \frac{q(|\vec{v_0}|)}{e^{h(\vec{v_0})}|\{\vec{w}\,:\, f(\vec{w})\le 1 \}|^k},$$

\noindent where $q:\mathbb{R} \rightarrow \mathbb{R}$ is a polynomial, and $k\in [1,d]\cap\mathbb{Z}.$ By Claim \ref{cl:porte de v0 gen} it follows that $q(|\vec{v_0}|)$ is bounded by a polynomial in $|\vec{\eta}|.$ Furthermore, by Claim \ref{cl:cota de f}, $|\{\vec{w}:f(\vec{w}) \le 1 \}|^{-k} \lesssim (1+r(\vec{v_0}))^{\frac{kn}{2}}.$ By Claim \ref{cl: r final} this latter bound is at most of polynomial growth in $|\vec{\eta}|.$ Thus, the derivatives of $\theta({\vec{\eta}})$ are bounded by sums of terms of the form $e^{-h(\vec{v_0})}\tilde{q}(|\vec{\eta}|),$ where $\tilde{q}$ grows at a polynomial rate. Finally, by equation (\ref{eq:termino dominante}) $e^{-h(\vec{v_0})}$ decays at an exponential rate in $|\vec{\eta}|.$ This finishes the proof that the derivatives of $\theta({\vec{\eta}})$ decay exponentially. Thus, $\theta$ is a Schwartz function. Moreover, it follows from the previous computations that its decay is independent of the coefficients of $g.$

\bigskip

\hfill$\square$
\bigskip

\section{Bounds for the Szeg\H{o} kernel}\label{main theorem}

In this section we present the proof of our main result. With $\Omega_b$ defined as before, let $(\vec{x},\vec{y},t)$ and $(\vec{x'},\vec{y'},t')$ be any two points in $\partial\Omega_b.$ Define

\begin{equation}\label{eq:b tilde}
\tilde{b}(\vec{v}) = b\left( \vec{v} +\frac{\vec{x}+\vec{x'}}{2}\right) - \nabla b\left(\frac{\vec{x}+\vec{x'}}{2}\right)\cdot\vec{v} - b\left(\frac{\vec{x}+\vec{x'}}{2}\right);
\end{equation}

\begin{equation}\label{eq:delta}
\delta(\vec{x},\vec{x'}) = b(\vec{x}) +b(\vec{x'}) -2b\left(\frac{\vec{x}+\vec{x'}}{2} \right);
\end{equation}

\noindent and

\begin{equation}\label{eq:w}
w = (t'-t)  + \nabla b \left(  \frac{\vec{x}+\vec{x'}}{2}\right) \cdot (\vec{y'}-\vec{y}).
\end{equation}

\noindent  We obtain the following estimate for the Szeg\H{o} kernel associated to the domain $\Omega_b:$

\begin{equation*}
\left\vert S\left((\vec{x},\vec{y},t);(\vec{x'},\vec{y'},t')\right)\right\vert \le \frac{C}{\sqrt{\delta^2+ \widetilde{b}(\vec{y}-\vec{y'})^2+ w^2 }\left| \left\{ \vec{v} \,:\, \tilde{b}(\vec{v}) < \sqrt{\delta^2+ \widetilde{b}(\vec{y}-\vec{y'})^2+ w^2 }\right\} \right|^2}.
\end{equation*}

\noindent Here the constant $C$ depends on the exponents $\{m_1,\ldots,m_n\}$ and the dimension of the space, but is independent of the two given points.

\begin{remark}

Here,  $ \tilde{b}$ is a strictly convex polynomial of the same combined degree as $b,$ but with $ \tilde{b}(\vec{0}) = 0,$ and  $ \nabla \tilde{b}(\vec{0}) = \vec{0}.$ Notice that we can write $ \tilde{b}(\vec{v}) = f(\vec{v}) - L(\vec{v}),$ where $f(\vec{v}) = b\left( \vec{v} +\frac{\vec{x}+\vec{x'}}{2}\right)$ and $L$ is the tangent hyperplane to $f$ at $\vec{v}=0.$ 
\end{remark}

\begin{remark}
Since $b$ is strictly convex, $\delta(\vec{x},\vec{x'}) > 0.$

\end{remark}

We obtain this bound by estimating the integral expression for the Szeg\H{o} kernel obtained in Proposition \ref{prop:integral szego1}. That is, we study 

\begin{equation*}
S((\vec{x},\vec{y},t);(\vec{x'},\vec{y'},t')) = 
{\scaleint{7ex}_{\bs 0}^{\infty}} e^{-2\pi\tau [b(\vec{x'}) + b(\vec{x}) + i(t'-t)]} 
\left({\scaleint{7ex}_{\bs \mathbb R^n}}  \frac{e^{2\pi \vec{\eta} \cdot [\vec{x}+\vec{x'} - i(\vec{y'}-\vec{y})]}   }{\displaystyle\int \limits_{\mathbb R^n} e^{4\pi [\vec\eta \cdot \vec{v}  - b(\vec{v})\tau]} \, d\vec{v}} \, d\vec{\eta} \,\right) d\tau.  
\end{equation*}

\noindent The proof is in essence an application of John ellipsoids.  Recall that by John \cite{Jo48}, given a symmetric convex compact region, there exists a maximal inscribed ellipsoid $\mathfrak{E}$ in that region (centered at the center of symmetry) such that $\sqrt{n}\,\mathfrak{E}$ contains the region, where $n$ is the dimension of the space and $\sqrt{n}\mathfrak{E}$ is the dilation of $\mathfrak{E}$ relative to its center of symmetry. The key step of our proof consists in introducing factors $\mu_1(\vec{x},\vec{x'},\tau), \,\ldots,\, \mu_n(\vec{x},\vec{x'},\tau)$ via a change of variable so that 

\begin{equation*}
\mu_1\cdots\mu_n \approx \left\vert \left\{ \vec{v}\,:\, \tilde{b}(\vec{v}) \le \frac{1}{\tau}  \right\}\right\vert.
\end{equation*}

\noindent These factors are chosen to be the length of the axes of the John ellipsoid associated to a symmetrization of the convex region $\left\{ \vec{v}\,:\, \tilde{b}(\vec{v}) \le \frac{1}{\tau}  \right\}.$ We explain this construction in the following subsection.

\bigskip 

\subsection{Construction of the factors $\mu_1\ldots,\mu_n$ }\label{mus}

 Let 

$$R = \left\{ \vec{v}\,:\, \tilde{b}(\vec{v}) \le \frac{1}{\tau}  \right\}.$$

\noindent Notice that since $b$ is convex, so is $\tilde{b},$ and the region R is convex. In order to be able to use John's bounds, we need to show that the set $R$ is also compact. We do so in the following claim. 

\bigskip

\begin{claim}\label{compacidad}
For any $M>0,$ the set $\{\vec{v}\,:\,{\tilde{b}}(\vec{v})\le M \}$ is compact. 

\end{claim}

\bigskip

\begin{proof} 

Notice that since $\tilde{b}$ is convex, $ \tilde{b}(\vec{0}) = 0,$ and $ \nabla \tilde{b}(\vec{0}) = \vec{0},$ then $\tilde{b}\ge 0$ (and $\tilde{b}$ is not the zero polynomial).

\bigskip 

Suppose towards a contradiction that the set $\{\vec{v}\,:\,{\tilde{b}}(\vec{v})\le M \}$ is unbounded. Then, by compactness of the unit ball, there exists a $\vec{x}\in\mathbb{R}^n$ such that $\forall c>0,$ $\tilde{b}(c\vec{x}) \le M.$ In particular, since $\tilde{b}(\vec{0}) = 0$ and $\tilde{b}$ is convex, $\forall c>0,$ $\tilde{b}(c\vec{x}) \equiv 0.$ For $\vec{x} = (x_1,\ldots,x_n)$ as above, define $\vec{w}(t) = t\cdot (x_1,\ldots,x_n).$ Then since $\tilde{b}(c\vec{x}) \equiv 0$ for all $c>0,$ it follows that $\tilde{b}(\vec{w}(t)) \equiv 0$.

\bigskip

On the other hand, since $\tilde{b}$ is of combined degree, its highest order terms corresponds to pure terms. Thus, the highest degree terms of $\tilde{b}(\vec{w}(t))$ are of the form $c_{\alpha}(tx_i)^{2m_i},$ for some $i$'s. In other words, the highest degree terms have coefficients of the form $c_{\alpha}x_i^{2m_i},$ for some $i$'s. The coefficients $c_{\alpha}$ of the highest degree pure terms of $\tilde{b}$ are positive, since $\tilde{b}(0,\ldots,0,tx_i,0,\ldots,0)$ takes positive values. Therefore, the sum of the coefficients $c_{\alpha}x_i^{2m_i}$ of the highest degree terms of $\tilde{b}\circ\vec{w}$ are also positive. Thus, $\tilde{b}\circ\vec{w}$ is not the zero polynomial. This yields the desired contradiction. It follows that  $\{\vec{v}\,:\,{\tilde{b}}(\vec{v})\le M \}$ is bounded, and therefore compact.

\end{proof}

\bigskip

 Recall that by construction $\tilde b(\vec{0}) = 0,$ so the region $R$ contains the origin. We would now like to show that there exists an ellipsoid $\mathfrak{E}$ centered at the origin such that 
 
 $$ \mathfrak{E} \subseteq R \subseteq C\mathfrak{E},$$
 
 \noindent for some independent positive constant $C.$ The existence of such an ellipsoid would follow immediately by John if the region were symmetric (with the origin as center of symmetry). However, we have made no symmetry assumptions on our domain. Nevertheless, we can show the following: 

\bigskip

\begin{claim}\label{cl:simetria}
Let $L$ be any line through the origin. This line $L$ will intersect $R$ in two points. Let $d_1$ be the shortest distance along $L$ from the origin to the boundary of $R,$ and let $d_2$ be the largest distance. Then there exist constants $m,$ $M$ depending only on the degree of the polynomial $b$ such that 

$$ 0 < m\le \frac{d_2}{d_1} \le M <+\infty.$$ 
\end{claim}

\bigskip

\begin{proof}
Let $h(\vec{v}) = \tau\tilde{b}(\vec{v}).$ Along the line $L,$ the polynomial $h(\vec{v})$ is a polynomial of one variable which we will call $h_L(t).$ This polynomial satisfies $h_L(0) = h_L'(0)=0.$ Write 

$$ h_L(t) =\sum_{j=2}^N c_j t^j.$$ 

\noindent Then there exists some $2\le k \le N,$ and $|c_k| \ne 0,$ such that 

$$ h_L(d_1) \le \sum_{j=2}^N |c_j| d_1^j \le (N-1) |c_k| d_1^k.$$

\noindent On the other hand, it follows from Lemma 2.1 on \cite{BrNaWa88} that there exists a constant $0< C_N \le 1$ such that

$$ h_L(d_2) \ge C_N \sum_{j=2}^N |c_j| d_2^j. $$ 

\noindent For $k$ as above, it follows that 

$$ C_N |c_k| d_2^k \le  C_N \sum_{j=2}^N |c_j| d_2^j \le h_L(d_2).$$

\noindent Moreover, $h_L(d_1) = h_L(d_2) =1,$ since $d_1$ and $d_2$ where chosen as the distances where $L$ intersects the boundary of the region $R = \{ \vec{v}\,:\, h(\vec{v})\le 1\}.$  Thus, 

$$ C_N |c_k| d_2^k \le h_L(d_2) =1= h_L(d_1) \le  (N-1) |c_k| d_1^k.$$

\noindent Therefore, 

$$ \frac{d_2}{d_1} \le \left(\frac{N-1}{C_N}\right)^{\frac{1}{k}}.$$

\noindent On the other hand, since $d_1$ is the shortest distance along $L$ from $R$ to the origin and $d_2$ is the largest, it follows that $d_1 \le d_2.$ Choosing $M = \max_{2\le k\le N}\left\{\left(\frac{N-1}{C_N}\right)^{\frac{1}{k}}\right\}$ and $ m =  1$ it follows that 

$$ m \le  \frac{d_2}{d_1} \le M.$$ 

\noindent This finishes the proof of Claim \ref{cl:simetria}. 

\end{proof}

\bigskip

We have shown that even though the region $R$ is not symmetric, the ratio between rays passing through the origin is bounded by universal constants that only depend on the degree of $b.$ In the following lemma we will show that this is enough to guarantee the existence of an ellipsoid centered at the origin contained in $R$ and such that a dilation by a universal constant contains $R.$ 

\bigskip

\begin{lemma}\label{lem:los mus}

Let $R = \left\{ \vec{v}\,:\, \tilde{b}(\vec{v}) \le \frac{1}{\tau} \right\}$ and $\widetilde{R} = \{ \vec{x} \,:\, -\vec{x} \in R\}.$ Let $\mathfrak{E}$ be the maximal inscribed ellipsoid in the region $R\cap\widetilde{R}.$ Then 

$$ \mathfrak{E} \subseteq R \subseteq M\sqrt{n}\mathfrak{E}$$

\noindent where $M$ is as in Claim \ref{cl:simetria}, and $n$ is the dimension of the space.

\end{lemma}

\bigskip

\begin{proof} By definition, the set $R\cap \widetilde{R}$ is symmetric about the origin. Moreover, since $R$ and $\widetilde{R}$ are compact and convex, their intersection is also compact and convex. It follows from John that there exists an ellipsoid $\mathfrak{E}$ centered at the origin such that 

$$ \mathfrak{E }\subseteq R\cap\widetilde{R} \subseteq \sqrt{n}\mathfrak{E}.$$ 

\bigskip

It is clear that $\mathfrak{E} \subseteq R.$ We would like to show that there is a dilation of $\mathfrak{E}$ which contains $R.$ Let $\vec{x}$ be any point in $R.$ Then, if $-\vec{x}\in R,$ it follows by definition that $x\in\sqrt{n}\mathfrak{E}.$ Now suppose that $-\vec{x} \notin R.$ Let $L$ be the line that goes through the origin and $\vec{x}.$ Using the notation of the previous claim, we have that $|\vec{x}| \le d_2.$ Given $M$ as in Claim \ref{cl:simetria}, let $\rho =\frac{1}{M}.$ Then $|-\rho \vec{x}|\le \rho d_2 \le \rho M d_1 = d_1.$ But since $d_1$ is the  minimum distance from the boundary of $R$ to the origin along line $L,$ it follows that $ -\rho\vec{x}\in R.$ Thus, given any point $\vec{x}\in R$ the point $-\frac{1}{M}\vec{x}$ is also contained in $R$. It follows that

$$\mathfrak{E} \subseteq R \subseteq M\sqrt{n}\mathfrak{E}.$$

\noindent This finishes the proof of Lemma \ref{lem:los mus}.

\end{proof}

\bigskip

 It follows from the previous lemma that

\begin{equation}\label{eq:elmu1}
 Vol(\mathfrak{E}) \approx   \left\vert\left\{ \vec{v}\,:\, \tilde{b}(\vec{v}) \le \frac{1}{\tau} \right\}\right\vert.
 \end{equation}

\noindent  Let $\mu_1,\ldots, \mu_n$ be the lengths of the semi-axes of $\sqrt{n}M\mathfrak{E},$ indexed so that $\mu_1 \ge \mu_2 \ge \ldots \ge \mu_n.$ Then,

\begin{equation}\label{eq:eleccion mus}
\mu_1\cdots\mu_n \approx  \left\vert\left\{ \vec{v}\,:\, \tilde{b}(\vec{v}) \le \frac{1}{\tau} \right\}\right\vert.
\end{equation}

\noindent In equations (\ref{eq:elmu1}) and (\ref{eq:eleccion mus}), the constant depends only on the combined degree of $b$ and the dimension of the space.

\bigskip

\subsection{Proof of the Main Theorem}

We are now ready to present the proof of the Main Theorem. For the reader's convenience, we have divided the proof into three subsections, corresponding to a bound in terms of $\delta,$ a bound in terms of $\widetilde{b}(\vec{y}-\vec{y'})$ and a bound in terms of $w$. We finish by combining all three bounds to obtain the estimate stated in the Main Theorem. It will be convenient in the course of the proof of all three bounds to rearrange the terms of the integral expression for the Szeg\H{o} kernel obtained in Theorem \ref {prop:integral szego1} as follows: 

\bigskip

Making the change of variables $\vec{v} \rightarrow \vec{v}+\frac{\vec{x}}{2}+\frac{\vec{x'}}{2}$  to get rid of the term $ e^{2\pi\vec{\eta}\cdot (\vec{x}+\vec{x'})}$ in the original expression obtained for the Szeg\H{o} kernel given by

\begin{equation*}
S((\vec{x},\vec{y},t);(\vec{x'},\vec{y'},t')) = 
{\scaleint{7ex}_{\bs 0}^{\infty}} e^{-2\pi\tau [b(\vec{x'}) + b(\vec{x}) + i(t'-t)]} 
\left({\scaleint{7ex}_{\bs \mathbb R^n}} \frac{e^{2\pi \vec{\eta} \cdot [\vec{x}+\vec{x'} - i(\vec{y'}-\vec{y})] }   }{\displaystyle\int_{\mathbb R^n} e^{4\pi [\vec\eta \cdot \vec{v}  - b(\vec{v})\tau]} \, d\vec{v}} \, d\vec{\eta} \,\right) d\tau
\end{equation*}

\noindent it follows that

\begin{equation*}
S = {\scaleint{7ex}_{\bs 0}^{\infty}} e^{-2\pi\tau [b(\vec{x'}) + b(\vec{x}) + i(t'-t)]} {\scaleint{7ex}_{\bs \mathbb R^n}} \frac{e^{2\pi i \vec{\eta} \cdot (\vec{y}-\vec{y'})}}{\displaystyle\int_{\mathbb R^n} e^{4\pi \left[\vec{\eta}\cdot \vec{v} - \tau b\left(\vec{v}+\frac{\vec{x}+\vec{x'}}{2}\right)\right]} \, d\vec{v}} \, d\vec{\eta}\, d\tau. 
\end{equation*}

We will modify the denominator integral so as to change it into an integral of the form $[\theta(\vec{\eta})]^{-1},$ where $\theta$ is the function studied in Lemma \ref{lem:Schwartz2general m}. In particular, the exponent of the denominator integral must be of the form $\vec{\eta}\cdot\vec{v} - g(\vec{v}),$ where $g(\vec{0}) = 0$ and $\nabla g(\vec{0}) = 0.$ We make the change of variables $\vec{\eta} \rightarrow \vec{\eta} + \nabla b\left(  \frac{\vec{x}+\vec{x'}}{2}\right) \tau$ so that 

\begin{equation*}\begin{split}
S = &\,
{\scaleint{7ex}_{\bs 0}^{\infty}} e^{-2\pi\tau [b(\vec{x'}) + b(\vec{x}) + i(t'-t)]}  e^{4\pi \tau \, b\left(\frac{\vec{x}+\vec{x'}}{2}\right) }  e^{2\pi i \tau \nabla b \left(  \frac{\vec{x}+\vec{x'}}{2}\right) \cdot (\vec{y}-\vec{y'}) } {\scaleint{7ex}_{\bs \mathbb R^n}} \frac{e^{2\pi i\vec{\eta}\cdot (\vec{y}-\vec{y'}) }}{\displaystyle\int_{\mathbb R^n} e^{4\pi [\vec{\eta}\cdot \vec{v} -\tau\tilde{b}(\vec{v})]} \, d\vec{v}} \, d\vec{\eta}\, d\tau,\\
\end{split}\end{equation*}

\noindent where $\tilde{b}(\vec{v})$ is as in equation (\ref{eq:b tilde}) and the term $b\left(\frac{\vec{x}+\vec{x'}}{2}\right)$ has been added so that $\tilde{b}(\vec{0}) = 0.$ With $\delta(\vec{x},\vec{x'})$ and $w$ be as in equations (\ref{eq:delta}) and (\ref{eq:w}), it follows that

\begin{equation}\label{eq:con im}
S = {\scaleint{7ex}_{\bs 0}^{\infty}} e^{-2\pi\tau \delta}  e^{-2\pi \tau i w}  {\scaleint{7ex}_{\bs \mathbb R^n}} \frac{e^{2\pi i\vec{\eta}\cdot (\vec{y}-\vec{y'}) }}{\displaystyle\int_{\mathbb R^n} e^{4\pi [\vec{\eta}\cdot \vec{v} -\tau\tilde{b}(\vec{v})]} \, d\vec{v}} \, d\vec{\eta}\, d\tau.
\end{equation}

\bigskip

Notice that since $b$ is strictly convex, $\tau\tilde{b}$ is also strictly convex. Moreover, if $b$ is of combined degree $(m_1,\ldots,m_n),$ so is $\tau\tilde{b}.$ Thus, $\tau\tilde{b}$ is a strictly convex polynomial satisfying conditions (\ref{itm:gorigen}), (\ref{itm:ggrad}) and (\ref{itm:grado}) of Lemma \ref{lem:Schwartz2general m}. It remains to renormalize $\tau\tilde{b}$ so that it also satisfies condition (\ref{itm:John}). 

\bigskip

With $\mu_1,\ldots,\mu_n$ chosen as in equation (\ref{eq:eleccion mus}), let 

\begin{equation}\label{eq:eleccion de g}
 g(\vec{v}) = \tau\tilde{b}(\vec{\mu}\vec{v}).
 \end{equation}

\noindent Here, $\vec{\mu}\vec{v} = [\mu_1v_1\,\, \mu_2v_2\,\,\ldots\,\, \mu_nv_n]^T.$  Notice that since $ \mathfrak{E} \subseteq \left\{ \vec{v}\,:\, \tau\tilde{b}(\vec{v}) \le 1\right\} \subseteq \sqrt{n}M\mathfrak{E}$ and letting $A= (\sqrt{n}M)^{-1}$ condition (\ref{itm:John}) of Lemma \ref{lem:Schwartz2general m} is satisfied.

\bigskip

 By making the change of variables $\vec{v}\rightarrow \vec{\mu}\vec{v}$ so as to introduce the factors  $\mu_1,\ldots,\mu_n$ in the denominator integral of equation (\ref{eq:con im}), as well as  the change of variables $\vec{\eta} \rightarrow \frac{\vec{\eta}}{\vec{\mu}}$ we have that

 \begin{equation}\label{eq:con todo}
S  ={\scaleint{7ex}_{\bs 0}^{\infty}} \frac{e^{-2\pi\tau\delta} e^{-2\pi\tau iw} }{\mu_1^2\cdots\mu_n^2} {\scaleint{7ex}_{\bs \mathbb R^n}} \frac{e^{2\pi i\vec{\eta}\cdot \left(\frac{\vec{y}-\vec{y'}}{\vec{\mu}}\right) }}{\displaystyle\int_{\mathbb R^n} e^{4\pi [\vec{\eta}\cdot \vec{v} -\tau\tilde{b}(\vec{\mu}\vec{v})]} \, d\vec{v}} \, d\vec{\eta}\, d\tau.
\end{equation}

\bigskip

\subsubsection{The bound in terms of $\delta$}\label{delta}

\begin{proposition}\label{thm: S original 1}

 Let $(\vec{x},\vec{y},t)$ and $(\vec{x'},\vec{y'},t')$ be any two points in $\partial\Omega_b.$ Then, 

\begin{equation*}
\left\vert S\left((\vec{x},\vec{y},t);(\vec{x'},\vec{y'},t')\right)\right\vert \lesssim \frac{1}{\delta | \{ \vec{v} \,:\, \tilde{b}(\vec{v}) <\delta \} |^2},
\end{equation*}

\noindent where the constant may depend on the combined degree of $b$ and the dimension of the space, but is independent of the two given points. 

\end{proposition}

\bigskip

\begin{proof}\label{demo}

It follows from equation (\ref{eq:con todo}) that

\begin{equation*}
|S| \le {\scaleint{7ex}_{\bs 0}^{\infty}} \frac{e^{-2\pi\tau \delta }}{\mu_1^2\cdots \mu_n^2}   
{\scaleint{7ex}_{\bs \mathbb R^n}} \frac{1 }{\displaystyle\int_{\mathbb R^n} e^{4\pi [\vec{\eta}\cdot \vec{v} -\tau\tilde{b}(\mu\vec{v})]} \, d\vec{v}} \, d\vec{\eta}\,d\tau.
\end{equation*}

\noindent But since $\mu_1\cdots\mu_n \approx  \left\vert\left\{ \vec{v}\,:\, \tilde{b}(\vec{v}) \le \frac{1}{\tau} \right\}\right\vert,$ it follows that

\begin{equation}\label{eq:con im 2}
|S| \lesssim {\scaleint{7ex}_{\bs 0}^{\infty}} \frac{e^{-2\pi\tau \delta }}{ \left\vert\left\{ \vec{v}\,:\, \tilde{b}(\vec{v}) \le \frac{1}{\tau}  \right\}\right\vert^2  }
{\scaleint{7ex}_{\bs \mathbb R^n}} \frac{1 }{\displaystyle\int \limits_{\mathbb R^n} e^{4\pi [\vec{\eta}\cdot \vec{v} -\tau\tilde{b}(\mu\vec{v}) ]} \, d\vec{v}} \, d\vec{\eta}\,d\tau.
\end{equation}

\noindent By Lemma \ref{lem:Schwartz2general m} the reciprocal of the denominator integral is Schwartz, and the decay is independent of the coefficients of $g.$ In particular, the decay does not depend on $\tau,$ and $\int_{\mathbb{R}^n}  \theta(\vec{\eta}) \, d\vec{\eta}$ converges. Hence,  

\begin{equation*}
|S|\lesssim {\scaleint{7ex}_{\bs 0}^{\infty}} \frac{e^{-2\pi\tau \delta }}{ \left\vert\left\{ \vec{v}\,:\, \tilde{b}(\vec{v}) \le \frac{1}{\tau}  \right\}\right\vert^2  } \, d\tau.
\end{equation*}

\noindent We can write 

\begin{equation*}
{\scaleint{7ex}_{\bs 0}^{\infty}} \frac{e^{-2\pi\tau \delta }}{ \left\vert\left\{ \vec{v}\,:\, \tilde{b}(\vec{v}) \le \frac{1}{\tau} \right\}\right\vert^2  } \, d\tau = {\sum_{j=-\infty}^{\infty}}\,  {\scaleint{6ex}_{\bs \tau\delta=2^j}^{\tau\delta=2^{j+1}}}  \frac{e^{-2\pi\tau\delta}}{ \left\vert\left\{ \vec{v}\,:\, \tilde{b}(\vec{v}) \le \frac{1}{\tau}  \right\}\right\vert^2 }\,d\tau .
\end{equation*}

\noindent Thus, 

\begin{equation*}
|S| \lesssim \sum_{j=-\infty}^{\infty} \frac{e^{-2\pi2^j}}{ \left\vert\left\{ \vec{v}\,:\, \tilde{b}(\vec{v}) \le \frac{\delta}{2^{j+1}}  \right\}\right\vert^2 } \scaleint{6ex}_{\bs \tau = 2^j\delta^{-1}}^{\tau = 2^{j+1}\delta^{-1}}\, d\tau = \sum_{j=-\infty}^{\infty} \frac{2^je^{-2\pi2^j}}{\delta \left\vert\left\{ \vec{v}\,:\, \tilde{b}(\vec{v}) \le \frac{\delta}{2^{j+1}}  \right\}\right\vert^2 }.
\end{equation*}

\noindent In order to get rid of the dependence on $j$ of $\left\vert\left\{ \vec{v}\,:\, \tilde{b}(\vec{v}) \le \frac{\delta}{2^{j+1}}  \right\}\right\vert$ we can write 

\begin{equation}\label{eq:dividida}
|S| \lesssim \sum_{-\infty < j<0} \frac{e^{-2\pi2^j}2^j}{\delta|\{\vec{v}\,:\,\tilde{b}(\vec{v}) \le \delta  \}|^2} +\sum_{0\le j<\infty} \frac{e^{-2\pi2^j}2^j}{\delta\left\vert\left\{\vec{v}\,:\,\tilde{b}(\vec{v}) \le \frac{\delta}{2^{j+1}}  \right\}\right\vert^2 }.
\end{equation}

\noindent But for $j\ge 0$ we have that (see Claim \ref{cl:B-M} in the Appendix)

\begin{equation*}
|\{\vec{v}\,:\, \tilde{b}(\vec{v}) \le \delta  \}| \le 2^{n(j+1)} \left\vert\left\{\vec{v}\,:\,\tilde{b}(\vec{v}) \le \frac{\delta}{2^{j+1}} \right\}\right\vert.
\end{equation*}

\noindent It follows that 

\begin{equation*}
|S| \lesssim  \frac{1}{\delta|\{\vec{v}\,:\,\tilde{b}(\vec{v}) \le \delta  \}|^2} \left( \sum_{-\infty < j<0} \,e^{-2^{j+1}\pi}2^j+\sum_{0\le j<\infty}\, e^{-2^{j+1}\pi}2^{2n(j+1)+j}\right).
\end{equation*}

\noindent Since both sums converge, we obtain the desired estimate. 

\end{proof}

\subsubsection{The bound in terms of  $\widetilde{b}(\vec{y}-\vec{y'})$}\label{segundo}

\begin{proposition}\label{thm: S original con y}

 Let $(\vec{x},\vec{y},t)$ and $(\vec{x'},\vec{y'},t')$ be any two points in $\partial\Omega_b.$ Then

\begin{equation*}
\left\vert S\left((\vec{x},\vec{y},t);(\vec{x'},\vec{y'},t')\right)\right\vert \lesssim \frac{1}{ \widetilde{b}(\vec{y}-\vec{y'})| \{ \vec{v} \,:\, \tilde{b}(\vec{v}) < \widetilde{b}(\vec{y}-\vec{y'})\} |^2},
\end{equation*}

\noindent where the constant may depend on the combined degree of $b$ and the dimension of the space, but is independent of the two given points. 

\end{proposition}

\begin{proof}

We had shown in equation (\ref{eq:con todo}) on page \pageref{eq:con todo} that

\begin{equation*}
S  = {\scaleint{7ex}_{\bs 0}^{\infty}} \frac{e^{-2\pi\tau\delta} e^{-2\pi\tau iw } }{\mu_1^2\cdots\mu_n^2} {\scaleint{7ex}_{\bs \mathbb R^n}} \frac{e^{2\pi i\vec{\eta}\cdot \left(\frac{\vec{y}-\vec{y'}}{\vec{\mu}}\right) }}{\displaystyle\int \limits_{\mathbb R^n} e^{4\pi [\vec{\eta}\cdot \vec{v} -\tau\tilde{b}(\vec{\mu}\vec{v})]} \, d\vec{v}} \, d\vec{\eta}\, d\tau.
\end{equation*}

\noindent Thus, and since $\delta > 0,$ 

\begin{equation*}
|S|  \le {\scaleint{7ex}_{\bs 0}^{\infty}} \frac{1 }{\mu_1^2\cdots\mu_n^2}\left\vert\,\, {\scaleint{7ex}_{ \bs \mathbb R^n}} \frac{e^{2\pi i\vec{\eta}\cdot \left(\frac{\vec{y}-\vec{y'}}{\vec{\mu}}\right) }}{\displaystyle\int_{\mathbb R^n} e^{4\pi [\vec{\eta}\cdot \vec{v} -\tau\tilde{b}(\vec{\mu}\vec{v})]} \, d\vec{v}} \, d\vec{\eta}\,\right\vert \,\, d\tau.
\end{equation*}

\noindent But by Lemma \ref{lem:Schwartz2general m}, the reciprocal of the denominator integral is Schwartz. Moreover, its decay is independent of  $\tau$ and the coefficients of $b.$ The same is true of its Fourier transform, $\hat{\theta}.$  We can write

\begin{equation*}
|S| \le \scaleint{6ex}_{\bs 0}^{\infty} \frac{1}{\mu_1^2\cdots \mu_n^2} \left\vert \hat{\theta}\left( \frac{\vec{y}-\vec{y'}}{\vec{\mu}}\right) \right\vert \, d\tau.
\end{equation*}

\noindent And as in the previous bound, we can write 

\begin{equation*}
|S| \lesssim \scaleint{7ex}_{\bs 0}^{\infty} \frac{1}{ \left|\left\{ \vec{v}\,:\, \widetilde{b}(\vec{v}) \le \frac{1}{\tau} \right\}\right|^2} \left\vert \hat{\theta}\left( \frac{\vec{y}-\vec{y'}}{\vec{\mu}}\right) \right\vert \, d\tau.
\end{equation*}

\noindent Let $\vec{\gamma} = \vec{y}- \vec{y'}.$ We can split the interval of integration into dyadic intervals in the following way: 

\begin{equation*}
\scaleint{7ex}_{\bs 0}^{\infty} \frac{1}{ \left|\left\{ \vec{v}\,:\, \widetilde{b}(\vec{v}) \le \frac{1}{\tau} \right\}\right|^2} \left\vert \hat{\theta}\left( \frac{\vec{y}-\vec{y'}}{\vec{\mu}}\right) \right\vert \, d\tau = \sum_{j=-\infty}^{\infty} {\scaleint{6ex}_{\bs \, 2^j (\widetilde{b}(\vec{\gamma}))^{-1} }^{\, 2^{j+1}(\widetilde{b}(\vec{\gamma}))^{-1}}} \frac{1}{ \left|\left\{ \vec{v}\,:\, \widetilde{b}(\vec{v}) \le \frac{1}{\tau} \right\}\right|^2} \left\vert \hat{\theta}\left( \frac{\vec{\gamma}}{\vec{\mu}}\right) \right\vert \, d\tau.
\end{equation*}

\noindent But since in each interval $\dfrac{\widetilde{b}(\vec{\gamma})}{2^{j+1}} \le \dfrac{1}{\tau},$ It follows that

\begin{equation*}\begin{split}
|S|\,  \lesssim&\, \sum_{-\infty < j <0} \scaleint{6ex}_{\bs \, 2^j (\widetilde{b}(\vec{\gamma}))^{-1}}^{\, 2^{j+1}(\widetilde{b}(\vec{\gamma}))^{-1}}  \frac{1}{ \left|\left\{ \vec{v}\,:\, \widetilde{b}(\vec{v}) \le  \frac{\widetilde{b}(\vec{\gamma})}{2^{j+1}}  \right\}\right|^2} \left\vert \hat{\theta}\left( \frac{\vec{\gamma}}{\vec{\mu}}\right) \right\vert \, d\tau\\ &+\, \sum_{0\le j <\infty} \scaleint{6ex}_{\bs \,2^j(\widetilde{b}(\vec{\gamma}))^{-1} }^{\,2^{j+1}(\widetilde{b}(\vec{\gamma}))^{-1}}  \frac{1}{ \left|\left\{ \vec{v}\,:\, \widetilde{b}(\vec{v}) \le  \frac{\widetilde{b}(\vec{\gamma})}{2^{j+1}}  \right\}\right|^2} \left\vert \hat{\theta}\left( \frac{\vec{\gamma}}{\vec{\mu}}\right) \right\vert \, d\tau.\\ 
\end{split}\end{equation*}

\noindent As in the previous bound, for $j<0$ we have that $ \left\vert \left\{ \vec{v}\,:\, \widetilde{b}(\vec{v}) \le \dfrac{\widetilde{b}(\vec{\gamma})}{2^{j+1}} \right\} \right\vert^{-1}  \le \left\vert \left\{ \vec{v}\,:\, \widetilde{b}(\vec{v}) \le \widetilde{b}(\vec{\gamma}) \right\} \right\vert^{-1},$ and for $j\ge 0,$ 

\begin{equation*}
 \left\vert \left\{ \vec{v}\,:\, \widetilde{b}(\vec{v}) \le \frac{\widetilde{b}(\vec{\gamma})}{2^{j+1}} \right\} \right\vert^{-1}  \le 2^{n(j+1)} \left\vert \left\{ \vec{v}\,:\, \widetilde{b}(\vec{v}) \le \widetilde{b}(\vec{\gamma}) \right\} \right\vert^{-1}.
\end{equation*}

\noindent Hence, 

\begin{equation*}\begin{split}
|S|&\,  \lesssim \sum_{-\infty < j <0}   \frac{1}{ \left|\left\{ \vec{v}\,:\, \widetilde{b}(\vec{v}) \le  \widetilde{b}(\vec{\gamma})  \right\}\right|^2} \scaleint{6ex}_{\bs \, 2^j (\widetilde{b}(\vec{\gamma}))^{-1}  }^{\, 2^{j+1}(\widetilde{b}(\vec{\gamma}))^{-1}} \left\vert \hat{\theta}\left( \frac{\vec{\gamma}}{\vec{\mu}}\right) \right\vert \, d\tau\\
 &\,+ \sum_{0\le j <\infty}   \frac{2^{2n(j+1)}}{ \left|\left\{ \vec{v}\,:\, \widetilde{b}(\vec{v}) \le  \widetilde{b}(\vec{\gamma})  \right\}\right|^2} \scaleint{6ex}_{\bs \, 2^j (\widetilde{b}(\vec{\gamma}))^{-1} }^{\, 2^{j+1}(\widetilde{b}(\vec{\gamma}))^{-1}} \left\vert \hat{\theta}\left( \frac{\vec{\gamma}}{\vec{\mu}}\right) \right\vert \, d\tau.\\ 
\end{split}\end{equation*}

\noindent For the first sum it suffices to bound $|\hat{\theta}|$ by a universal constant. It follows that

\begin{equation*}\begin{split}
& \sum_{-\infty < j <0}   \frac{1}{ \left|\left\{ \vec{v}\,:\, \widetilde{b}(\vec{v}) \le  \widetilde{b}(\vec{\gamma})  \right\}\right|^2} {\scaleint{6ex}_{\bs \, 2^j (\widetilde{b}(\vec{\gamma}))^{-1}}^{\, 2^{j+1}(\widetilde{b}(\vec{\gamma}))^{-1}} \left\vert \hat{\theta}\left( \frac{\vec{\gamma}}{\vec{\mu}}\right) \right\vert \, d\tau }\\
&\, \lesssim  \sum_{-\infty \le j <0}   \frac{1}{ \left|\left\{ \vec{v}\,:\, \widetilde{b}(\vec{v}) \le  \widetilde{b}(\vec{\gamma})  \right\}\right|^2} {\scaleint{6ex}_{\bs \, 2^j (\widetilde{b}(\vec{\gamma}))^{-1}}^{\, 2^{j+1}(\widetilde{b}(\vec{\gamma}))^{-1}}  \, d\tau}\\
  &= \, \sum_{-\infty \le j<0}   \frac{2^j}{\widetilde{b}(\vec{\gamma}) \left|\left\{ \vec{v}\,:\, \widetilde{b}(\vec{v}) \le  \widetilde{b}(\vec{\gamma})  \right\}\right|^2} \approx   \frac{1}{\widetilde{b}(\vec{\gamma}) \left|\left\{ \vec{v}\,:\, \widetilde{b}(\vec{v}) \le  \widetilde{b}(\vec{\gamma})  \right\}\right|^2}.\\
\end{split} \end{equation*}

\noindent We would like to obtain a similar bound for the second sum, with $j\ge 0.$ The main obstacle is obtaining decay in $j$ to counteract the growth of the term $2^{2n(j+1)},$ thus ensuring the convergence of the series. We will show that

 \begin{equation*}
  \scaleint{7ex}_{\bs \, 2^j (\widetilde{b}(\vec{\gamma}))^{-1} }^{\, 2^{j+1}(\widetilde{b}(\vec{\gamma}))^{-1}} \left\vert \hat{\theta}\left( \frac{\vec{\gamma}}{\vec{\mu}}\right) \right\vert \, d\tau \le \frac{C_k}{2^{Kj}\widetilde{b}(\vec{\gamma})},
 \end{equation*}
 
 \noindent for any positive constant $K,$ and a positive constant $C_k$ which depends only on $K,$ on the combined degree of $b,$ and on the dimension of the space.  
 
 \bigskip
 
  Since $\hat{\theta}$ is Schwartz,  
 $$  \left\vert \hat{\theta}\left( \frac{\vec{\gamma}}{\vec{\mu}}\right) \right\vert  \le \frac{D}{\left\vert C+ \left\vert p\left(\frac{\vec{\gamma}}{\vec{\mu}} \right)\right\vert \right\vert^N}$$
 \noindent  for any polynomial $p:\mathbb{R}^n\rightarrow\mathbb{R}$ and positive constants $C$ and $N.$ Here, the constant $D>0$ depends on $C,$ $N,$ and $p.$ In this interval, $2^j \le  \tau\widetilde{b}(\vec{\gamma})$ so it suffices to show that there exists some polynomial $p$ and constant $C$ such that $ \tau\widetilde{b}(\vec{\gamma}) \le C+  \left\vert p\left(\frac{\vec{\gamma}}{\vec{\mu}} \right)\right\vert.$ In fact, it would follow that for any $N>0,$

  $$  \left\vert \hat{\theta}\left( \frac{\vec{\gamma}}{\vec{\mu}}\right) \right\vert  \le \frac{D}{\left\vert  \tau\widetilde{b}(\vec{\gamma})    \right\vert^N} \le \frac{D}{2^{Nj} },$$
  
  \noindent so that 
  
  \begin{equation*}\begin{split}
  &\, \sum_{j=0}^{\infty}   \frac{2^{2n(j+1)}}{ \left|\left\{ \vec{v}\,:\, \widetilde{b}(\vec{v}) \le  \widetilde{b}(\vec{\gamma})  \right\}\right|^2} \scaleint{6ex}_{\bs \, 2^j (\widetilde{b}(\vec{\gamma}))^{-1} }^{\, 2^{j+1}(\widetilde{b}(\vec{\gamma}))^{-1}} \left\vert \hat{\theta}\left( \frac{\vec{\gamma}}{\vec{\mu}}\right) \right\vert \, d\tau   \le    \sum_{j=0}^{\infty}   \frac{D 2^j2^{2n(j+1)}}{2^{Nj} \widetilde{b}(\vec{\gamma}) \left|\left\{ \vec{v}\,:\, \widetilde{b}(\vec{v}) \le  \widetilde{b}(\vec{\gamma})  \right\}\right|^2}. 
  \end{split}\end{equation*}
  
  \noindent For sufficiently large $N,$ the series converges, and we would obtain the desired estimate. 
  
  \bigskip
  
  In order to find a polynomial $p$ and positive constant $C$ such that $ \tau\widetilde{b}(\vec{\gamma}) \le C+  \left\vert p\left(\frac{\vec{\gamma}}{\vec{\mu}} \right)\right\vert,$ let $\vec{s} = \frac{\vec{\gamma}}{\vec{\mu}}$ and write the above requirement as $ \tau\widetilde{b}(\vec{\mu}\vec{s}) \le C + |p\left(\vec{s} \right)|.$ By construction of the factors $\mu_1,\ldots,\mu_n$ the polynomial $\tau\widetilde{b}(\vec{\mu}\vec{s})$ satisfies all the hypothesis of Lemma \ref{lem:Schwartz2general m}. In particular, Claim \ref{cl:bound g sin coefs en 2 dim} on page \pageref{cl:bound g sin coefs en 2 dim} holds. Thus, there exists a constant $C,$ which depends only on the combined degree of $b$ and on the dimension of the space, such that

$$ \tau\widetilde{b}(\vec{\mu}\vec{s}) \le C(1 + s_1^{2m_1}+\ldots +s_n^{2m_n}).$$

\noindent This finishes the proof of Proposition \ref{thm: S original con y}.

\end{proof}

\subsubsection{The bound in terms of w}\label{tercer}

\begin{proposition}\label{thm: S original con t}
Let $(\vec{x},\vec{y},t)$ and $(\vec{x'},\vec{y'},t')$ be any two points in $\partial\Omega_b.$ Then, 

\begin{equation*}
\left\vert S\left((\vec{x},\vec{y},t);(\vec{x'},\vec{y'},t')\right)\right\vert \lesssim  \frac{1}{ |w| \,\,| \{ \vec{v} \,:\, \widetilde{b}(\vec{v}) <|w| \} |^2},
\end{equation*}

\noindent where the constant may depend on the combined degree of $b$ and the dimension of the space, but is independent of the two given points.

\end{proposition}

The derivation of this last bound is rather long and technical. Before giving all the technical details, however, we shall begin by briefly outlining the main ideas behind the proof. It follows from equation (\ref{eq:con im}) on page \pageref{eq:con im} that

 \begin{equation}\label{eq:original en t}
S = \int_0^{\frac{\pi}{|u|}}e^{- iu\tau}  F(\tau)\, d\tau + \int_{\frac{\pi}{|u|}}^{\infty}e^{- iu\tau}  F(\tau)\, d\tau.
\end{equation}

\noindent where for convenience we have set $u =2\pi w,$ and 

\begin{equation}
F(\tau) =   e^{-2\pi\tau\delta} {\scaleint{7ex}_{\bs \mathbb R^n}} \frac{e^{2\pi i\vec{\eta}\cdot (\vec{y}-\vec{y'}) }}{\displaystyle\int \limits_{\mathbb R^n} e^{4\pi [\vec{\eta}\cdot \vec{v} -\tau\tilde{b}(\vec{v})]} \, d\vec{v}} \, d\vec{\eta}. 
\end{equation}

\noindent The integral for $0\le \tau \le {\frac{\pi}{|u|}}$ in equation (\ref{eq:original en t}) yields the desired estimate by using similar techniques as those detailed in the proof of the previous two bounds. Thus, the main difficulty lies in estimating the integral for ${\frac{\pi}{|u|}}\le \tau <\infty.$ In particular, we must show that the integral converges. To do so, we will take advantage of the oscillation of the term $e^{- iu\tau}.$ Integrating the latter by parts $N$ times, for an arbitrary positive integer $N,$ we obtain formally an equation of the form 

$$ \scaleint{7ex}_{\bs \frac{\pi}{|u|}}^{\infty}\frac{e^{- iu\tau}}{|u|^N}  F^{(N)}(\tau)\, d\tau. $$

\noindent We then show that after introducing the factors $\vec{\mu}$ as in the two previous bounds, every derivative of $F(\tau)$ yields a factor of $\frac{1}{\tau}$ times a bounded function, so that 

\begin{equation*}\begin{split}
\left\vert \scaleint{6ex}_{\bs \frac{\pi}{|u|}}^{\infty}\frac{e^{- iu\tau}}{|u|^N}  F^{(N)}(\tau)\, d\tau\right\vert & \, \approx \frac{1}{|u|^N}\scaleint{6ex}_{\bs \frac{\pi}{|u|}}^{\infty}\frac{1}{\mu_1^2\cdots\mu_n^2}\cdot \frac{1}{\tau^N}\,d\tau \approx \frac{1}{|u|^N}\scaleint{6ex}_{\bs \frac{\pi}{|u|}}^{\infty}\frac{1}{\left|\left\{ \vec{v}\,:\, \widetilde{b}(\vec{v}) \le \frac{1}{\tau} \right\}\right|^2}\cdot \frac{1}{\tau^N}\,d\tau.\\
\end{split}\end{equation*}

\noindent Finally, using Claim \ref{cl:B-M} (see Appendix), we show that this last integral is bounded by an expression of the form

$$\frac{1}{|u|^N \left|\left\{ \vec{v}\,:\, \widetilde{b}(\vec{v}) \le |u| \right\}\right|^2 }\scaleint{6ex}_{\bs \frac{\pi}{|u|}}^{\infty}\frac{|u|^{2n}\tau^{2n}}{\tau^N}\,d\tau,$$

\noindent yielding the desired estimate for large enough values of $N.$  

\bigskip

Before presenting a rigorous proof of Proposition \ref{thm: S original con t}, we discuss three technical results that will be used in the course of the proof. In Claim \ref{cl:truco de wiener} we obtain an upper bound for $\int_0^{\infty} e^{-iu\tau}F(\tau)\,d\tau$ in terms of the $(N+1)^{st}$ derivative of $F.$ The method we use is analogous to integration by parts, but does not yield boundary terms, making the computation slightly simpler (see, e.g., Proposition  $X_{19}$ on p. 14 \cite{Wi58}). In Claim \ref{cl:deriv de F} we compute the $N^{th}$ derivative of $F.$ In Claim \ref{bound de Delta} we show that, after introducing the factors $\vec{\mu},$ the $N^{th}$ derivative of $F$ is dominated by $\frac{1}{\tau^N}$ times a bounded function. 

\begin{claim}\label{cl:truco de wiener}
 Let $t\in\mathbb{R}$ and 
 
 $$I(t) = \int_0^{\infty} e^{-it\tau}F(\tau)\,d\tau,$$
 
 \noindent where $F \in C^{\infty}(\mathbb{R})$ and $F\in L^1(\mathbb{R}).$ Then given $N \in \mathbb{N}$ there exist positive coefficients $c_1, \ldots,c_{N+1},$ such that
 
 \begin{equation}\label{eq:truco de Wiener}\begin{split}
 |I(t)| \le &\, \sum_{j=0}^{N+1} c_j \left\vert \int_0^{\frac{\pi}{|t|}} e^{-it\tau} F\left( \tau +\frac{j\pi}{|t|}\right) \, d\tau \right\vert \\ &+ \frac{1}{2^{N+1}}\left\vert \int_\frac{\pi}{|t|}^{\infty}e^{-it\tau} \int_0^{\frac{\pi}{|t|}}\cdots\int_0^{\frac{\pi}{|t|}} F^{(N+1)}(\tau+s_1+\ldots + s_{N+1}) \,ds_1\cdots ds_{N+1} \,d\tau \right\vert. \\
 \end{split}\end{equation}

\end{claim}

\begin{proof}
We can write   

\begin{equation}\label{eq:S y L}
I(t)  = \int_0^{\frac{\pi}{|t|}} e^{-it\tau} F(\tau) \, d\tau + \int_{\frac{\pi}{|t|}}^{\infty} e^{-it\tau}F(\tau)\,d\tau = S+L.
\end{equation}

\noindent Introducing a factor of $e^{i\,\sgn{(t)}\pi},$ we can split $L$ as follows: 

\begin{equation*}\begin{split}
L &\, = \frac{1}{2} \left[    \int_{\frac{\pi}{|t|}}^{\infty} e^{-it\tau}F(\tau)\,d\tau    -  \int_{\frac{\pi}{|t|}}^{\infty} e^{i\, \sgn{(t)}\pi}e^{-it\tau}F(\tau)\,d\tau    \right]\\
&\, = \frac{1}{2} \left[    \int_{\frac{\pi}{|t|}}^{\infty} e^{-it\tau}F(\tau)\,d\tau    -  \int_{\frac{\pi}{|t|}}^{\infty} e^{-it\left(\tau-\frac{\pi}{|t|}\right)}F(\tau)\,d\tau  \right]\\
&\, = \frac{1}{2} \left[    \int_{\frac{\pi}{|t|}}^{\infty} e^{-it\tau}F(\tau)\,d\tau    -  \int_{0}^{\infty} e^{-it\tau}F\left(\tau+\frac{\pi}{|t|}\right)\,d\tau  \right].\\
\end{split}\end{equation*}

Writing $F\left( \tau+ \frac{\pi}{|t|} \right)  =  \left[F\left( \tau+ \frac{\pi}{|t|} \right) - F(\tau)\right]+F(\tau),$ we have that 

\begin{equation*}\begin{split}
L &\,=   \frac{1}{2} \left(   - \int_0^{\frac{\pi}{|t|}}e^{-it\tau}F(\tau)\,d\tau    -  \int_{0}^{\infty} e^{-it\tau}\left[ F\left(\tau+\frac{\pi}{|t|}\right)- F(\tau)\right]\,d\tau \right).\\
\end{split}\end{equation*}

\noindent Using this last expression in equation (\ref{eq:S y L}), it follows that 

\begin{equation*}
I(t) = \frac{1}{2} \left( \int_0^{\frac{\pi}{|t|}} e^{-it\tau} F(\tau) \, d\tau - \int_{0}^{\infty} e^{-it\tau}\left[ F\left(\tau+\frac{\pi}{|t|}\right)- F(\tau)\right]\,d\tau \right).
\end{equation*}

\noindent Now let $F_1(\tau) = F\left(\tau+\frac{\pi}{|t|}\right) - F(\tau)$ and let $I_1(t) = \int_{0}^{\infty} e^{-it\tau} F_1(\tau)\,d\tau.$ Then, by the same argument, it follows that 

\begin{equation*}
I_1(t) = \frac{1}{2} \left( \int_0^{\frac{\pi}{|t|}} e^{-it\tau} F_1(\tau) \, d\tau -  \int_{0}^{\infty} e^{-it\tau}\left[ F_1\left(\tau+\frac{\pi}{|t|}\right)- F_1(\tau)\right]\,d\tau \right).
\end{equation*}

\noindent After $N$ times of repeating this process, we have that

\begin{equation*}\begin{split}
I_N(t) = & \,\int_{0}^{\infty} e^{-it\tau} F_N(\tau)\,d\tau \\ =&\, \frac{1}{2} \left( \int_0^{\frac{\pi}{|t|}} e^{-it\tau} F_N(\tau) \, d\tau -  \int_{0}^{\infty} e^{-it\tau}\left[ F_N\left(\tau+\frac{\pi}{|t|}\right)- F_N(\tau)\right]\,d\tau \right),\\
\end{split}\end{equation*}

\noindent where $F_N(\tau) = F_{N-1}\left(\tau+\frac{\pi}{|t|}\right) - F_{N-1}(\tau).$ 

\bigskip 

Letting 

\begin{equation*}
 S_j(t) = \int_0^{\frac{\pi}{|t|}} e^{-it\tau}F_j(\tau)\,d\tau
\end{equation*}

\noindent for $1 \le j \le N-1,$ it follows that

\begin{equation*}
 I(t)  =  \frac{1}{2} \left[ S(t) - \frac{1}{2} \left[ S_1(t) - \frac{1}{2} \left[ S_2(t) \cdots -\frac{1}{2}\left[ S_{N-1}(t) - I_N(t) \right] \right] \right]\right]. 
\end{equation*}

\noindent That is, 

\begin{equation*}\begin{split}
I(t)  =&\, \frac{1}{2}S(t) + \sum_{k=1}^{N} \frac{(-1)^{k}}{2^{k+1}} \int_0^{\frac{\pi}{|t|}} e^{-it\tau} F_{k}(\tau)\, d\tau  +\frac{(-1)^{N+1}}{2^{N+1}}    \int_{0}^{\infty} e^{-it\tau} F_{N+1}(\tau)\,d\tau.\\
\end{split}\end{equation*}

\noindent Notice that after expanding and rearranging terms, we can write for $1\le k\le N$

\begin{equation*}
F_k(\tau) = \sum_{j=0}^{k-1}(-1)^{k+j+1}{k-1 \choose j} \left[F\left( \tau + \frac{(j+1)\pi}{|t|}\right) - F\left( \tau + \frac{j\pi}{|t|}\right) \right].
\end{equation*}

\noindent It follows that 

\begin{equation*}\begin{split}
I(t) =&\, \frac{1}{2}S(t)+ \sum_{k=1}^{N} \sum_{j=0}^{k-1}{k-1 \choose j} \frac{(-1)^{j+1}}{2^{k+1}} \int_0^{\frac{\pi}{|t|}} e^{-it\tau} \left[F\left( \tau + \frac{(j+1)\pi}{|t|}\right) - F\left( \tau + \frac{{j}\pi}{|t|}\right) \right] \, d\tau\\ 
+&\, \sum_{j=0}^N\frac{(-1)^{j+1}}{2^{N+1}}{N\choose j} \int_0^{\frac{\pi}{|t|}}e^{-it\tau}\left[ F\left(\tau +\frac{(j+1)\pi}{|t|} \right) - F\left(\tau +\frac{j\pi}{|t|}  \right) \right]\,d\tau\\
+&\,\frac{(-1)^{N+1}}{2^{N+1}}    \int_{\frac{\pi}{|t|}}^{\infty} e^{-it\tau} F_{N+1}(\tau)\,d\tau.\\
\end{split}\end{equation*}

\noindent Changing the order of summation, we get

\begin{equation*}\begin{split}
I(t) =&\,  \frac{1}{2}S(t)+ \sum_{j=1}^{N} \sum_{k=j}^{N}{k-1 \choose j-1} \frac{(-1)^{j}}{2^{k+1}} \int_0^{\frac{\pi}{|t|}} e^{-it\tau} F\left( \tau + \frac{j\pi}{|t|}\right)  \, d\tau\\
+&\,  \sum_{j=0}^{N-1} \sum_{k=j+1}^{N}{k-1 \choose j} \frac{(-1)^{j}}{2^{k+1}} \int_0^{\frac{\pi}{|t|}} e^{-it\tau}  F\left( \tau + \frac{{j}\pi}{t}\right)\, d\tau\\
+& \, \sum_{j=1}^{N+1}\frac{(-1)^{j}}{2^{N+1}}{N\choose j-1} \int_0^{\frac{\pi}{|t|}}e^{-it\tau} F\left(\tau +\frac{j\pi}{|t|} \right) \,d\tau        \\
+&  \, \sum_{j=0}^N\frac{(-1)^{j}}{2^{N+1}}{N\choose j} \int_0^{\frac{\pi}{|t|}}e^{-it\tau} F\left(\tau +\frac{j\pi}{|t|} \right) \,d\tau+\frac{(-1)^{N+1}}{2^{N+1}}    \int_{\frac{\pi}{|t|}}^{\infty} e^{-it\tau} F_{N+1}(\tau)\,d\tau. \\
\end{split}\end{equation*}

\noindent Letting

\begin{equation*}\begin{split}
& c_0= \frac{1}{2} + \sum_{k=1}^{N}\frac{1}{2^{k+1}}+\frac{1}{2^{N+1}};\\
&c_j = \sum_{k=j}^{N}{k-1 \choose j-1} \frac{1}{2^{k+1}} + \sum_{k=j+1}^{N}{k-1\choose j}\frac{1}{2^{k+1}}+{N\choose j-1}\frac{1}{2^{N+1}} +{N\choose j}\frac{1}{2^{N+1}}\,\,\,{\text {for}}\,\,\, 1\le J\le N-1;\\
& c_{N} = \frac{N+2}{2^{N+1}};\,\,\,{\text{and}}\\
& c_{N+1} = \frac{1}{2^{N+1}},\\
\end{split}\end{equation*}

\noindent it follows that

 \begin{equation}\label{eq:antes de F}
 |I(t)| \le \sum_{j=0}^{N+1} c_j \left\vert \int_0^{\frac{\pi}{|t|}} e^{-it\tau} F\left( \tau +\frac{j\pi}{|t|}\right) \, d\tau \right\vert + \frac{1}{2^{N+1}} \left\vert  \int_{0}^{\infty} e^{-it\tau} F_{N+1}(\tau)\,d\tau \right\vert. 
\end{equation}

\noindent It is worth noting that the exact form of the coefficients $c_k$ for $1\le k\le N$ is irrelevant. The only fact that will be needed is that they exist and are positive. 

\bigskip

It suffices now to show that 

\begin{equation*}
F_{N+1}(\tau) = \int_0^{\frac{\pi}{|t|}}\cdots\int_0^{\frac{\pi}{|t|}} F^{(N+1)}(\tau+s_1+\ldots + s_{N+1}) \,ds_1\cdots ds_{N+1},
\end{equation*}

\noindent where

\begin{equation}\label{eq:F}\begin{split}
 F_{N+1}(\tau) =&\, \sum_{j=0}^{N}(-1)^{N+j}{N \choose j} \left[F\left( \tau + \frac{(j+1)\pi}{|t|}\right) - F\left( \tau + \frac{j\pi}{|t|}\right) \right] \\
=&\, \sum_{j=0}^{N}(-1)^{N+j}{N \choose j} \int_0^{\frac{\pi}{|t|}} F'\left(\tau + s +  \frac{j\pi}{|t|}\right)\,ds.\\
\end{split}\end{equation}

\noindent Using the identity 

\begin{equation*}
 {n \choose k} = {n-1\choose k} + {n-1\choose k-1},
\end{equation*}

\noindent it follows that for any integer $M>0$ and for any function $h,$

\begin{equation}\label{eq:choose}\begin{split}
&\sum_{j=0}^M (-1)^{M+j}{M\choose j}h(j) = \sum_{j=0}^{M-1}(-1)^{M+j+1}{M-1\choose j}\left[ h(j+1)-h(j)\right].\\
\end{split}
\end{equation}

\noindent Let $h(j) =  \int_0^{\frac{\pi}{|t|}} F'\left(\tau + s +  \frac{j\pi}{|t|}\right)\,ds$ and $M=N.$ Notice that we can write

\begin{equation}\label{eq:dif}\begin{split}
h(j+1) - h(j) =&\, \int_0^{\frac{\pi}{|t|}}\int_0^{\frac{\pi}{|t|}}F''\left(\tau + s_1+s_2 +  \frac{j\pi}{|t|}\right)\,ds_1\,ds_2. \\
\end{split}\end{equation}

\noindent It follows from equations (\ref{eq:F}), (\ref{eq:choose}) and (\ref{eq:dif}) that

\begin{equation*}\begin{split}
F_{N+1}(\tau) =  \sum_{j=0}^{N-1}(-1)^{N+j+1}{N-1\choose j} \int_0^{\frac{\pi}{|t|}}\int_0^{\frac{\pi}{|t|}}F''\left(\tau + s_1+s_2 +  \frac{j\pi}{|t|}\right)\,ds_1\,ds_2. \\
\end{split}\end{equation*}

\noindent Repeating this process $N-1$ times, where the $i^{th}$ time  we choose
$$h(j) =  \int_0^{\frac{\pi}{|t|}}\int_0^{\frac{\pi}{|t|}} F^{(i+1)}\left(\tau + s_1+\ldots +s_{i+1} +  \frac{j\pi}{|t|}\right)\,ds_1\ldots\,ds_{i+1}$$ 
\noindent and $M=N-i,$   we obtain 

\begin{equation*}\begin{split}
F_{N+1}(\tau) = \int_0^{\frac{\pi}{|t|}}\cdots\int_0^{\frac{\pi}{|t|}}F^{(N+1)}\left(\tau + s_1+\ldots+ s_{N+1}  \right)\,ds_1\,\cdots\,ds_{N+1}.\\
\end{split}\end{equation*}

\noindent This finishes the proof of Claim \ref{cl:truco de wiener}.

\end{proof}

\bigskip

\begin{claim}\label{cl:deriv de F}

The $N^{th}$ derivative of 

\begin{equation}\label{eq:FFF}
F(\tau) =  e^{-2\pi\tau\delta} {\scaleint{7ex}_{\bs \mathbb R^n}} \frac{e^{2\pi i\vec{\eta}\cdot (\vec{y}-\vec{y'}) }}{\displaystyle\int_{\mathbb R^n} e^{4\pi [\vec{\eta}\cdot \vec{v} -\tau\tilde{b}(\vec{v})]} \, d\vec{v}} \, d\vec{\eta}
\end{equation}

\noindent consists of sums of terms of the form 

\begin{equation*}
\frac{C(\tau\delta)^{N-k}e^{-2\pi\tau\delta}}{\tau^N} {\scaleint{7ex}_{\bs \mathbb R^n}} \frac{e^{2\pi i\vec{\eta}\cdot (\vec{y}-\vec{y'}) }f_1(\tau,\vec{\eta}) \cdots f_k(\tau,\vec{\eta})}{\gamma(\tau,\vec{\eta})} \, d\vec{\eta},
\end{equation*}

\noindent where  

$$f_s(\tau,\vec{\eta}) = \left[\,\,\displaystyle\int \limits_{\mathbb R^n}(\tau\widetilde{b}(\vec{v}))^{s} e^{4\pi [\vec{\eta}\cdot \vec{v} -\tau\tilde{b}(\vec{v})]} \, d\vec{v}\right]^{a_s};$$

\begin{equation}\label{eq:ga}
 \gamma(\tau,\vec{\eta}) = \left[\,\,\displaystyle\int \limits_{\mathbb R^n} e^{4\pi [\vec{\eta}\cdot \vec{v} -\tau\tilde{b}(\vec{v})]} \, d\vec{v}\right]^d;
 \end{equation}

\noindent  $a_1,\ldots,a_k,k, d \in \mathbb{N};$ $0\le k\le N;$ $a_1 + \ldots + a_k = d-1;$ and $a_1 + 2a_2 +\ldots + ka_k = k.$ 

\end{claim}

\begin{proof} We will begin by showing by induction that the $k^{th}$ derivative of 

$$ J(\tau) = {\scaleint{7ex}_{\bs \mathbb R^n}} \frac{e^{2\pi i\vec{\eta}\cdot (\vec{y}-\vec{y'}) }}{\displaystyle\int \limits_{\mathbb R^n} e^{4\pi [\vec{\eta}\cdot \vec{v} -\tau\tilde{b}(\vec{v})]} \, d\vec{v}} \, d\vec{\eta}$$

\noindent consists of sums of terms of the form 

$$ C {\scaleint{9ex}_{\bs \mathbb R^n}} \frac{e^{2\pi i\vec{\eta}\cdot (\vec{y}-\vec{y'}) }\left[\,\,\displaystyle\int_{\mathbb R^n}\widetilde{b}(\vec{v}) e^{4\pi [\vec{\eta}\cdot \vec{v} -\tau\tilde{b}(\vec{v})]} \, d\vec{v}\right]^{a_1}\cdots \left[\,\,\displaystyle\int_{\mathbb R^n}\widetilde{b}(\vec{v})^{k} e^{4\pi [\vec{\eta}\cdot \vec{v} -\tau\tilde{b}(\vec{v})]} \, d\vec{v}\right]^{a_k}}{\left[\,\,\displaystyle\int_{\mathbb R^n} e^{4\pi [\vec{\eta}\cdot \vec{v} -\tau\tilde{b}(\vec{v})]} \, d\vec{v}\right]^{d}} \, d\vec{\eta},$$ 

\noindent where $a_1 + \ldots + a_k = d-1;$ and $a_1 + 2a_2 +\ldots + ka_k = k.$ 

\bigskip

Notice that $J'(\tau) $ is of this form, with $a_1=1,$ and $d=2.$ Suppose $J^{(k)}(\tau)$ is of this form.  We will show that $J^{(k+1)}(\tau)$ is of this form. Let

$$g_s(\tau) = \left[\,\,\displaystyle\int \limits_{\mathbb R^n}\widetilde{b}(\vec{v})^{s} e^{4\pi [\vec{\eta}\cdot \vec{v} -\tau\tilde{b}(\vec{v})]} \, d\vec{v}\right]^{a_s}.$$

\noindent Then $\dfrac{d}{d\tau}\left[ J^{(k)}(\tau)\right]$ consists of sums of terms of the form 

\begin{equation}\label{eq:termino 1}
C \scaleint{7ex}_{\bs \mathbb{R}^n} \frac{e^{2\pi i\vec{\eta}\cdot (\vec{y}-\vec{y'}) }g_1\cdots g_{s-1}\frac{d}{d\tau}(g_s)g_{s+1}\cdots g_k}{\gamma}\, d\vec{\eta}
\end{equation}

\noindent or 

\begin{equation}\label{eq:termino 2}
C \scaleint{7ex}_{\bs \mathbb{R}^n} \frac{e^{2\pi i\vec{\eta}\cdot (\vec{y}-\vec{y'}) }g_1\cdots  g_k\frac{d}{d\tau}(\gamma)}{\gamma^2}\, d\vec{\eta},
\end{equation}

\noindent with $\gamma$ as in equation (\ref{eq:ga}). But 

$$ \frac{d}{d\tau}(g_s) = -4\pi a_s \left[\,\,\displaystyle\int_{\mathbb R^n}\widetilde{b}(\vec{v})^{s} e^{4\pi [\vec{\eta}\cdot \vec{v} -\tau\tilde{b}(\vec{v})]} \, d\vec{v}\right]^{a_s-1}\left[\,\,\displaystyle\int_{\mathbb R^n}\widetilde{b}(\vec{v})^{s+1} e^{4\pi [\vec{\eta}\cdot \vec{v} -\tau\tilde{b}(\vec{v})]} \, d\vec{v}\right],$$ 

\noindent and 

\begin{equation*}\begin{split} \frac{d}{d\tau}(\gamma) 
=&\, -4\pi d  \gamma^\frac{d-1}{d}  \displaystyle\int_{\mathbb R^n}\widetilde{b}(\vec{v}) e^{4\pi [\vec{\eta}\cdot \vec{v} -\tau\tilde{b}(\vec{v})]} \, d\vec{v}. \\
\end{split}\end{equation*}

\noindent Thus, a generic term of the form given in equation (\ref{eq:termino 1}) is given by

\begin{equation}\label{eq:111}
C \scaleint{7ex}_{\bs \mathbb{R}^n} \frac{e^{2\pi i\vec{\eta}\cdot (\vec{y}-\vec{y'}) }g_1\cdots g_k}{\gamma}\left[\,\,\displaystyle\int_{\mathbb R^n}\widetilde{b}(\vec{v})^{s} e^{4\pi [\vec{\eta}\cdot \vec{v} -\tau\tilde{b}(\vec{v})]} \, d\vec{v}\right]^{-1}\left[\,\,\displaystyle\int_{\mathbb R^n}\widetilde{b}(\vec{v})^{s+1} e^{4\pi [\vec{\eta}\cdot \vec{v} -\tau\tilde{b}(\vec{v})]} \, d\vec{v}\right]\, d\vec{\eta}.
\end{equation}

Let $h_j = g_j$ for $j \neq  s, s+1;$ $h_s =\left[\,\, \displaystyle\int_{\mathbb R^n}\widetilde{b}(\vec{v})^{s} e^{4\pi [\vec{\eta}\cdot \vec{v} -\tau\tilde{b}(\vec{v})]} \, d\vec{v}\right]^{\widetilde{a}_s},$ where $\widetilde{a}_s= a_s-1;$ and $h_{s+1} = \left[\,\, \displaystyle\int_{\mathbb R^n}\widetilde{b}(\vec{v})^{s+1} e^{4\pi [\vec{\eta}\cdot \vec{v} -\tau\tilde{b}(\vec{v})]} \, d\vec{v}\right]^{\widetilde{a}_{s+1}},$ where $\widetilde{a}_{s+1}= a_{s+1}+1.$ Then equation (\ref{eq:111}) can be written as 

\begin{equation*}
C \scaleint{7ex}_{\bs \mathbb{R}^n} \frac{e^{2\pi i\vec{\eta}\cdot (\vec{y}-\vec{y'}) }h_1\cdots h_k}{\gamma}\, d\vec{\eta}.
\end{equation*}

\noindent For this term to have the desired form, the exponents must satisfy $a_1+\ldots +a_{s-1}+\widetilde{a}_s + \widetilde{a}_{s+1}+ a_{s+2} + \ldots + a_k = d-1,$ and $a_1+\ldots +(s-1)a_{s-1}+s\widetilde{a}_s + (s+1)\widetilde{a}_{s+1}+ (s+2)a_{s+2} + \ldots + ka_k = k+1.$ The former holds, since by inductive hypothesis $a_1+\ldots +a_{s-1}+\widetilde{a}_s + \widetilde{a}_{s+1}+ a_{s+2} + \ldots + a_k  = a_1 +\ldots + a_s -1 + a_s+1 + \ldots +a_k = a_1+\ldots+a_k = d-1.$ The latter also holds, since by inductive hypothesis,  $a_1+\ldots +(s-1)a_{s-1}+s\widetilde{a}_s + (s+1)\widetilde{a}_{s+1}+ (s+2)a_{s+2} + \ldots + ka_k =  a_1+\ldots +(s-1)a_{s-1}+s{a}_s -s + (s+1){a}_{s+1} +s+1+ (s+2)a_{s+2} + \ldots + ka_k = a_1 +\ldots +ka_k +1 = k+1.$

\bigskip

 In the same way, a generic term of the form given in equation (\ref{eq:termino 2}) is given by
 
  \begin{equation}\label{eq:222}
  C \scaleint{8ex}_{\bs \mathbb{R}^n} \frac{e^{2\pi i\vec{\eta}\cdot (\vec{y}-\vec{y'}) }g_1\cdots  g_k  \displaystyle\int_{\mathbb R^n}\widetilde{b}(\vec{v}) e^{4\pi [\vec{\eta}\cdot \vec{v} -\tau\tilde{b}(\vec{v})]} \, d\vec{v}}{\left[\,\,\displaystyle\int_{\mathbb R^n} e^{4\pi [\vec{\eta}\cdot \vec{v} -\tau\tilde{b}(\vec{v})]} \, d\vec{v}\right]^{d+1}} \, d\vec{\eta}.
  \end{equation}

Let $h_j = g_j$ for $j \neq  1,$ and $h_1 =\left[\,\, \displaystyle\int_{\mathbb R^n}\widetilde{b}(\vec{v}) e^{4\pi [\vec{\eta}\cdot \vec{v} -\tau\tilde{b}(\vec{v})]} \, d\vec{v}\right]^{\widetilde{a}_1},$ where $\widetilde{a}_1= a_1+1$ so that equation (\ref{eq:222}) can be written as 

  \begin{equation*}
  C \scaleint{8ex}_{\bs \mathbb{R}^n} \frac{e^{2\pi i\vec{\eta}\cdot (\vec{y}-\vec{y'}) }h_1\cdots  h_k  }{\left[\,\,\displaystyle\int_{\mathbb R^n} e^{4\pi [\vec{\eta}\cdot \vec{v} -\tau\tilde{b}(\vec{v})]} \, d\vec{v}\right]^{d+1}} \, d\vec{\eta}.
  \end{equation*}

\noindent For this term to have the desired form, the exponents must satisfy $\widetilde{a}_1+a_2+\ldots  + a_k = d,$ and $\widetilde{a}_1+ 2a_2 +\ldots + ka_k = k+1.$ The former holds, since by inductive hypothesis, $\widetilde{a}_1+a_2+\ldots  + a_k = a_1 +\ldots+ a_k + 1 = (d-1)+1 = d.$ The latter also holds, since by inductive hypothesis $\widetilde{a}_1+ 2a_2 +\ldots + ka_k = a_1+2a_2+\ldots+ka_k + 1 = k+1.$ 

\bigskip

It follows that for any $k\in\mathbb{N},$ the $k^{th}$ derivative of $J(\tau)$ consists of sums of terms of the form 

$$ C  {\scaleint{8ex}_{\bs \mathbb R^n}} \frac{e^{2\pi i\vec{\eta}\cdot (\vec{y}-\vec{y'}) }\left[\,\,\displaystyle\int_{\mathbb R^n}\widetilde{b}(\vec{v}) e^{4\pi [\vec{\eta}\cdot \vec{v} -\tau\tilde{b}(\vec{v})]} \, d\vec{v}\right]^{a_1}\cdots \left[\,\,\displaystyle\int_{\mathbb R^n}\widetilde{b}(\vec{v})^{k} e^{4\pi [\vec{\eta}\cdot \vec{v} -\tau\tilde{b}(\vec{v})]} \, d\vec{v}\right]^{a_k}}{\left[\,\,\displaystyle\int_{\mathbb R^n} e^{4\pi [\vec{\eta}\cdot \vec{v} -\tau\tilde{b}(\vec{v})]} \, d\vec{v}\right]^{d}} \, d\vec{\eta},$$

\noindent where $a_1 + \ldots + a_k = d-1;$ and $a_1 + 2a_2 +\ldots + ka_k = k.$ 

\bigskip

Finally, since $F(\tau) = e^{-2\pi\tau\delta} J(\tau),$ the $N^{th}$ derivative of $F$ is given by

$$ F^{(N)}(\tau) = \sum_{k=0}^N{N\choose k} \left( e^{-2\pi\tau\delta}\right)^{(N-k)} J^{(k)}(\tau).$$

\noindent But $\left( e^{-2\pi\tau\delta}\right)^{(N-k)} = C \delta^{N-k} e^{-2\pi\tau\delta}.$ Thus, the $N^{th}$ derivative of $F$ consists of sums of multiples of terms of the form 

\begin{equation*}
 \delta^{N-k}e^{-2\pi\tau\delta}{\scaleint{8ex}_{\bs \mathbb R^n}} \frac{e^{2\pi i\vec{\eta}\cdot (\vec{y}-\vec{y'}) } \left[\,\,\displaystyle\int_{\mathbb R^n}\widetilde{b}(\vec{v}) e^{4\pi [\vec{\eta}\cdot \vec{v} -\tau\tilde{b}(\vec{v})]} \, d\vec{v}\right]^{a_1} \cdots  \left[\,\,\displaystyle\int_{\mathbb R^n}\widetilde{b}(\vec{v})^{k} e^{4\pi [\vec{\eta}\cdot \vec{v} -\tau\tilde{b}(\vec{v})]} \, d\vec{v}\right]^{a_k}}{\left[\,\,\displaystyle\int_{\mathbb R^n} e^{4\pi [\vec{\eta}\cdot \vec{v} -\tau\tilde{b}(\vec{v})]} \, d\vec{v}\right]^d} \, d\vec{\eta},
\end{equation*}

\noindent  where $a_1 + \ldots + a_k = d-1$ and $a_1 + 2a_2 +\ldots + ka_k = k.$ 

\bigskip

Finally, writing for $1\le s\le k,$

$$\left[\,\,\displaystyle\int \limits_{\mathbb R^n}\widetilde{b}(\vec{v})^{s} e^{4\pi [\vec{\eta}\cdot \vec{v} -\tau\tilde{b}(\vec{v})]} \, d\vec{v}\right]^{a_s} =\frac{1}{\tau^{sa_s}}\left[\,\,\displaystyle\int \limits_{\mathbb R^n}(\tau\widetilde{b}(\vec{v}))^{s} e^{4\pi [\vec{\eta}\cdot \vec{v} -\tau\tilde{b}(\vec{v})]} \, d\vec{v}\right]^{a_s}, $$

\noindent yields the desired expression. This finishes the proof of Claim \ref{cl:deriv de F}.

\end{proof}

\bigskip

\begin{claim}\label{bound de Delta}

Let 

 \begin{equation}
 \Delta_{N,k}^{\vec{\mu}}(\tau) =  (\tau\delta)^{N-k}e^{-2\pi\tau\delta} {\scaleint{7ex}_{\bs\mathbb R^n}} \frac{e^{2\pi i\frac{\vec{\eta}}{\vec{\mu}}\cdot (\vec{y}-\vec{y'}) }f_1^{\mu}(\tau,\vec{\eta}) \cdots f_k^{\mu}(\tau,\vec{\eta})}{\gamma^{\mu}(\tau,\vec{\eta})} \, d\vec{\eta},
\end{equation}

\noindent where

$$f_s^{\mu}(\tau,\vec{\eta}) = \left[\,\,\displaystyle\int \limits_{\mathbb R^n}(\tau\widetilde{b}(\vec{\mu}\vec{v}))^{s} e^{4\pi [\vec{\eta}\cdot \vec{v} -\tau\tilde{b}(\vec{\mu}\vec{v})]} \, d\vec{v}\right]^{a_s};$$

$$ \gamma^{\mu}(\tau,\vec{\eta}) = \left[\,\,\displaystyle\int \limits_{\mathbb R^n} e^{4\pi [\vec{\eta}\cdot \vec{v} -\tau\tilde{b}(\vec{\mu}\vec{v})]} \, d\vec{v}\right]^d;$$

\noindent $a_1,\ldots,a_k,k, d \in \mathbb{N};$ $0\le k\le N;$ $a_1 + \ldots + a_k = d-1;$ and $a_1 + 2a_2 +\ldots + ka_k = k.$ Then there exists a constant $C $ that depends only on $N,$ $k,$ the combined degree of $b$ and the dimension of the space such that

$$ |\Delta_{N,k}^{\vec{\mu}}(\tau)| \le C.$$

\end{claim}

\begin{remark}
Notice that $\mu$ is a function of $\tau.$ 
\end{remark}

\begin{proof} It is easy to check that $(\tau\delta)^{N-k}e^{-2\pi\tau\delta}$ is bounded. Thus, it suffices to show that 
$${\scaleint{7ex}_{\bs \mathbb R^n}} \frac{f_1^{\mu}(\tau,\vec{\eta}) \cdots f_k^{\mu}(\tau,\vec{\eta})}{\gamma^{\mu}(\tau,\vec{\eta})} \, d\vec{\eta}$$

\noindent is bounded. We will begin by studying the properties of the  $f_s^{\mu}(\tau,\vec{\eta})$  defined above. By Claim \ref{cl:bound g sin coefs en 2 dim} on page \pageref{cl:bound g sin coefs en 2 dim} there exists a universal constant such that

$$ \tau\widetilde{b}(\vec{\mu}\vec{v}) \le C(1 + v_1^{2m_1}+\ldots +v_n^{2m_n}).$$

\noindent Thus, we must study the behavior of integrals of the form 

$$ L_s = \int_{\mathbb{R}^n} p_s(\vec{v})e^{\vec{\eta}\cdot\vec{v} - g(\vec{v})}\,d\vec{v},$$

\noindent for polynomials $p_s:\mathbb{R}^n\rightarrow\mathbb{R}^+$ with non-negative coefficients and $g(\vec{v}) = \tau\tilde{b}(\vec{\mu}\vec{v}).$ As usual (see, e.g., equation (\ref{eq:def de I}) on page \pageref{eq:def de I}), we can write 

\begin{equation*}
L_s = e^{h(\vec{v_0})} \int_{\mathbb{R}^n} e^{h(\vec{v})-h(\vec{v_0})} p_s(\vec{v})\, d\vec{v},
\end{equation*}

\noindent where $h(\vec{v}) = \vec{\eta}\cdot \vec{v} - g(\vec{v})$ and $\vec{v_0}$ is the point where $h(\vec{v})$ attains its maximum. Making the change of variables $\vec{v} = \vec{w}+\vec{v_0},$ it follows that

$$ L_s= e^{h(\vec{v_0})} \int_{\mathbb{R}^n} e^{-f(\vec{w})}p_s(\vec{w}+\vec{v_0)} \, d\vec{w},$$

\noindent where $f(\vec{w})= g(\vec{v_0}+\vec{w}) -g(\vec{v_0}) - \nabla g(\vec{v_0})\cdot \vec{w}.$ Since the coefficients of $p$ are non-negative, it follows that 

$$p_s(w_1+v_{01}, \ldots, w_n+v_{0n}) \le p_s(|\vec{w}|+|\vec{v_0}|, \ldots, |\vec{w}|+|\vec{v_0}| ).$$

\noindent This last polynomial is now a polynomial of just one variable, and after expanding and regrouping all the terms, it consists of sums of terms of the form $|\vec{w}|^{j_s}|\vec{v_0}|^{i_s}$ for indices $i_s$ and $j_s.$ Hence, there exist positive coefficients $c_{i_s,j_s}$ such that

$$ L_s \le e^{h(\vec{v_0})} \sum_{i_s,j_s}c_{i_s,j_s}  |\vec{v_0}|^{i_s}\int_{\mathbb{R}^n}|\vec{w}|^{j_s}e^{-f(\vec{w})} \, d\vec{w}.$$

\noindent Let $J_{j_s} = \displaystyle\int_{\mathbb{R}^n}|\vec{w}|^{j_s}e^{-f(\vec{w})} \, d\vec{w}.$ This integral is identical to the one studied in equation (\ref{eq:jota s}) in Claim \ref{cl:bound de las derivadas} on page \pageref{eq:jota s}. Thus, 

$$J_{j_s} \lesssim  |\{\vec{w}\,:\, f(\vec{w})\le 1 \}|  \,  (1 + |\vec{v_0}|^{j_sB}) + \Theta,$$

\noindent where $\Theta$ is a constant that depends only on $m_1,\ldots,m_n$ and the dimension of the space; and $B =4 \max \{m_1,\ldots,m_n\}.$ It follows that

$$ L_s \lesssim e^{h(\vec{v_0})} \sum_{i_s,j_s}c_{i_s,j_s}  |\vec{v_0}|^{i_s} \left[ |\{\vec{w}\,:\, f(\vec{w})\le 1 \}|  \,  (1 + |\vec{v_0}|^{j_sB}) + \Theta\right].$$

\noindent That is, 

$$ L_s \lesssim e^{h(\vec{v_0})}\left[\,|\{\vec{w}\,:\, f(\vec{w})\le 1 \}|\,\phi_s(|\vec{v_0}|) +\psi_s(|\vec{v_0}|)\,\right],$$

\noindent for some polynomials $\phi_s:\mathbb{R}\rightarrow \mathbb{R}^+,$  $\psi_s:\mathbb{R}\rightarrow \mathbb{R}^+.$ 

\bigskip

On the other hand, by equation (\ref{eq:numerador der})

$$ \displaystyle\int \limits_{\mathbb R^n} e^{4\pi [\vec{\eta}\cdot \vec{v} -\tau\tilde{b}(\vec{\mu}\vec{v})]} \, d\vec{v} \approx e^{h(\vec{v_0})}|\{\vec{w}\,:\, f(\vec{w})\le 1 \}| .$$

\noindent Therefore, there exist some polynomials $q_j:\mathbb{R}\rightarrow\mathbb{R}^+$ such that

$${\scaleint{7ex}_{\bs\mathbb R^n}} \frac{f_1^{\mu}(\tau,\vec{\eta}) \cdots f_k^{\mu}(\tau,\vec{\eta})}{\gamma^{\mu}(\tau,\vec{\eta})} \, d\vec{\eta}\lesssim  {\scaleint{7ex}_{\bs\mathbb R^n}}\frac{[e^{h(\vec{v_0})}]^{a_1 + \ldots + a_k}\left( \sum_{j=0}^{a_1+\ldots+a_k} |\{\vec{w}\,:\, f(\vec{w})\le 1 \}|^j q_j(|\vec{v_0}|)\right)}{ [e^{h(\vec{v_0})}]^d |\{\vec{w}\,:\, f(\vec{w})\le 1 \}|^d}\, d\vec{\eta}.   $$

\noindent But  $a_1 + \ldots + a_k = d-1,$ so 

$${\scaleint{7ex}_{\bs \mathbb R^n}} \frac{f_1^{\mu}(\tau,\vec{\eta}) \cdots f_k^{\mu}(\tau,\vec{\eta})}{\gamma^{\mu}(\tau,\vec{\eta})} \, d\vec{\eta}\lesssim  {\scaleint{7ex}_{\bs\mathbb R^n}} e^{-h(\vec{v_0})}\left( \sum_{j=0}^{d-1} |\{\vec{w}\,:\, f(\vec{w})\le 1 \}|^{j-d} q_j(|\vec{v_0}|)\right)\, d\vec{\eta}.  $$

\noindent Moreover, by Claim \ref{cl:cota de f},

$$|\{\vec{w}:f(\vec{w}) \le 1 \}| \gtrsim (1+r(\vec{v_0}))^{-\frac{n}{2}}.$$ 

\noindent where $r(\vec{v}) = v_1^{2m_1}+\ldots + v_n^{2m_n}.$ Thus, and since $j-d <0,$ 

$${\scaleint{7ex}_{\bs \mathbb R^n}} \frac{f_1^{\mu}(\tau,\vec{\eta}) \cdots f_k^{\mu}(\tau,\vec{\eta})}{\gamma^{\mu}(\tau,\vec{\eta})} \, d\vec{\eta}\lesssim  {\scaleint{7ex}_{\bs\mathbb R^n}} e^{-h(\vec{v_0})}\left( \sum_{j=0}^{d-1} (1+r(\vec{v_0}))^{\frac{n(d-j)}{2}} q_j(|\vec{v_0}|)\right)\, d\vec{\eta}.  $$

\noindent By Claims \ref{cl: r final} and \ref{cl:porte de v0 gen}, $\sum_{j=0}^{d-1} (1+r(\vec{v_0}))^{\frac{n(d-j)}{2}} q_j(|\vec{v_0}|)$ is at most of polynomial growth in $|\vec{\eta}|.$ On the other hand, by equation (\ref{eq:termino dominante}), $e^{-h(\vec{v_0})}$ decays exponentially in $|\vec{\eta}|.$ Hence, 

$${\scaleint{6ex}_{\bs \mathbb R^n}} \frac{f_1^{\mu}(\tau,\vec{\eta}) \cdots f_k^{\mu}(\tau,\vec{\eta})}{\gamma^{\mu}(\tau,\vec{\eta})} \, d\vec{\eta}$$

\noindent is bounded. This finishes the proof of Claim \ref{bound de Delta}.

\end{proof}

\begin{corollary}\label{cor:fint}
Let 

\begin{equation}
F(\tau) =   e^{-2\pi\tau\delta} {\scaleint{7ex}_{\bs \mathbb R^n}} \frac{e^{2\pi i\vec{\eta}\cdot (\vec{y}-\vec{y'}) }}{\displaystyle\int_{\mathbb R^n} e^{4\pi [\vec{\eta}\cdot \vec{v} -\tau\tilde{b}(\vec{v})]} \, d\vec{v}} \, d\vec{\eta}.
\end{equation}

\noindent Then $F \in L^1(\mathbb{R}).$ 
\end{corollary}

\begin{proof}

Making the change of variables $\vec{\eta} \rightarrow \frac{\vec{\eta}}{\vec{\mu}}$ and $\vec{v} \rightarrow \vec{\mu}\vec{v},$ we have that

\begin{equation}\label{eq:fconmu}
F(\tau) =  \frac{e^{-2\pi\tau\delta}}{\mu_1^2\cdots\mu_n^2} {\scaleint{7ex}_{\bs \mathbb R^n}} \frac{e^{2\pi i \frac{\vec{\eta}}{\vec{\mu}}\cdot (\vec{y}-\vec{y'}) }}{\displaystyle\int_{\mathbb R^n} e^{4\pi [\vec{\eta}\cdot \vec{v} -\tau\tilde{b}(\vec{\mu}\vec{v})]} \, d\vec{v}} \, d\vec{\eta}.
\end{equation}

\noindent Taking $k=N = 0,$ and $d=1,$ it follows from the proof of Claim \ref{bound de Delta} that 

\begin{equation*}
\left\vert{\scaleint{7ex}_{\bs \mathbb R^n}} \frac{e^{2\pi i \frac{\vec{\eta}}{\vec{\mu}}\cdot (\vec{y}-\vec{y'}) }}{\displaystyle\int_{\mathbb R^n} e^{4\pi [\vec{\eta}\cdot \vec{v} -\tau\tilde{b}(\vec{\mu}\vec{v})]} \, d\vec{v}} \, d\vec{\eta}\right\vert
\end{equation*} 

\noindent is bounded. Thus, and by equation (\ref{eq:eleccion mus}), it follows that 

$$| F(\tau) | \lesssim  \frac{e^{-2\pi\tau\delta}}{\left|\left\{ \vec{v}\,:\, \widetilde{b}(\vec{v}) \le \frac{1}{\tau} \right\}\right|^2}.$$

If $\tau \le 1, $ then $| F(\tau) |$ is bounded. If $\tau >1,$ it follows from Claim \ref{cl:B-M} (see Appendix) that 

$$| F(\tau) | \lesssim  \frac{e^{-2\pi\tau\delta}\tau^{2n}}{\left|\left\{ \vec{v}\,:\, \widetilde{b}(\vec{v}) \le 1 \right\}\right|^2},$$ 

\noindent which decays exponentially as $\tau \rightarrow \infty.$

\end{proof}

\bigskip

We are now ready to present the proof of Proposition \ref{thm: S original con t}.

\begin{proof}
We had shown in equation (\ref{eq:con im}) on page \pageref{eq:con im} that

 \begin{equation*}\begin{split}
S = &\,
\int_0^{\infty}e^{- i u\tau}   F(\tau)\, d\tau,\\
\end{split}\end{equation*}

\noindent where $u =2\pi w,$ and 

\begin{equation}\label{eq:FF}
F(\tau) =   e^{-2\pi\tau\delta} {\scaleint{7ex}_{\bs \mathbb R^n}} \frac{e^{2\pi i\vec{\eta}\cdot (\vec{y}-\vec{y'}) }}{\displaystyle\int_{\mathbb R^n} e^{4\pi [\vec{\eta}\cdot \vec{v} -\tau\tilde{b}(\vec{v})]} \, d\vec{v}} \, d\vec{\eta}.
\end{equation}

\noindent Then by Corollary \ref{cor:fint} and Claim \ref{cl:truco de wiener} it follows that

 \begin{equation*}\begin{split}
 |S| \le&\, \sum_{j=0}^{N+1} c_j \left\vert \int_0^{\frac{\pi}{|u|}} e^{-iu\tau} F\left( \tau +\frac{j\pi}{|u|}\right) \, d\tau \right\vert\\
 &\, + \frac{1}{2^{N+1}} \left\vert  \int_{\frac{\pi}{|u|}}^{\infty} e^{-iu\tau} \int_0^{\frac{\pi}{|u|}}\cdots\int_0^{\frac{\pi}{|u|}}F^{(N+1)}\left(\tau + s_1+\ldots+ s_{N+1}  \right)\,ds_1\,\cdots\,ds_{N+1}\,d\tau \right\vert. \\
\end{split}\end{equation*}

\noindent We will begin by obtaining the desired bound for the terms of the form
 $$I_j(u)=  \scaleint{6ex}_{\bs 0}^{\frac{\pi}{|u|}} e^{-iu\tau} F\left( \tau +\frac{j\pi}{|u|}\right) \, d\tau.$$
 
\noindent We can write

 $$|I_j| \le  \scaleint{7ex}_{\bs 0}^{\frac{\pi}{|u|}} \left\vert F\left( \tau +\frac{j\pi}{|u|}\right)  \right\vert \, d\tau = \scaleint{7ex}_{\bs 0}^{\frac{\pi}{|u|}}\left\vert  e^{-2\pi\delta\left( \tau +\frac{j\pi}{|u|}\right) }{\scaleint{7ex}_{\bs \mathbb R^n}} \frac{e^{2\pi i\vec{\eta}\cdot (\vec{y}-\vec{y'}) }}{\displaystyle\int \limits_{\mathbb R^n} e^{4\pi [\vec{\eta}\cdot \vec{v} -\left(\tau+\frac{j\pi}{|u|}\right)\tilde{b}(\vec{v})]} \, d\vec{v}} \, d\vec{\eta}\right\vert\,d\tau.$$

\noindent Making the change of variables $s = \tau +\frac{j\pi}{|u|}$ as well as $\vec{\eta} \rightarrow \frac{\vec{\eta}}{\vec{\mu}}$ and $\vec{v} \rightarrow \vec{\mu}\vec{v},$ we have that

\begin{equation*}
|I_j(u)| \le \scaleint{7ex}_{\bs \frac{j\pi}{|u|}}^{\frac{(j+1)\pi}{|u|}} \frac{e^{-2\pi\delta s}}{\mu_1^2\cdots\mu_n^2} {\scaleint{7ex}_{\bs \mathbb R^n}} \frac{1}{\displaystyle\int_{\mathbb R^n} e^{4\pi [\vec{\eta}\cdot \vec{v} -s\tilde{b}(\vec{\mu}\vec{v})]} \, d\vec{v}} \, d\vec{\eta}\, ds. 
\end{equation*}

\noindent By Lemma \ref{lem:Schwartz2general m}, 

$$ {\scaleint{7ex}_{\bs \mathbb R^n}} \frac{1}{\displaystyle\int_{\mathbb R^n} e^{4\pi [\vec{\eta}\cdot \vec{v} -s\tilde{b}(\vec{\mu}\vec{v})]} \, d\vec{v}} \, d\vec{\eta}$$

\noindent converges. Also, by convexity of $b,$ $\delta> 0.$ Thus, 

\begin{equation*}
|I_j(u)| \lesssim \scaleint{6ex}_{\bs \frac{j\pi}{|u|}}^{\frac{(j+1)\pi}{|u|}} \frac{1}{\mu_1^2\cdots\mu_n^2} \, ds. 
\end{equation*}

\noindent By the choice of the factors $\vec{\mu},$ it follows that 

\begin{equation*}
|I_j(u)| \lesssim \scaleint{7ex}_{\bs \frac{j\pi}{|u|}}^{\frac{(j+1)\pi}{|u|}} \frac{1}{\left|\left\{ \vec{v}\,:\, \widetilde{b}(\vec{v}) \le \frac{1}{\tau} \right\}\right|^2}\, d\tau.
\end{equation*}

\noindent But since we are considering $\tau \le \frac{(j+1)\pi}{|u|},$ then

\begin{equation*}
|I_j(u)| \lesssim \scaleint{6ex}_{\bs \frac{j\pi}{|u|}}^{\frac{(j+1)\pi}{|u|}} \frac{1}{\left|\left\{ \vec{v}\,:\, \widetilde{b}(\vec{v}) \le \frac{|u|}{\pi(j+1)} \right\}\right|^2}\, d\tau =  \frac{\pi}{|u|\left|\left\{ \vec{v}\,:\, \widetilde{b}(\vec{v}) \le \frac{|u|}{\pi(j+1)} \right\}\right|^2}.
\end{equation*}

\noindent Moreover, since $u =2\pi w,$ the desired bound follows immediately for $j=0$ and $j=1.$  For $j>1,$ it follows from Claim \ref{cl:B-M} (see Appendix) that

\begin{equation*}
\left\vert\left\{ \vec{v}\,:\, \widetilde{b}(\vec{v}) \le |w| \right\}\right\vert \le \frac{(j+1)^n}{2^n}\left\vert\left\{ \vec{v}\,:\, \widetilde{b}(\vec{v}) \le  \frac{2|w|}{(j+1)} \right\}\right\vert.
\end{equation*}

\noindent Hence, 

\begin{equation*}
|I_j(w)| \lesssim   \frac{(j+1)^{2n}}{|w|\left|\left\{ \vec{v}\,:\, \widetilde{b}(\vec{v}) \le |w| \right\}\right|^2}.
\end{equation*}

\noindent Since the sum over $j$ is finite, it follows that 

 \begin{equation}\label{eq:S con 2 partes}\begin{split}
& |S| \lesssim \frac{1}{|w|\left|\left\{ \vec{v}\,:\, \widetilde{b}(\vec{v}) \le |w| \right\}\right|^2} \\
 &\, + \frac{1}{2^{N+1}} \left\vert  \int_{\frac{\pi}{|u|}}^{\infty} e^{-iu\tau} \int_0^{\frac{\pi}{|u|}}\cdots\int_0^{\frac{\pi}{|u|}}F^{(N+1)}\left(\tau + s_1+\ldots+ s_{N+1}  \right)\,ds_1\,\cdots\,ds_{N+1}\,d\tau \right\vert. \\
\end{split}\end{equation}

\bigskip

 We must now bound the term 

$$\left\vert  \int_{\frac{\pi}{|u|}}^{\infty} e^{-iu\tau} \int_0^{\frac{\pi}{|u|}}\cdots\int_0^{\frac{\pi}{|u|}}F^{(N+1)}\left(\tau + s_1+\ldots+ s_{N+1}  \right)\,ds_1\,\cdots\,ds_{N+1}\,d\tau \right\vert.$$

\noindent As previously discussed, the $(N+1)^{th}$ derivative of $F(\tau)$ consists of sums of terms of the form 

\begin{equation*}
\frac{C(\tau\delta)^{N+1-k}e^{-2\pi\tau\delta}}{\tau^{N+1}} {\scaleint{7ex}_{\bs \mathbb R^n}} \frac{e^{2\pi i\vec{\eta}\cdot (\vec{y}-\vec{y'}) }f_1(\tau,\vec{\eta}) \cdots f_k(\tau,\vec{\eta})}{\gamma(\tau,\vec{\eta})} \, d\vec{\eta},
\end{equation*}

\noindent where $f_s(\tau,\vec{\eta}) = \left[\,\,\displaystyle\int \limits_{\mathbb R^n}(\tau\widetilde{b}(\vec{v}))^{s} e^{4\pi [\vec{\eta}\cdot \vec{v} -\tau\tilde{b}(\vec{v})]} \, d\vec{v}\right]^{a_s};$ $ \gamma(\tau,\vec{\eta}) = \left[\,\,\displaystyle\int \limits_{\mathbb R^n} e^{4\pi [\vec{\eta}\cdot \vec{v} -\tau\tilde{b}(\vec{v})]} \, d\vec{v}\right]^d;$ $a_1,\ldots,a_k,k, d \in \mathbb{N};$ $0\le k\le N+1;$ $a_1 + \ldots + a_k = d-1;$ and $a_1 + 2a_2 +\ldots + ka_k = k.$ These terms can be written as $ C\tau^{-(N+1)} \Delta_{N+1,k}(\tau),$ where

\begin{equation*}
\Delta_{N+1,k}(\tau) = (\tau\delta)^{N+1-k}e^{-2\pi\tau\delta} {\scaleint{7ex}_{\bs \mathbb R^n}} \frac{e^{2\pi i\vec{\eta}\cdot (\vec{y}-\vec{y'}) }f_1(\tau,\vec{\eta}) \cdots f_k(\tau,\vec{\eta})}{\gamma(\tau,\vec{\eta})} \, d\vec{\eta}.
\end{equation*}

\noindent Thus, we must study integrals of the form

$$J =  \scaleint{6ex}_{\bs \frac{\pi}{|u|}}^{\infty} \frac{e^{-iu\tau}}{\tau^{N+1}} \int_0^{\frac{\pi}{|u|}}\cdots\int_0^{\frac{\pi}{|u|}}\Delta_{N+1,k}\left(\tau + s_1+\ldots+ s_{N+1}  \right)\,ds_1\,\cdots\,ds_{N+1}\,d\tau.
 $$
 
\noindent With $\mu_1,\ldots,\mu_n$ chosen as in equation (\ref{eq:eleccion mus}) on page \pageref{eq:eleccion mus} we can make the changes of variable $\vec{\eta}\rightarrow \frac{\vec{\eta}}{\vec{\mu}}$ and $\vec{v}\rightarrow \vec{\mu}\vec{v}.$ We obtain an integral of the form 
 
 $$J = \scaleint{6ex}_{\bs \frac{\pi}{|u|}}^{\infty}  \frac{(\mu_1\cdots\mu_n)^{a_1+\ldots+a_k}}{\tau^{N+1}(\mu_1\cdots\mu_n)^{d+1}}  \int_0^{\frac{\pi}{|u|}}\cdots\int_0^{\frac{\pi}{|u|}}\Delta_{N+1,k}^{\vec{\mu}}\left(\tau + s_1+\ldots+ s_{N+1}  \right)\,ds_1\,\cdots\,ds_{N+1}\,d\tau,$$
 
 \noindent where now
 
 \begin{equation}
 \Delta_{N+1,k}^{\vec{\mu}}(\tau) =  (\tau\delta)^{N+1-k}e^{-2\pi\tau\delta}{\scaleint{6ex}_{\bs \mathbb R^n}} \frac{e^{2\pi i\frac{\vec{\eta}}{\vec{\mu}}\cdot (\vec{y}-\vec{y'}) }f_1^{\mu}(\tau,\vec{\eta}) \cdots f_k^{\mu}(\tau,\vec{\eta})}{\gamma^{\mu}(\tau,\vec{\eta})} \, d\vec{\eta}
\end{equation}

\noindent is as in Claim \ref{bound de Delta}. Thus, $| \Delta_{N+1,k}^{\vec{\mu}}(\tau)|\le C.$ It follows that

 $$|J| \lesssim \scaleint{6ex}_{\bs \frac{\pi}{|u|}}^{\infty}  \frac{(\mu_1\cdots\mu_n)^{a_1+\ldots+a_k}}{|u|^{N+1}\tau^{N+1}(\mu_1\cdots\mu_n)^{d+1}} \,d\tau.$$
 
 \noindent Also, $a_1 + \ldots +a_k  =d-1.$ Thus, 
 
  $$|J| \lesssim \frac{1}{|u|^{N+1}} \scaleint{6ex}_{\bs \frac{\pi}{|u|}}^{\infty}  \frac{1}{\tau^{N+1}(\mu_1\cdots\mu_n)^{2}} \,d\tau.$$

 \noindent Since $u = 2\pi w,$ and by choice of $\vec{\mu},$ it follows that

    $$|J| \lesssim \frac{1}{|w |^{N+1}} \scaleint{6ex}_{\bs \frac{1}{2|w|}}^{\infty}  \frac{1}{\tau^{N+1}\left|\left\{ \vec{v}\,:\, \widetilde{b}(\vec{v}) \le \frac{1}{\tau} \right\}\right|^2} \,d\tau.$$

 \noindent  Notice that on the interval under consideration, $(2|w|\tau)^{-1} \le 1.$ Using Claim \ref{cl:B-M}, it follows that 
    
   $$  \left|\left\{ \vec{v}\,:\, \widetilde{b}(\vec{v}) \le \frac{1}{\tau} \right\}\right| = \left|\left\{ \vec{v}\,:\, \widetilde{b}(\vec{v}) \le \frac{2|w|}{2|w|\tau} \right\}\right| \ge \frac{1}{(2|w|\tau)^n }     \left|\left\{ \vec{v}\,:\, \widetilde{b}(\vec{v}) \le 2|w| \right\}\right|.$$ 
   
   \noindent Thus, and since $\left|\left\{ \vec{v}\,:\, \widetilde{b}(\vec{v}) \le 2|w| \right\}\right| \ge \left|\left\{ \vec{v}\,:\, \widetilde{b}(\vec{v}) \le |w| \right\}\right| $ It follows that  
   
      $$  \left|\left\{ \vec{v}\,:\, \widetilde{b}(\vec{v}) \le \frac{1}{\tau} \right\}\right| \ge \frac{1}{(2|w|\tau)^n }     \left|\left\{ \vec{v}\,:\, \widetilde{b}(\vec{v}) \le |w| \right\}\right|.$$ 
   
   \noindent Taking $N \ge 2n+1,$ we have that
   
   \begin{equation*}\begin{split}    
       |J| \lesssim&\, \frac{|w|^{2n}}{|w |^{N+1}} \scaleint{6ex}_{\bs \frac{1}{2|w|}}^{\infty}  \frac{\tau^{2n}}{\tau^{N+1}\left|\left\{ \vec{v}\,:\, \widetilde{b}(\vec{v}) \le |w| \right\}\right|^2} \,d\tau \approx  \frac{1}{|w |^{N+1-2n}\left|\left\{ \vec{v}\,:\, \widetilde{b}(\vec{v}) \le |w| \right\}\right|^2} \cdot \tau^{2n-N} \biggr\vert_\frac{1}{2|w|}^{\infty}.\\
        \end{split}\end{equation*}

\noindent That is, 
  
  $$ |J| \lesssim \frac{1}{\left\vert w\right\vert  \left\vert \left\{\vec{v}\,:\,\tilde{b}(\vec{v}) \le |w|\right\}\right\vert^2}.$$

\noindent It follows from equation (\ref{eq:S con 2 partes}) that 
   
   $$ |S| \lesssim \frac{1}{\left\vert w\right\vert  \left\vert \left\{\vec{v}\,:\,\tilde{b}(\vec{v}) \le |w|\right\}\right\vert^2}.$$
   
  \noindent This finishes the proof of our third and last bound. 

\end{proof}

\bigskip

\subsubsection{Conclusion}\label{conclusion}

We have shown that 

\begin{equation*}
\left\vert S\left((\vec{x},\vec{y},t);(\vec{x'},\vec{y'},t')\right)\right\vert \lesssim \min{\{A,B,C\}},
\end{equation*}

\noindent where

\begin{equation*}
A=\frac{1}{\delta | \{ \vec{v} \,:\, \tilde{b}(\vec{v}) <\delta \} |^2};
\end{equation*}

\begin{equation*}
B= \frac{1}{ \widetilde{b}(\vec{y}-\vec{y'})| \{ \vec{v} \,:\, \tilde{b}(\vec{v}) < \widetilde{b}(\vec{y}-\vec{y'})\} |^2};
\end{equation*}

\noindent and

\begin{equation*}
C=\frac{1}{ |w| \,\,| \{ \vec{v} \,:\, \widetilde{b}(\vec{v}) <|w| \} |^2}.
\end{equation*}

\noindent Thus, to conclude the proof of the Main Theorem, it suffices to show that 

$$ \min\{A,B,C\} \lesssim \frac{1}{\sqrt{\delta^2+ \widetilde{b}(\vec{y}-\vec{y'})^2+ w^2 }\left| \left\{ \vec{v} \,:\, \tilde{b}(\vec{v}) < \sqrt{\delta^2+ \widetilde{b}(\vec{y}-\vec{y'})^2+ w^2 }\right\} \right|^2}.$$

\bigskip

Without loss of generality, suppose that $\delta \le \widetilde{b}(\vec{y}-\vec{y'}) \le |w|.$ Then, 

$$ \frac{1}{|w|} \le \frac{\sqrt{3}}{\sqrt{\delta^2+ \widetilde{b}(\vec{y}-\vec{y'})^2+ w^2 } }.$$ 

\noindent and

$$\left\{ \vec{v} \,:\, \tilde{b}(\vec{v}) < \sqrt{\delta^2+ \widetilde{b}(\vec{y}-\vec{y'})^2+ w^2 }\right\} \subseteq \left\{ \vec{v} \,:\, \tilde{b}(\vec{v}) < \sqrt{3}|w| \right\}.$$

\noindent But by Claim \ref{cl:B-M}, 

$$\left|\left\{ \vec{v} \,:\, \tilde{b}(\vec{v}) < \sqrt{3}|w| \right\}\right|^2 \le (\sqrt{3})^{2n}\left|\left\{ \vec{v} \,:\, \tilde{b}(\vec{v}) < |w| \right\}\right|^2. $$

\noindent Therefore, 

\begin{equation}\label{eq:ultima}\begin{split}
|S| \lesssim &\, \frac{1}{ |w| \,\,| \{ \vec{v} \,:\, \widetilde{b}(\vec{v}) <|w| \} |^2}\\ 
\le&\,  \frac{(\sqrt{3})^{2n+1}}{\sqrt{\delta^2+ \widetilde{b}(\vec{y}-\vec{y'})^2+ w^2 }\left| \left\{ \vec{v} \,:\, \tilde{b}(\vec{v}) < \sqrt{\delta^2+ \widetilde{b}(\vec{y}-\vec{y'})^2+ w^2 }\right\} \right|^2}.\\
\end{split}\end{equation}

\noindent This finishes the proof of the Main Theorem.

\section{Appendix}

We have included in this appendix two technical claims for convex functions that are used repeatedly throughout the paper. 

\begin{claim}\label{cl:B-M}
 Let $f:\mathbb{R}^n \rightarrow \mathbb{R}$ be a convex function such that $f(\vec{0})=0.$ Given $x>0,$ let 
 
 $$ A_x = \{ \vec{w}\in\mathbb{R}^n\,:\, f(\vec{w})\le x\}.$$
 
 \noindent Then for any constant $0\le \lambda \le 1,$ $ vol(A_{\lambda x}) \ge \lambda^n vol(A_x).$
\end{claim}

\begin{proof}
By convexity of $f,$ for any vectors $\vec{w}$ and $\vec{u}$ in $\mathbb{R}^n,$ and any constant $0\le \lambda\le 1,$

$$ f( \lambda\vec{w} + (1-\lambda)\vec{u}) \le \lambda f(\vec{w})+(1-\lambda)f(\vec{u}).$$

\noindent In particular, taking $\vec{u}=\vec{0},$ and since by hypothesis $f(\vec{0})=0,$ it follows that $ f( \lambda\vec{w}) \le \lambda f(\vec{w}).$ Thus, if $\vec{w}\in A_x,$ then $ f(\lambda \vec{w}) \le \lambda f(\vec{w}) \le \lambda x.$ That is, $\lambda\cdot A_x \subseteq A_{\lambda x}.$ It follows that 

$$  vol(A_{\lambda x}) \ge vol(\lambda\cdot A_x)  = \lambda^n vol(A_x).$$

\end{proof}

\begin{claim}\label{prop:brunn mink}

If $f:\mathbb{R}^n \rightarrow \mathbb{R}$ is a convex function such that $f(\vec{0})=0$ and $\nabla f(\vec{0}) = 0$ then

$$ I = \int_{\mathbb{R}^n}e^{-f(\vec{w})}\,d\vec{w} \approx |\{ \vec{w} \, : \, f(\vec{w}) \le 1 \}|.$$

\end{claim}

\begin{proof}
 Without loss of generality we can assume $f \not\equiv 0.$ Notice that under these hypothesis $f(\vec{v}) \ge 0\,\, \forall \vec{v} \in \mathbb{R}^n.$ A lower bound for $I$ can be easily obtained, since

$$\int_{\mathbb{R}^n}e^{-f(\vec{w})}\,d\vec{w} \ge \int_{\{\vec{w}\, : \, f(\vec{w}) \le 1 \}}e^{-f(\vec{w})}\,d\vec{w} \ge \frac{1}{e} \, |\{ \vec{w} \,:\, f(\vec{w}) \le 1 \}|.$$

\bigskip

\noindent To obtain an upper bound we can write 

\begin{equation}\label{eq:descomposicion}
\int_{\mathbb{R}^n}e^{-f(\vec{w})}\,d\vec{w} = \int_{\{\vec{w}\, : \, f(\vec{w}) \le 1 \}}e^{-f(\vec{w})}\,d\vec{w} \, + \, \sum_{j=1}^\infty \int_{\{\vec{w}\,:\, j \le f(\vec{w}) \le j+1\}} e^{-f(\vec{w})} \, d\vec{w}. 
\end{equation}

\noindent But 

$$ \int_{\{\vec{w}\,:\, j \le f(\vec{w}) \le j+1\}} e^{-f(\vec{w})} \, d\vec{w} \le e^{-j} \, |\{\vec{w} \, : \, f(\vec{w}) \le j+1 \} |.$$

\bigskip

\noindent But by Claim \ref{cl:B-M}, and taking $\lambda  = \frac{1}{j+1}$ and $x=j+1,$ it follows that 

$$ |\{ \vec{w}\,:\, f(\vec{w}) \le j+1 \}| \le (j+1)^n |\{ \vec{w}\,:\, f(\vec{w}) \le 1 \}|.$$

\noindent Hence, by equation (\ref{eq:descomposicion}), we have that

$$ I \le |\{ \vec{w} \, : \, f(\vec{w}) \le 1\}| \left( 1+\sum_{j=1}^\infty e^{-j}(j+1)^n\right).$$

\noindent Since the sum converges we get the desired upper bound. 

\end{proof}

\section{Acknowledgments} 
The results presented in this paper first appeared in my Ph.D. thesis \cite{Be14}.
I would like to thank my advisor Alexander Nagel for his guidance and contributions to this work. I would also like to thank an anonymous referee for many insightful comments and suggestions. This work was partially supported by a doctoral fellowship from CONICYT (Chile).


\begin{thebibliography}{10}

\bibitem{Be14}
{\sc S.~Benguria}, {\em Estimates for the Szego kernel on unbounded convex
  domains}, PhD thesis, University of Wisconsin--Madison, 2014.
  
 \bibitem{Bo85}
{\sc H.~P. Boas}, {\em Regularity of the {S}zeg{\H o} projection in weakly
  pseudoconvex domains}, Indiana Univ. Math. J., 34 (1985), pp.~217--223.

\bibitem{Bo87}
\leavevmode\vrule height 2pt depth -1.6pt width 23pt, {\em The {S}zeg{\H o}
  projection: {S}obolev estimates in regular domains}, Trans. Amer. Math. Soc.,
  300 (1987), pp.~109--132.

\bibitem{BoSj75}
{\sc L.~Boutet~de Monvel and J.~Sj{\"o}strand}, {\em Sur la singularit\'e des
  noyaux de {B}ergman et de {S}zeg{\H o}}, in Journ\'ees: \'{E}quations aux
  {D}\'eriv\'ees {P}artielles de {R}ennes (1975), Soc. Math. France, Paris,
  1976, pp.~123--164. Ast\'erisque, No. 34--35.

\bibitem{BrNaWa88}
{\sc J.~Bruna, A.~Nagel, and S.~Wainger}, {\em Convex hypersurfaces and
  {F}ourier transforms}, Ann. of Math. (2), 127 (1988), pp.~333--365.

\bibitem{Ca07}
{\sc C.~Carracino}, {\em Estimates for the {S}zeg\"o kernel on a model
  non-pseudoconvex domain}, Illinois J. Math., 51 (2007), pp.~1363--1396.

\bibitem{Ch88}
{\sc M.~Christ}, {\em Regularity properties of the {$\overline\partial\sb b$}
  equation on weakly pseudoconvex {CR} manifolds of dimension {$3$}}, J. Amer.
  Math. Soc., 1 (1988), pp.~587--646.

\bibitem{Fe74}
{\sc C.~Fefferman}, {\em The {B}ergman kernel and biholomorphic mappings of
  pseudoconvex domains}, Invent. Math., 26 (1974), pp.~1--65.

\bibitem{FeKoMa90}
{\sc C.~L. Fefferman, J.~J. Kohn, and M.~Machedon}, {\em H\"older estimates on
  {CR} manifolds with a diagonalizable {L}evi form}, Adv. Math., 84 (1990),
  pp.~1--90.

\bibitem{FrHa95}
{\sc G.~Francsics and N.~Hanges}, {\em Explicit formulas for the {S}zeg{\H o}
  kernel on certain weakly pseudoconvex domains}, Proc. Amer. Math. Soc., 123
  (1995), pp.~3161--3168.

\bibitem{GiHa11}
{\sc M.~Gilliam and J.~Halfpap}, {\em The {S}zeg{\H o} kernel for certain
  non-pseudoconvex domains in {$\Bbb C\sp 2$}}, Illinois J. Math., 55 (2011),
  pp.~871--894.

\bibitem{GiHa14}
\leavevmode\vrule height 2pt depth -1.6pt width 23pt, {\em The {S}zeg\"o kernel
  for nonpseudoconvex tube domains in {$\Bbb{C}\sp 2$}}, Complex Var. Elliptic
  Equ., 59 (2014), pp.~769--786.

\bibitem{Gi64}
{\sc S.~G. Gindikin}, {\em Analysis in homogeneous domains}, Uspehi Mat. Nauk,
  19 (1964), pp.~3--92.

\bibitem{GrSt78}
{\sc P.~Greiner and E.~Stein}, {\em On the solvability of some differential
  operators of type $\square_b$}, Proc. Internat. Conf. (Cortona, Italy,
  1976-1977),  (1978), pp.~106--165.

\bibitem{HaNaWa10}
{\sc J.~Halfpap, A.~Nagel, and S.~Wainger}, {\em The {B}ergman and {S}zeg{\H o}
  kernels near points of infinite type}, Pacific J. Math., 246 (2010),
  pp.~75--128.

\bibitem{Ha95}
{\sc F.~Haslinger}, {\em Singularities of the {S}zeg{\H o} kernel for certain
  weakly pseudoconvex domains in {${\bf C}\sp 2$}}, J. Funct. Anal., 129
  (1995), pp.~406--427.

\bibitem{Jo48}
{\sc F.~John}, {\em Extremum problems with inequalities as subsidiary
  conditions}, in Studies and {E}ssays {P}resented to {R}. {C}ourant on his
  60th {B}irthday, {J}anuary 8, 1948, Interscience Publishers, Inc., New York,
  N. Y., 1948, pp.~187--204.

\bibitem{Kr80}
{\sc S.~G. Krantz}, {\em Holomorphic functions of bounded mean oscillation and
  mapping properties of the {S}zeg{\H o} projection}, Duke Math. J., 47 (1980),
  pp.~743--761.

\bibitem{Kr14}
\leavevmode\vrule height 2pt depth -1.6pt width 23pt, {\em A direct connection
  between the {B}ergman and {S}zeg{\H o} projections}, Complex Anal. Oper.
  Theory, 8 (2014), pp.~571--579.

\bibitem{LaSt04}
{\sc L.~Lanzani and E.~M. Stein}, {\em Szeg\"o and {B}ergman projections on
  non-smooth planar domains}, J. Geom. Anal., 14 (2004), pp.~63--86.

\bibitem{LaSt13}
\leavevmode\vrule height 2pt depth -1.6pt width 23pt, {\em Cauchy-type
  integrals in several complex variables}, Bull. Math. Sci., 3 (2013),
  pp.~241--285.

\bibitem{Ma88}
{\sc M.~Machedon}, {\em Szeg{\H o} kernels on pseudoconvex domains with one
  degenerate eigenvalue}, Ann. of Math. (2), 128 (1988), pp.~619--640.

\bibitem{Mc94}
{\sc J.~D. McNeal}, {\em Estimates on the {B}ergman kernels of convex domains},
  Adv. Math., 109 (1994), pp.~108--139.

\bibitem{McSt97}
{\sc J.~D. McNeal and E.~M. Stein}, {\em The {S}zeg{\H o} projection on convex
  domains}, Math. Z., 224 (1997), pp.~519--553.

\bibitem{Na86}
{\sc A.~Nagel}, {\em Vector fields and nonisotropic metrics}, in Beijing
  lectures in harmonic analysis ({B}eijing, 1984), vol.~112 of Ann. of Math.
  Stud., Princeton Univ. Press, Princeton, NJ, 1986, pp.~241--306.

\bibitem{NaRoStWa88}
{\sc A.~Nagel, J.-P. Rosay, E.~M. Stein, and S.~Wainger}, {\em Estimates for
  the {B}ergman and {S}zeg{\H o} kernels in certain weakly pseudoconvex
  domains}, Bull. Amer. Math. Soc. (N.S.), 18 (1988), pp.~55--59.

\bibitem{NaRoStWa89}
{\sc A.~Nagel, J.-P. Rosay, E.~M. Stein, and S.~Wainger}, {\em Estimates for
  the {B}ergman and {S}zeg{\H o} kernels in {${\bf C}\sp 2$}}, Ann. of Math.
  (2), 129 (1989), pp.~113--149.

\bibitem{Pe15}
{\sc A.~{Peterson}}, {\em {Estimates for the Szego Projection on Uniformly
  Finite-Type Subdomains of $\mathbb{C}^2$}}, ArXiv e-prints,  (2015).

\bibitem{PhSt77}
{\sc D.~H. Phong and E.~M. Stein}, {\em Estimates for the {B}ergman and
  {S}zeg\"o projections on strongly pseudo-convex domains}, Duke Math. J., 44
  (1977), pp.~695--704.

\bibitem{RaTi15}
{\sc A.~Raich and M.~Tinker}, {\em The {S}zeg\"o kernel on a class of
  non-compact {CR} manifolds of high codimension}, Complex Var. Elliptic Equ.,
  60 (2015), pp.~1366--1373.

\bibitem{Wi58}
{\sc N.~Wiener}, {\em The {F}ourier {I}ntegral and {C}ertain of its
  {A}pplications}, Dover Publications, Inc., New York, 1959.

\end{thebibliography}

\end{document}